\documentclass[12pt,reqno]{amsart}
\usepackage[margin=1in]{geometry}
\usepackage{amsmath,amssymb,amsthm,graphicx,amsxtra, setspace}
\usepackage[utf8]{inputenc}
\usepackage{mathrsfs}
\usepackage{hyperref}
\usepackage{upgreek}
\usepackage{mathtools}
\allowdisplaybreaks

\usepackage[pagewise]{lineno}

\newtheorem{theorem}{Theorem}[section]
\newtheorem{lemma}[theorem]{Lemma}

\newtheorem{remark}[theorem]{Remark}

\let\originalleft\left
\let\originalright\right
\renewcommand{\left}{\mathopen{}\mathclose\bgroup\originalleft}
\renewcommand{\right}{\aftergroup\egroup\originalright}

\renewcommand{\d}{\/\mathrm{d}\/}

\def\w{\textbf{W}^{\varepsilon}_{{\theta}^{\varepsilon}}}

\def\L{\mathbb{L}}

\def\C{\mathrm{C}}
\def\f{\boldsymbol{f}}

\def\v{\boldsymbol{v}}
\def\V{\mathbb{v}}
\def\w{\boldsymbol{w}}

\def\G{\boldsymbol{G}}

\def\V{\mathbb{V}}
\def\wi{\widetilde}

\def\u{\mathrm{U}}
\def\P{\mathrm{P}}
\def\u{\boldsymbol{u}}
\def\H{\mathbb{H}}
\def\n{\boldsymbol{n}}

\newcommand{\R}{\mathbb{R}}

\renewcommand{\d}{\/\mathrm{d}\/}

\newcommand{\Addresses}{{
		\footnote{
			
			\noindent \textsuperscript{1,2}Department of Mathematics, Indian Institute of Technology Roorkee-IIT Roorkee,
			Haridwar Highway, Roorkee, Uttarakhand 247667, INDIA.\par\nopagebreak
			\noindent 	\textit{e-mail:} \texttt{Pardeep Kumar: pkumar3@ma.iitr.ac.in.}
			
			\noindent  \textit{e-mail:} \texttt{Manil T. Mohan: manilfma@iitr.ac.in, maniltmohan@gmail.com.}
			
			\noindent \textsuperscript{*}Corresponding author.

			\textit{Key words:} Convective Brinkman-Forchheimer equations, inverse source problem, final overdetermination, Schauder's fixed point theorem.
			
			Mathematics Subject Classification (2010): Primary 35R30; Secondary 35Q35, 35Q30.

}}}

\begin{document}
	
	
	\title[An inverse problem for 2D and 3D convective Brinkman-Forchheimer equations]{Well-posedness of an inverse problem for two and three dimensional convective Brinkman-Forchheimer equations with the final overdetermination
		\Addresses}
	\author[P. Kumar and M. T. Mohan ]{Pardeep Kumar\textsuperscript{1} and Manil T. Mohan\textsuperscript{2*}}

	\maketitle
	
	\begin{abstract}
		In this article, we study an inverse problem for the following convective Brinkman-Forchheimer (CBF) equations:
		\begin{align*}
			\boldsymbol{u}_t-\mu \Delta\boldsymbol{u}+(\boldsymbol{u}\cdot\nabla)\boldsymbol{u}+\alpha\boldsymbol{u}+\beta|\boldsymbol{u}|^{r-1}\boldsymbol{u}+\nabla p=\boldsymbol{F}:=\boldsymbol{f} g, \ \ \  \nabla\cdot\boldsymbol{u}=0,
		\end{align*}
		in bounded domains $\Omega\subset\mathbb{R}^d$ ($d=2,3$) with smooth boundary, where $\alpha,\beta,\mu>0$ and $r\in[1,\infty)$. The CBF equations describe the motion of incompressible fluid flows in a saturated porous medium.	The inverse problem under our consideration consists of reconstructing the vector-valued velocity function $\boldsymbol{u}$, the 	pressure field $p$ and the vector-valued function $\boldsymbol{f}$. We have proved the well-posedness result (existence, uniqueness and stablility) for the inverse problem for the 2D and 3D CBF equations with the final overdetermination condition  using Schauder's fixed point theorem for arbitrary smooth initial data. The well-posedness results hold for $r\geq 1$ in two dimensions and for $r \geq 3$ in three dimensions. The global solvability results available in the literature helped us to obtain the uniqueness and stability results for the model with fast growing nonlinearities.  
	\end{abstract}

	\section{Introduction}\label{sec1}\setcounter{equation}{0}
	The main objective of this work is to discuss the well-posedness of an inverse problem to convective Brinkman-Forchheimer (CBF) equations. Physically, CBF equations describe the motion of incompressible fluid flows in a saturated porous medium.  The CBF equations in a bounded domain $\Omega\subset\R^d$ ($d=2,3$) with a smooth boundary $\partial\Omega$ (at least $\C^2$-boundary) are given by
	\begin{align}
		\u_t-\mu \Delta\u+(\u\cdot\nabla)\u+\alpha\u+\beta|\u|^{r-1}\u+\nabla p=\boldsymbol{F}&:=\f g, \ \text{ in } \ \Omega\times(0,T), \label{1a}\\ \nabla\cdot\u&=0, \ \ \ \ \text{ in } \ \Omega\times(0,T),\label{1b}
	\end{align}
	with initial condition
	\begin{align}\label{1c}
		\u=\u_0, \ \text{ in } \ \Omega \times \{0\},
	\end{align}
	and	boundary condition
	\begin{align}\label{1d}
		\u=\boldsymbol{0},\ \ \text{ on } \ \partial\Omega\times[0,T).
	\end{align}
	Here $\u(x,t) \in \R^d$ represents the velocity field at position $x$ and time $t$, $p(x,t)\in\R$ denotes the pressure field and $\boldsymbol{F}(x,t)\in\R^d$ stands for a divergence free (that is, $\nabla\cdot\boldsymbol{F}=0$) external force. The constant $\mu$ denotes the positive Brinkman coefficient (effective viscosity), the positive constants $\alpha$ and $\beta$ stands the Darcy coefficient (permeability of porous medium) and the Forchheimer coefficient (proportional to the porosity of the material), respectively. The absorption exponent $r\in[1,\infty)$ and the cases, $r=3$ and $r>3$, are known as the critical exponent and the fast growing nonlinearity, respectively.  For $\alpha=\beta=0$, we obtain the classical  Navier-Stokes equations (NSE). Thus, one can consider the equations \eqref{1a}-\eqref{1d} as a modification (by introducing an absorption term $\alpha\u+\beta|\u|^{r-1}\u$) of the classical NSE. Thus, one may refer the model \eqref{1a} as NSE with damping.  In order to obtain the uniqueness of the pressure $p$, one can impose the condition $\int_{\Omega}p(x,t)\d x=0, $ for $t\in (0,T)$.   The model given in \eqref{1a}-\eqref{1d} is recognized to be more accurate when the flow velocity is too large for the Darcy's law to be valid alone, and apart from that, the porosity is not too small, so that we call these types of models as \emph{non-Darcy models} (cf. \cite{PAM}).   	  It has been proved  in Proposition 1.1, \cite{KWH}  that the critical homogeneous CBF equations have the same scaling as NSE only when $\alpha=0$ and no scale invariance property for other values of $\alpha$ and $r$.

	Let us now discuss some global solvability results available in the literature for the system \eqref{1a}-\eqref{1d} (direct problem). The Navier-Stokes problem in bounded domains with compact boundary, modified by the absorption term $|\u|^{r-1}\u$, for $r>1$ is considered in \cite{SNA}. The existence of Leray-Hopf weak solutions  for any dimension $d\geq 2$ and its uniqueness for $d=2$ is established in \cite{SNA}.  The existence of regular dissipative solutions and global attractors for the system \eqref{1a}-\eqref{1d}  in three dimensions for the fast growing nonlinearites is proved in \cite{KT2}. As a global smooth solution exists  for the system \eqref{1a}-\eqref{1d} with $r>3,$  the energy equality is satisfied by the weak solutions. The authors in \cite{CLF} proved that all weak solutions of the 3D critical CBF equations  ($r=3$) in  bounded domain satisfy the energy equality (see \cite{KWH} for the case of periodic domains).  The author in \cite{MTM4} proved the existence and uniqueness of a global weak solution in the Leray-Hopf sense satisfying the energy equality to the system \eqref{1a}-\eqref{1d}	 (in 3D,  $r>3$ for all values of $\beta$ and $\mu $, and   $r=3$ for $2\beta\mu \geq 1$). The monotonicity as well as the demicontinuity properties of the linear and nonlinear operators and the Minty-Browder technique were exploited in the proofs. The stochastic counterpart of the problem was considered in \cite{MTM6}.  For the global solvability results  for 3D CBF equations and related models in the whole space as well as in periodic domains, the interested readers are referred to see \cite{ZCQJ,ZZXW,YZ,KWH,KWH1}, etc.

	Though the direct problem is important, it's study requires a significant amount of information of the physical parameters such as the Brinkman coefficient $\mu$, Darcy coefficient $\alpha$ and Forchheimer coefficient $\beta$ and the forcing term $\boldsymbol{F}:=\f g$. In such modeling, it is better to consider the inverse source problems but posing an inverse problem requires some additional information of the solution besides the given initial and boundary conditions. 	In this work, we will be using the trace of the velocity $\u$ and the pressure gradient $\nabla p$, prescribed at the final moment $t=T$ of the segment $[0,T]$, as additional information.  We assume that $\boldsymbol{F}$, the vector-valued external forcing appearing in \eqref{1a}, can be written in the form  $$\boldsymbol{F}(x,t):=\f(x)g(x,t),$$ where the vector-valued function $\f$ is unknown and $g$ is a given scalar function such that  $g$, $g_t$ are continuous on $\overline{\Omega} \times[0,T]$.	We pose the nonlinear inverse problem of finding the vector-valued velocity function $\u$, the pressure gradient $\nabla p$ and the vector-valued function $\f$, satisfying the system \eqref{1a}-\eqref{1d}, with the final overdetermination condition:
	\begin{align}\label{1e}
		\u(x,T)=\boldsymbol{\varphi}(x), \ \ \ \ \nabla p(x,T)=\nabla\psi(x), \ \ \ x \in \Omega,
	\end{align} 
	where the functions $\u_0, \boldsymbol{\varphi}, \nabla \psi,$ and $g$ are given.\

	Inverse problems with the final overdetermination have been well studied for parabolic equations (see \cite{IB,NLG,VI,VI1,PT}, etc and the references therein). The works \cite{IB,NLG, VI,VI1,POV,PT}, etc assumed that the initial data is smooth (at least in $\H^2(\Omega)$).  The authors  in \cite{VP} established the solvability results of an inverse problem to the nonlinear NSE with the final overdetermination data in  three dimensions using Schauder's fixed point theorem.	Based on the weak solution of the NSE,  the existence results of an inverse problem to the NSE  with the integral overdetermination conditions as well as with the final overdetermination data in both two and three dimension using Schauder's fixed point theorem in \cite{POV}. Here it should be noted that they have assumed that the given data is small and it satisfies the following conditions:
	\begin{align}\label{1f}
		|g(x,T)|\geq g_T >0 \  \text{ for some positive constant} \ g_T \ \text{for} \ x\in  \overline{\Omega},
	\end{align}
	\begin{align}\label{1g}
		\u_0, \ \boldsymbol{\varphi} \in \H^2(\Omega) \cap \V , \ \ \ \nabla \psi \in \G(\Omega),\
	\end{align}
	where the function spaces are defined in subsection \ref{sub2.1}. 	However, neither  uniqueness nor stability is considered in \cite{POV}. 	The well-posedness  of an inverse problem for 2D NSE  with the final overdetermination data using Tikhonov fixed point theorem is discussed  in \cite{JF}. The authors   have assumed that initial data $\u_0 \in \H$ and the viscosity constant is sufficiently large. In \cite{MC}, the authors considered an inverse problem of finding a spatially varying factor in a source term in the non-stationary linearized NSE  from the observation data in an arbitrarily fixed sub-domain over some time interval. They have proved the Lipschitz stability based on the new Carleman estimate for the linearized NSE, provided the $t$-dependent factor satisfies a non-degeneracy condition. For an extensive study on different inverse problems for Navier-Stokes equations and related models, the interested readers are referred to see \cite{MBFC,AYC,AYC1,JFMD,JFGN,OYMY,YJJF,AIK,PKKK,RYL}, etc and the references therein.
	
	By a solution of the inverse problem  \eqref{1a}-\eqref{1e}, we mean by a triple $(\u,\nabla p,\f)$ such that $$\u\in\mathrm{H}^1(0,T;\V)\cap\mathrm{L}^{\infty}(0,T;\H^2(\Omega)\cap\V),\ \nabla p(\cdot,t)\in\G(\Omega),\ \f\in\L^2(\Omega),$$ for any $t\in[0,T]$ and it continuously depends on $t$ in the $\L^2$-norm on the segment $[0,T]$, and in addition, all the relations \eqref{1a}-\eqref{1e} hold. In order to prove the existence of solutions for the above  formulated inverse problem, we use the method developed in \cite{POV}. Note that the work \cite{POV} does not discuss the uniqueness or stability of solutions.	The aim of this article is to prove 
	\begin{enumerate}
		\item [(i)] the existence of a solution and its uniqueness, 
		\item [(ii)] the stability of the solution in the norm of corresponding functions, 
	\end{enumerate}
	to the inverse problem \eqref{1a}-\eqref{1e} under the assumptions \eqref{1f}-\eqref{1g}, using Schauder's fixed point theorem for arbitrary smooth initial data. In contrast to the results obtained in \cite{JF,POV}, etc for NSE, for the CBF equations \eqref{1a}-\eqref{1d} with fast growing nonlinearities, the results  are true for $\mu$ satisfying \eqref{1.12} or \eqref{1.13}, which is independent of the data $\boldsymbol{\varphi}$. To the best of our knowledge, there are no results available in the literature on the inverse problem for CBF equations and this work appears to be the first one which discusses the well-posedness of an inverse problem for CBF equations or NSE with damping. 
	
	We point out here that the method applied in \cite{JF} (for the initial data $\u_0\in\H$) may be applicable for the case $d=2, \ r\in[1,3]$ only, due a technical difficulty in working with bounded domains. In  bounded domains,  $\mathrm{P}_{\H}(|\u|^{r-1}\u)$ ($\mathrm{P}_{\H}:\L^2(\Omega)\to\H$ is the Helmholtz-Hodge orthogonal projection, see subsection \ref{sub2.1}) need not be zero on the boundary, and $\mathrm{P}_{\H}$ and $-\Delta$ are not necessarily commuting (for a counter example, see Example 2.19, \cite{JCR4}). Moreover, $-\Delta\u\cdot\n\big|_{\partial\Omega}\neq 0$ in general and the term with pressure will not disappear (see \cite{KT2}), while taking inner product with $-\Delta\u$ in \eqref{1a}. Therefore, the equality (\cite{KWH})
	\begin{align}\label{3}
		&\int_{\Omega}(-\Delta\u(x))\cdot|\u(x)|^{r-1}\u(x)\d x\nonumber\\&=\int_{\Omega}|\nabla\u(x)|^2|\u(x)|^{r-1}\d x+\frac{r-1}{4}\int_{\Omega}|\u(x)|^{r-3}|\nabla|\u(x)|^2|^2\d x,
	\end{align}
	may not be useful in the context of bounded domains. The case of $\u_0\in\H$ will be addressed in a future work. 
	
	Let us now state the main results of this paper. Let $\mathcal{D}$ be a subset of $\L^2(\Omega)$ defined by
	\begin{align*}
		\mathcal{D}:=\left\{\f \in \L^2(\Omega) \ : \ \|\f\|_{\L^2}\leq 1\right\}.
	\end{align*}
	Next, we define the nonlinear operator $\mathcal{A}:\mathcal{D} \to \L^2(\Omega)$ by
	\begin{align}\label{1h}
		(\mathcal{A}\f)(x):=\u_t(x,T), \ \text{ for }\ x \in \Omega,
	\end{align}
	where $\u(x,T)$ has been found via the unique solution $\u(x,t)$ of the direct problem \eqref{1a}-\eqref{1d}.	Another nonlinear operator $\mathcal{B}:\mathcal{D} \to \L^2(\Omega)$ defined by
	\begin{align}\label{1i}
		(\mathcal{B}\f)(x):=\frac{1}{g(x,T)}(\mathcal{A}\f)(x), \ \text{ for }\  x \in \Omega,
	\end{align}
	complements careful analysis of the nonlinear operator equation of the second kind for $\f$:
	\begin{align}\label{1j}
		\f=\mathcal{B}\f:=\frac{1}{g(x,T)}\left(\mathcal{A}\f+(\boldsymbol{\varphi} \cdot \nabla)\boldsymbol{\varphi}+\nabla \psi-\mu \Delta \boldsymbol{\varphi}+\alpha \boldsymbol{\varphi}+\beta|\boldsymbol{\varphi}|^{r-1}\boldsymbol{\varphi}\right).
	\end{align}
	The following result verifies the relation between the solvability of the inverse problem \eqref{1a}-\eqref{1e} and the nonlinear operator equation of the second kind \eqref{1j}.
	\begin{theorem}\label{thm1}
		Let $\Omega \subset \mathbb{R}^d \ (d=2,3)$ be a bounded domain with smooth boundary $\partial\Omega$, $\u_0, \boldsymbol{\varphi} \in \H^2(\Omega) \cap \V , \  \nabla \psi \in \G(\Omega)$ and $g, g_t \in \C(\overline{\Omega} \times[0,T])$ satisfy the assumption \eqref{1f}, and let 
		\begin{align}
			\|\boldsymbol{\varphi}\|_{\widetilde{\L}^4}\left(\frac{2}{\lambda_1}\right)^\frac{1}{4} <\mu, \ \ \ \text{for} \ d=2 \ \text{and} \ r\geq1, \label{1.11}\\
			\frac{r-3}{\lambda_1\mu(r-1)}\left(\frac{2}{\beta\mu (r-1)}\right)^{\frac{2}{r-3}}< \mu, \ \ \ \text{for} \ d=3 \ \text{and} \ r>3, \label{1.12}\\
			\frac{1}{2\beta}< \mu ,  \ \ \ \text{for} \ d=3 \  \text{and} \ r=3,\label{1.13}
		\end{align}
		where $\lambda_1$ is the smallest  eigenvalue of the Stokes operator. If the operator equation \eqref{1j} has a solution lying within $\mathcal{D}$,
		then there exists a solution of the inverse problem \eqref{1a}-\eqref{1e}. Conversely, if the inverse problem \eqref{1a}-\eqref{1e} is solvable, then so is the operator equation \eqref{1j}.
	\end{theorem}
	\begin{remark}
		One can replace the condition in \eqref{1.12} by	$\mu>\frac{1}{2}\max\left\{\frac{1}{\beta},\frac{1}{\sqrt{\lambda_1}}\right\}$ also (see Remark \ref{rem2.1} below)
	\end{remark}
	We are now in a position to state our main result on the well-posedness of
	solutions of the inverse problem \eqref{1a}-\eqref{1e}.
	\begin{theorem}\label{thm2}
		Let the assumptions of Theorem \ref{thm1} hold true.	Then, the following assertions hold for the inverse problem \eqref{1a}-\eqref{1e}.
		\begin{enumerate}
			\item [(i)]  There exists a  solution $\{\u,\nabla p,\f\}$ to the inverse problem \eqref{1a}-\eqref{1e}.
			\item [(ii)]  Let $(\u_i,\nabla p_i,\f_i)$ $(i=1,2)$ be two solutions to the inverse problem \eqref{1a}-\eqref{1e} corresponding to the input data  $(\u_{0i},\boldsymbol{\varphi}_i,\nabla\psi_i,g_i) \ (i=1,2)$. Then there exists a constant $C$ such that
			\begin{align}\label{1k}
				\nonumber\|\u_1-\u_2\|_{\mathrm{L}^\infty(0,T;\H)}&+	\|\u_1-\u_2\|_{\mathrm{L}^2(0,T;\V)}+	\|\u_1-\u_2\|_{\mathrm{L}^{r+1}(0,T;\widetilde{\L}^{r+1})}+\|\f_1-\f_2\|_{\L^2} \nonumber\\&\leq C\big(\|\u_{01}-\u_{02}\|_{\H^2(\Omega) \cap \V}+\|g_1-g_2\|_0+\|(g_1-g_2)_t\|_0\nonumber\\&\quad+\| \nabla (\boldsymbol{\varphi}_1-\boldsymbol{\varphi}_2)\|_{\H}+\| \nabla (\psi_1-\psi_2)-\mu \Delta(\boldsymbol{\varphi}_1-\boldsymbol{\varphi}_2)\|_{\H}\big),
			\end{align}
			where $C$ depends on the input data, $\mu,\alpha,\beta,r$, $T$ and $\Omega$. The uniqueness of solutions follows from \eqref{1k}. 
		\end{enumerate}	
	\end{theorem}
	\noindent

	The plan of the paper is as follows: In the next section, we first provide the proof of the relation between solvability the inverse problem \eqref{1a}-\eqref{1e} and the equivalent nonlinear operator equation of second kind (Theorem \ref{thm1}). After that, we obtain a number of a priori estimates which are necessary to handle the inverse problem \eqref{1a}-\eqref{1e}. In the final section, we prove our main result (Theorem \ref{thm2}) by firstly showing the existence of a solution of the equivalent operator equation by using Schauder's fixed point theorem, and then establishing the uniqueness and stability of the solution to the inverse problem. 

	\section{A Priori Estimates}\label{sec4}\setcounter{equation}{0}
	In this section, first we provide a proof of Theorem \ref{thm1}, which transforms original inverse problem \eqref{1a}-\eqref{1e} into an equivalent nonlinear operator equation of second kind \eqref{1j}. Since   sufficiently regular solutions for the  system \eqref{1a}-\eqref{1e} are known, we obtain a number of a priori estimates	for the solutions. These estimates will be used in the next section, where we establish the existence, uniqueness and stability of the solution of our inverse problem (proof of Theorem \ref{thm2}). We start this section by introducing function spaces and standard notations, which will be used throughout the paper. 
	\subsection{Function spaces}\label{sub2.1} Let $\C_0^{\infty}(\Omega;\R^d)$ be the space of all infinitely differentiable functions  ($\R^d$-valued) with compact support in $\Omega\subset\R^d$.  Let us define 
	\begin{align*} 
		\mathcal{V}&:=\{\u\in\C_0^{\infty}(\Omega,\R^d):\nabla\cdot\u=0\},\\
		\mathbb{H}&:=\text{the closure of }\ \mathcal{V} \ \text{ in the Lebesgue space } \L^2(\Omega)=\mathrm{L}^2(\Omega;\R^d),\\
		\mathbb{V}&:=\text{the closure of }\ \mathcal{V} \ \text{ in the Sobolev space } \H_0^1(\Omega)=\mathrm{H}_0^1(\Omega;\R^d),\\
		\widetilde{\L}^{p}&:=\text{the closure of }\ \mathcal{V} \ \text{ in the Lebesgue space } \L^p(\Omega)=\mathrm{L}^p(\Omega;\R^d),
	\end{align*}
	for $p\in(2,\infty]$. Then, under some smoothness assumptions on the boundary, we characterize the spaces $\H$, $\V$, $\widetilde{\L}^p$ and $\widetilde{\L}^\infty$ as 
	$$
	\H=\{\u\in\L^2(\Omega):\nabla\cdot\u=0,\u\cdot\boldsymbol{n}\big|_{\partial\Omega}=0\},$$ where $\boldsymbol{n}$ is the unit outward drawn normal to $\partial\Omega$, and $\u\cdot\n\big|_{\partial\Omega}$ should be understood in the sense of trace in $\H^{-1/2}(\partial\Omega)$ (cf. Theorem 1.2, Chapter 1, \cite{Te}),   with norm  $\|\u\|_{\H}^2:=\int_{\Omega}|\u(x)|^2\d x,
	$
	$$
	\V=\{\u\in\H_0^1(\Omega):\nabla\cdot\u=0\},$$  with norm $ \|\u\|_{\V}^2:=\int_{\Omega}|\nabla\u(x)|^2\d x,
	$ $$\widetilde{\L}^p=\{\u\in\L^p(\Omega):\nabla\cdot\u=0, \u\cdot\boldsymbol{n}\big|_{\partial\Omega}=0\},$$ with norm $\|\u\|_{\widetilde{\L}^p}^p:=\int_{\Omega}|\u(x)|^p\d x$, and $\widetilde{\L}^\infty=\{\u\in\L^\infty(\Omega):\nabla\cdot\u=0, \u\cdot\boldsymbol{n}\big|_{\partial\Omega}=0\},$ with norm $\|\u\|_{\widetilde{\L}^\infty}:=\operatorname*{ess\,sup}\limits_{x \in \Omega}|\u(x)|$, respectively. Let $(\cdot,\cdot)$ denote the inner product in the Hilbert space $\H$. 
	Wherever needed, we assume that $p_0\in\mathrm{H}^1(\Omega)\cap\mathrm{L}^2_0(\Omega),$ where $\mathrm{L}^2_0(\Omega):=\left\{p\in\mathrm{L}^2(\Omega):\int_{\Omega}p(x)\d x=0\right\}$. 
	The norm in the space $\C(\overline\Omega\times [0,T])$ is denoted by $\|\cdot\|_0$, that is,  $\|g\|_0:=\sup\limits_{(x,t)\in\overline\Omega\times [0,T]}|g(x,t)|$.   
	
	\subsubsection{Projection operator}	It is well-known from \cite{DFHM,OAL} that every vector field $\u\in\mathbb{L}^p(\Omega)$, for $1<p<\infty$ can be uniquely represented as $\u=\v+\nabla q,$ where $\v\in\mathbb{L}^p(\Omega)$ with $\mathrm{div \ }\v=0$ in the sense of distributions in $\Omega$ with $\v\cdot\n=0$ on $\partial\Omega$ and $q\in\mathrm{W}^{1,p}(\Omega)$ (Helmholtz-Weyl or Helmholtz-Hodge decomposition). For smooth vector fields in $\Omega$, such a decomposition is an orthogonal sum in $\mathbb{L}^2(\Omega)$. Note that $\u=\v+\nabla q$ holds for all $\u\in\mathbb{L}^p(\Omega)$, so that we can define the projection operator $\mathcal{P}_p$  by $\mathcal{P}_p\u = \v$. Let us consider the set $\G_{p}(\Omega):=\left\{\nabla q:q\in \mathrm{W}^{1,p}({\Omega})\right\}$ equipped with the norm $\|\nabla q\|_{\L^p}$. Then, from the above discussion, we obtain $\L^p(\Omega)=\wi\L^p(\Omega)\oplus\G_{p}(\Omega)$.   	For $p=2$, we obtain $\L^2(\Omega)=\H \oplus \G(\Omega)$, where $\G(\Omega)$ is the orthogonal complement of $\H$ in $\L^2(\Omega)$. We use the notation $\P_{\H}:=\mathcal{P}_2$ for  the orthogonal projection operator from $\L^2(\Omega)$ into $\H$. Since $\partial\Omega$ is of $\C^2$-boundary, note that $\P_{\H}$ maps $\H^1(\Omega)$ into itself and is bounded (Remark 1.6, \cite{Te}). 
	
	\subsubsection{Important inequalities} In the sequel, $C$ denotes a generic constant which may take different values at different places. 
	The following Gagliardo-Nirenberg's and Agmon's inequalities are used repeatedly in the paper: 
	\begin{align}
		\|\u\|_{\L^p}&\leq C\|\nabla\u\|_{\L^2}^{d\left(\frac{1}{2}-\frac{1}{p}\right)}\|\u\|_{\L^2}^{1-d\left(\frac{1}{2}-\frac{1}{p}\right)}, \ \text{ for all }\ \u\in\H_0^1(\Omega),\\
		\|\nabla\u\|_{\L^p}&\leq C\|\u\|_{\H^2}^{\frac{1}{2}+\frac{d}{2}\left(\frac{1}{2}-\frac{1}{p}\right)}\|\u\|_{\L^2}^{\frac{1}{2}-\frac{d}{2}\left(\frac{1}{2}-\frac{1}{p}\right)}, \ \text{ for all }\ \u\in\H^2(\Omega),
	\end{align}
	where  $2\leq p<\infty$ for $d=2$ and $2\leq p\leq 6$ for $d=3$, and 
	\begin{align}
		\|\u\|_{\L^{\infty}}\leq \left\{\begin{array}{cc}C\|\u\|_{\L^2}^{1/2}\|\u\|_{\H^2}^{1/2},&\text{ for }d=2,\\
			C\|\u\|_{\H^1}^{1/2}\|\u\|_{\H^2}^{1/2},&\text{ for }d=3,\end{array}\right.
	\end{align}
	for all $\u\in\H^2(\Omega)$. The well-known Poincar\'e inequality ($\|\u\|_{\H}\leq\frac{1}{\sqrt{\lambda_1}}\|\nabla\u\|_{\H}$, for all $\u\in\V$, where $\lambda_1$ is the smallest eigenvalue of the Stokes operator) as well as Ladyzhenkaya inequality ($\|\u\|_{\L^4}\leq 2^{1/4}\|\u\|_{\L^2}\|\nabla\u\|_{\L^2}$, for all $\u\in\H_0^1(\Omega)$) will also be used.

	\subsection{Proof of Theorem \ref{thm1}}	We are now ready to prove Theorem \ref{thm1}.
	\begin{proof}[Proof of Theorem \ref{thm1}]
		We assume that the nonlinear operator equation \eqref{1j} has a solution in $\mathcal{D}$, say $\f$.
		Upon substituting $\f$ into \eqref{1a}, we make use of the system \eqref{1a}-\eqref{1d} to find a pair of functions $\{\u,\nabla p\}$ as the solution of the direct problem corresponding to the external forcing function $\boldsymbol{F}(x,t)=\f(x)g(x,t)$.
		
		In order to prove that $\u$ and $\nabla p$ satisfy the overdetermination  condition \eqref{1e}, we consider $$\u(x,T)=\boldsymbol{\varphi}_1(x) \ \ \ \ \text{and} \ \ \ \ \nabla p=\nabla\psi_1(x),  \ \ \ \ x \ \in \ \Omega,$$
		and
		\begin{align}\label{3a}
			\u_t(x,T)-\mu \Delta \boldsymbol{\varphi}_1+(\boldsymbol{\varphi}_1\cdot\nabla)\boldsymbol{\varphi}_1+\alpha \boldsymbol{\varphi}_1+\beta|\boldsymbol{\varphi}_1|^{r-1}\boldsymbol{\varphi}_1+\nabla \psi_1=\boldsymbol{f}g(x,T). 
		\end{align}
		On the other hand, \eqref{1j} implies that
		\begin{align}\label{3b}
			\mathcal{A}\f-\mu \Delta \boldsymbol{\varphi}+(\boldsymbol{\varphi}\cdot\nabla)\boldsymbol{\varphi}+\alpha \boldsymbol{\varphi}+\beta|\boldsymbol{\varphi}|^{r-1}\boldsymbol{\varphi}+\nabla \psi=\boldsymbol{f}g(x,T). 
		\end{align}
		Using \eqref{1h} in \eqref{3b}, and then from \eqref{3a} and \eqref{3b}, it is not difficult to verify that the functions $(\boldsymbol{\varphi}-\boldsymbol{\varphi}_1)$ and $\nabla(\psi-\psi_1)$ satisfy the system of equations
		\begin{equation}\label{3c}
			\left\{
			\begin{aligned}
				-\mu \Delta (\boldsymbol{\varphi}-\boldsymbol{\varphi}_1)+\alpha(\boldsymbol{\varphi}-\boldsymbol{\varphi}_1)+((\boldsymbol{\varphi}-\boldsymbol{\varphi}_1)\cdot\nabla)\boldsymbol{\varphi}+((\boldsymbol{\varphi}_1&\cdot\nabla)(\boldsymbol{\varphi}-\boldsymbol{\varphi}_1)\\+\beta(|\boldsymbol{\varphi}|^{r-1}\boldsymbol{\varphi}-|\boldsymbol{\varphi}_1|^{r-1}\boldsymbol{\varphi}_1)+\nabla (\psi-\psi_1)&=\boldsymbol{0}, \ \ \text{in} \ \Omega, \\ \nabla \cdot  (\boldsymbol{\varphi}-\boldsymbol{\varphi}_1)&=0, \ \  \text{in} \ \Omega, \\
				\boldsymbol{\varphi}-\boldsymbol{\varphi}_1&=\boldsymbol{0}, \ \  \text{on} \ \partial\Omega.
			\end{aligned}
			\right.
		\end{equation}
		Taking the inner product with $(\boldsymbol{\varphi}-\boldsymbol{\varphi}_1)$ to the first equation in \eqref{3c}, we find
		\begin{align}\label{3d}
			\mu\|\nabla (\boldsymbol{\varphi}-\boldsymbol{\varphi}_1)\|_{\H}^2+\alpha \| \boldsymbol{\varphi}-\boldsymbol{\varphi}_1\|_{\H}^2&=-(((\boldsymbol{\varphi}-\boldsymbol{\varphi}_1)\cdot\nabla)\boldsymbol{\varphi},\boldsymbol{\varphi}-\boldsymbol{\varphi}_1)\nonumber\\&\quad-\beta(|\boldsymbol{\varphi}|^{r-1}\boldsymbol{\varphi}-|\boldsymbol{\varphi}_1|^{r-1}\boldsymbol{\varphi}_1,\boldsymbol{\varphi}-\boldsymbol{\varphi}_1).
		\end{align}
		For $r\geq 1$, we have (see Sec. 2.4 \cite{MTM4,MTM6})
		\begin{align}
			\beta\left(|\boldsymbol{\varphi}|^{r-1}\boldsymbol{\varphi}-|\boldsymbol{\varphi}_1|^{r-1}\boldsymbol{\varphi}_1,\boldsymbol{\varphi}-\boldsymbol{\varphi}_1\right)&\geq \frac{\beta}{2}\||\boldsymbol{\varphi}|^\frac{r-1}{2}(\boldsymbol{\varphi}-\boldsymbol{\varphi}_1)\|_{\H}^2+\frac{\beta}{2}\||\boldsymbol{\varphi}_1|^\frac{r-1}{2}(\boldsymbol{\varphi}-\boldsymbol{\varphi}_1)\|_{\H}^2  \label{3e} \\&\geq\frac{\beta}{2}\||\boldsymbol{\varphi}|^\frac{r-1}{2}(\boldsymbol{\varphi}-\boldsymbol{\varphi}_1)\|_{\H}^2 \geq0.\label{3e1}
		\end{align}
		\vskip 0.2 cm
		\noindent\textbf{Case I:} \emph{$d=2$ and $r\in[1,3]$.} 
		Using H\"older's, Ladyzhenskaya's and Poincar\'e's inequalities, we get
		\begin{align}\label{3f}
			\nonumber	&\big|\big(((\boldsymbol{\varphi}-\boldsymbol{\varphi}_1)\cdot\nabla)\boldsymbol{\varphi},\boldsymbol{\varphi}-\boldsymbol{\varphi}_1\big)\big| =\big|\big(((\boldsymbol{\varphi}-\boldsymbol{\varphi}_1)\cdot\nabla)(\boldsymbol{\varphi}-\boldsymbol{\varphi}_1),\boldsymbol{\varphi}\big)\big|\\&\quad \leq\|\nabla (\boldsymbol{\varphi}-\boldsymbol{\varphi}_1)\|_{\H}\| \boldsymbol{\varphi}-\boldsymbol{\varphi}_1\|_{\widetilde{\L}^4}\| \boldsymbol{\varphi}\|_{\widetilde{\L}^4} \leq \left(\frac{2}{\lambda_1}\right)^\frac{1}{4}\| \boldsymbol{\varphi}\|_{\widetilde{\L}^4}\|\nabla (\boldsymbol{\varphi}-\boldsymbol{\varphi}_1)\|_{\H}^2,
		\end{align}
		where $\lambda_1$ is the first eigenvalue of the Stokes operator (cf. \cite{MTM4,MTM6}). 	Substituting \eqref{3e1} and \eqref{3f} in \eqref{3d}, we obtain
		\begin{align*}
			\left(\mu-\left(\frac{2}{\lambda_1}\right)^\frac{1}{4}\| \boldsymbol{\varphi}\|_{\widetilde{\L}^4}\right)\|\nabla (\boldsymbol{\varphi}-\boldsymbol{\varphi}_1)\|_{\H}^2\leq0.
		\end{align*}
		If $\| \boldsymbol{\varphi}\|_{\widetilde{\L}^4}\left(\frac{2}{\lambda_1}\right)^\frac{1}{4}< \mu$, then $\|\nabla (\boldsymbol{\varphi}-\boldsymbol{\varphi}_1)\|_{\H}=0$, so that  $\boldsymbol{\varphi}=\boldsymbol{\varphi}_1$ and $\psi=\psi_1$.
		\vskip 0.2 cm
		\noindent\textbf{Case II:} 	\emph{$d=2,3$ and $r>3$.} 
		Applying the Cauchy-Schwarz and Young's inequalities, we find 
		\begin{align}\label{3h}
			\big|\big(((\boldsymbol{\varphi}-\boldsymbol{\varphi}_1)\cdot\nabla)(\boldsymbol{\varphi}-\boldsymbol{\varphi}_1),\boldsymbol{\varphi}\big)\big| &\leq\|\nabla (\boldsymbol{\varphi}-\boldsymbol{\varphi}_1)\|_{\H}\|\boldsymbol{\varphi} (\boldsymbol{\varphi}-\boldsymbol{\varphi}_1)\|_{\H}\nonumber\\&\leq \frac{\mu}{2}\|\nabla (\boldsymbol{\varphi}-\boldsymbol{\varphi}_1)\|_{\H}^2+\frac{1}{2\mu}\|\boldsymbol{\varphi} (\boldsymbol{\varphi}-\boldsymbol{\varphi}_1)\|_{\H}^2.
		\end{align}
		We take the term $\|\boldsymbol{\varphi} (\boldsymbol{\varphi}-\boldsymbol{\varphi}_1)\|_{\H}^2$ from \eqref{3h} and use the H\"older's and Young's inequalities to estimate it as (see \cite{KWH})
		\begin{align}\label{3i}
			&\int_\Omega|\boldsymbol{\varphi}(x)|^2|\boldsymbol{\varphi}(x)-\boldsymbol{\varphi}_1(x)|^2\d x\nonumber\\&=\int_{\Omega}|\boldsymbol{\varphi}(x)|^2|\boldsymbol{\varphi}(x)-\boldsymbol{\varphi}_1(x)|^{\frac{4}{r-1}}|\boldsymbol{\varphi}(x)-\boldsymbol{\varphi}_1(x)|^{\frac{2(r-3)}{r-1}}\d x\nonumber\\&\leq\left(\int_{\Omega}|\boldsymbol{\varphi}(x)|^{r-1}|\boldsymbol{\varphi}(x)-\boldsymbol{\varphi}_1(x)|^2\d x\right)^{\frac{2}{r-1}}\left(\int_{\Omega}|\boldsymbol{\varphi}(x)-\boldsymbol{\varphi}_1(x)|^2\d x\right)^{\frac{r-3}{r-1}}\nonumber\\&\leq{\beta\mu }\left(\int_{\Omega}|\boldsymbol{\varphi}(x)|^{r-1}|\boldsymbol{\varphi}(x)-\boldsymbol{\varphi}_1(x)|^2\d x\right)+\frac{r-3}{r-1}\left(\frac{2}{\beta\mu (r-1)}\right)^{\frac{2}{r-3}}\left(\int_{\Omega}|\boldsymbol{\varphi}(x)-\boldsymbol{\varphi}_1(x)|^2\d x\right),
		\end{align}
		for $r>3$. Using \eqref{3i} in \eqref{3h}, we find
		\begin{align}\label{3j}
			\nonumber	&\big|\big(((\boldsymbol{\varphi}-\boldsymbol{\varphi}_1)\cdot\nabla)(\boldsymbol{\varphi}-\boldsymbol{\varphi}_1),\boldsymbol{\varphi}\big)\big|\\&
			\leq \frac{\mu}{2}\|\nabla (\boldsymbol{\varphi}-\boldsymbol{\varphi}_1)\|_{\H}^2+\frac{\beta}{2}\||\boldsymbol{\varphi}|^\frac{r-1}{2}(\boldsymbol{\varphi}-\boldsymbol{\varphi}_1)\|_{\H}^2+\frac{r-3}{2 \mu(r-1)}\left(\frac{2}{\beta\mu (r-1)}\right)^{\frac{2}{r-3}}\|\boldsymbol{\varphi}-\boldsymbol{\varphi}_1\|_{\H}^2.
		\end{align}
		Combining \eqref{3e1} and \eqref{3j}, and then substituting it in \eqref{3d}, we deduce that
		\begin{align*}
			&\frac{\mu }{2}\|\nabla(\boldsymbol{\varphi}-\boldsymbol{\varphi}_1)\|_{\H}^2\leq\frac{r-3}{2\mu(r-1)}\left(\frac{2}{\beta\mu (r-1)}\right)^{\frac{2}{r-3}}\|\boldsymbol{\varphi}-\boldsymbol{\varphi}_1\|_{\H}^2.
		\end{align*}
		By Poincar\'e's inequality, we have
		\begin{align*}
			\left(\mu-\frac{r-3}{\lambda_1\mu(r-1)}\left(\frac{2}{\beta\mu (r-1)}\right)^{\frac{2}{r-3}}\right)\|\nabla(\boldsymbol{\varphi}-\boldsymbol{\varphi}_1)\|_{\H}^2\leq 0. 
		\end{align*}
		If $\frac{r-3}{\lambda_1\mu(r-1)}\left(\frac{2}{\beta\mu (r-1)}\right)^{\frac{2}{r-3}}<\mu$, then $\|\nabla(\boldsymbol{\varphi}-\boldsymbol{\varphi}_1)\|_{\H}=0$, so that 
		$\boldsymbol{\varphi}=\boldsymbol{\varphi}_1$ and $\psi=\psi_1$.
		\vskip 0.3cm
		\noindent
		\vskip 0.2 cm
		\noindent\textbf{Case III:} \emph{$d=r=3$.} From \eqref{3e1}, we get
		\begin{align}\label{3k}
			\beta(|\boldsymbol{\varphi}|\boldsymbol{\varphi}-|\boldsymbol{\varphi}_1|\boldsymbol{\varphi}_1,\boldsymbol{\varphi}-\boldsymbol{\varphi}_1)\geq \frac{\beta}{2}\||\boldsymbol{\varphi}|(\boldsymbol{\varphi}-\boldsymbol{\varphi}_1)\|_{\H}^2,
		\end{align}
		and \begin{align}\label{3l}
			\nonumber\big|\big(((\boldsymbol{\varphi}-\boldsymbol{\varphi}_1)\cdot\nabla)(\boldsymbol{\varphi}-\boldsymbol{\varphi}_1),\boldsymbol{\varphi}\big)\big| &\leq\|\nabla (\boldsymbol{\varphi}-\boldsymbol{\varphi}_1)\|_{\H}\|\boldsymbol{\varphi} (\boldsymbol{\varphi}-\boldsymbol{\varphi}_1)\|_{\H}\\&\leq\frac{1}{2\beta}\|\nabla (\boldsymbol{\varphi}-\boldsymbol{\varphi}_1)\|_{\H}^2+\frac{\beta}{2}\|\boldsymbol{\varphi} (\boldsymbol{\varphi}-\boldsymbol{\varphi}_1)\|_{\H}^2.
		\end{align}
		Combining \eqref{3k} and \eqref{3l}, and then substituting it in \eqref{3d}, we obtain
		\begin{align}
			\left(\mu-\frac{1}{2\beta}\right)\|\nabla (\boldsymbol{\varphi}-\boldsymbol{\varphi}_1)\|_{\H}^2 \leq 0.
		\end{align}
		For $\frac{1}{2\beta}<\mu $, we get  $\|\nabla(\boldsymbol{\varphi}-\boldsymbol{\varphi}_1)\|_{\H}=0$, and hence 
		$\boldsymbol{\varphi}=\boldsymbol{\varphi}_1$ and $\psi=\psi_1$.\
		\vskip 0.3cm
		Conversely, let us assume that the inverse problem \eqref{1a}-\eqref{1e} has a solution, say $\{\u,\nabla p,\f\}$. By using the elliptic regularity of the solution (Cattabriga regularity theorem, see \cite{LCa,Te1}, etc or  Corollary 2, \cite{KT2}), when the system \eqref{1a} is considered at $t=T$, we find
		\begin{align}
			\nonumber &\u_t(x,T)-\mu \Delta\u(x,T)+(\u(x,T)\cdot\nabla)\u(x,T)+\alpha \u(x,T)\\&\quad+\beta|\u(x,T)|^{r-1}\u(x,T)+\nabla p(x,T)=\boldsymbol{f}(x)g(x,T).
		\end{align}
		By the final overdetermination data \eqref{1e} and the definition of operator $\mathcal{A}$, we have
		\begin{align*}
			\nonumber\mathcal{A}\f&-\mu \Delta \boldsymbol{\varphi}+(\boldsymbol{\varphi}\cdot\nabla)\boldsymbol{\varphi}+\alpha \boldsymbol{\varphi}+ \beta|\boldsymbol{\varphi}|^{r-1}\boldsymbol{\varphi}+\nabla \psi=\boldsymbol{f}g(x,T),
		\end{align*}
		so that 
		\begin{align*}
			\f&=\frac{1}{g(x,T)}\left(\mathcal{A}\f-\mu \Delta \boldsymbol{\varphi}+(\boldsymbol{\varphi}\cdot\nabla)\boldsymbol{\varphi}+\alpha \boldsymbol{\varphi}+\beta|\boldsymbol{\varphi}|^{r-1}\boldsymbol{\varphi}+\nabla \psi\right)=\mathcal{B}\f.
		\end{align*}
		Hence, the nonlinear operator equation \eqref{1j} has a solution, which completes the proof. 
	\end{proof}
	\begin{remark}\label{rem2.1}
		One can perform the calculations in \eqref{3h} and \eqref{3i} in the following way also:
		\begin{align}\label{2.26}
			&\big|\big(((\boldsymbol{\varphi}-\boldsymbol{\varphi}_1)\cdot\nabla)(\boldsymbol{\varphi}-\boldsymbol{\varphi}_1),\boldsymbol{\varphi}\big)\big| \nonumber\\&\leq \theta\mu \|\nabla(\boldsymbol{\varphi}-\boldsymbol{\varphi}_1)\|_{\H}^2+\frac{1}{4\theta\mu }\int_{\Omega}|\boldsymbol{\varphi}(x)|^2|\boldsymbol{\varphi}(x)-\boldsymbol{\varphi}_1(x)|^2\d x\nonumber\\&= \theta\mu \|\nabla(\boldsymbol{\varphi}-\boldsymbol{\varphi}_1)\|_{\H}^2+\frac{1}{4\theta\mu }\int_{\Omega}|\boldsymbol{\varphi}(x)-\boldsymbol{\varphi}_1(x)|^2\left(|\boldsymbol{\varphi}(x)|^{r-1}+1\right)\frac{|\boldsymbol{\varphi}(x)|^2}{|\boldsymbol{\varphi}(x)|^{r-1}+1}\d x\nonumber\\&\leq \theta\mu \|\nabla(\boldsymbol{\varphi}-\boldsymbol{\varphi}_1)\|_{\H}^2+\frac{1}{4\theta\mu }\int_{\Omega}|\boldsymbol{\varphi}(x)|^{r-1}|\boldsymbol{\varphi}(x)-\boldsymbol{\varphi}_1(x)|^2\d x+\frac{1}{4\theta\mu }\int_{\Omega}|\boldsymbol{\varphi}(x)-\boldsymbol{\varphi}_1(x)|^2\d x,
		\end{align}
		for $0<\theta<1$,	where we have used the fact that $\left\|\frac{|\boldsymbol{\varphi}|^2}{|\boldsymbol{\varphi}|^{r-1}+1}\right\|_{\widetilde{\L}^{\infty}}<1$, for $r\geq 3$. Combining \eqref{3d}, \eqref{3e1} and \eqref{2.26}, we obtain 
		\begin{align}
			\left[\mu(1-\theta)-\frac{1}{4\lambda_1\theta\mu }\right]\|\nabla(\boldsymbol{\varphi}-\boldsymbol{\varphi}_1)\|_{\H}^2+\left(\frac{\beta}{2}-\frac{1}{4\theta\mu }\right)\||\boldsymbol{\varphi}|^\frac{r-1}{2}(\boldsymbol{\varphi}-\boldsymbol{\varphi}_1)\|_{\H}^2 \leq 0 .
		\end{align}
		Thus for $\mu>\frac{1}{2}\max\left\{\frac{1}{\beta},\frac{1}{\sqrt{\lambda_1}}\right\}$, one can obtain that $\boldsymbol{\varphi}=\boldsymbol{\varphi}_1$ and $\psi=\psi_1$. 
	\end{remark}
	\subsection{Energy estimates}\label{sec2.3} First we will concentrate on the direct problem \eqref{1a}-\eqref{1d}.
	To obtain energy estimates of the solutions of CBF equations \eqref{1a}-\eqref{1d}, we assume that $\u_0 \in \H^2(\Omega) \cap \V $, $g, g_t \in \C(\overline\Omega\times [0,T])$ satisfy the assumption \eqref{1f} and $\f \in \L^2(\Omega)$.\
	
	The next lemma provides the usual energy estimate for the CBF equations.
	\begin{lemma}\label{lemma1}
		Let $\u(\cdot)$ be the strong solution of the CBF  equations \eqref{1a}-\eqref{1d}. Then, for $r \in [1,\infty)$, we have 
		\begin{align}\label{E1}
			\sup_{t\in[0,T]}\|\u(t)\|_{\H}^2+2\mu \int_0^T\|\u(t)\|_{\V}^2\d t+2\beta\int_0^T\|\u(t)\|_{\widetilde{\L}^{r+1}}^{r+1}\d t\leq \|\u_0\|_{\H}^2+\frac{T}{\alpha}\|\f\|_{\L^2}^2\|g\|_0^2.
		\end{align}
	\end{lemma} 
	\begin{proof}
		Let us take inner product with $\u(\cdot)$ of the equation \eqref{1a} and use the fact that $((\u \cdot \nabla)\u,\u)=0$ to obtain 
		\begin{align}\label{E2}
			\frac{1}{2}\frac{\d}{\d t}\|\u(t)\|_{\H}^2+\mu \|\u(t)\|_{\V}^2+\alpha \|\u(t)\|_{\H}^2+\beta\|\u(t)\|_{\widetilde{\L}^{r+1}}^{r+1}=(\f g(t),\u(t)),
		\end{align}
		for a.e. $t\in[0,T]$, where we have performed the integration by parts. Using H\"older's and Young's inequalities, we estimate $|(\f g,\u)|$ as
		\begin{align}\label{E3}
			|(\f g,\u)| \leq\|\f\|_{\L^2}\|g\|_{\mathrm{L}^{\infty}}\|\u\|_{\H} \leq\frac{1}{2\alpha}\|\f\|_{\L^2}^2\|g\|_{\mathrm{L}^{\infty}}^{2}+\frac{\alpha}{2}\|\u\|_{\H}^{2}.
		\end{align}
		Substituting \eqref{E3} in \eqref{E2}, and then integrating from $0$ to $t$, we find 
		\begin{align*}
			&\|\u(t)\|_{\H}^2+2\mu \int_0^t\|\u(s)\|_{\V}^2\d s+\alpha\int_0^t\|\u(s)\|_{\H}^2\d s +2\beta\int_0^t\|\u(s)\|_{\widetilde{\L}^{r+1}}^{r+1}\d s\\&\quad\leq \|\u_0\|_{\H}^2+\frac{T}{\alpha}\|\f\|_{\L^2}^2\|g\|_0^{2},
		\end{align*}
		for all $t \in [0,T]$ and \eqref{E1} follows. 
	\end{proof}
	Let us now differentiate \eqref{1a}-\eqref{1d} with respect to time $t$ and define $\v(\cdot):=\u_t(\cdot)$. Then $\v(\cdot)$ satisfies: 
	\begin{equation}\label{E5}
		\left\{
		\begin{aligned}
			\v_t-\mu \Delta\v+(\v\cdot \nabla)\u+(\u \cdot\nabla)\v+\alpha \v+\beta r|\u|^{r-1}\v+\nabla p_t=\f g_t, \ \ \ \text{in} \ \Omega \times (0,T), \\
			\nabla \cdot \v=0, \ \ \ \ \text{in} \ \Omega \times (0,T), \ \\
			\v=\boldsymbol{0}, \ \ \  \text{on} \ \partial\Omega \times [0,T),\\
			\v=\P_{\H} \left(\mu \Delta\u_0-(\u_0 \cdot \nabla)\u_0-\alpha \u_0-\beta|\u_0|^{r-1}\u_0+\f(x)g(x,0)\right), \ \ \  \text{in} \ \Omega \times \{0\}, 
		\end{aligned}
		\right.
	\end{equation}
	where $\P_{\H}$ is an orthogonal projection on $\H$. Moreover, for $d=2,3$, using the embedding  $\H^2(\Omega) \cap \V \subset\H^2(\Omega)\subset \L^p(\Omega),$ for all $p \in [1, \infty)$, and Agmon's inequality, we see that
	\begin{align}\label{E6}
		\nonumber	\|\v(0)\|_{\H} &\leq \mu\|\Delta\u_0\|_{\H}+\|\u_0\|_{\wi\L^{\infty}}\|\nabla \u_0\|_{\H}+\alpha \|\u_0\|_{\H}+\beta\|\u_0\|_{\widetilde{\L}^{2r}}^r+\|\f\|_{\L^2}\|g\|_0\nonumber\\&\leq C(\mu,\alpha,\beta,\|g\|_0,\|\f\|_{\L^2},\|\Delta\u_0\|_{\H}) < +\infty,
	\end{align}
	whenever $\u_0 \in \H^2(\Omega) \cap \V$. 
	Therefore, the $\L^2$-norm of the initial data for $\v$ is  controllable if $\u_0 \in \H^2(\Omega) \cap \V$.
	
	The next lemma gives an estimate of $\v(t)$ for all $t \geq 0$.
	\begin{lemma}\label{lemma2}
		Let $\u(\cdot)$ be the strong solution of the CBF  equations \eqref{1a}-\eqref{1d}. Then, for $r \geq 1, \ (d=2), \text{ and for} \ r \geq 3, \ (d=3)$, we have 
		\begin{align}\label{E7}
			\sup_{t\in[0,T]}\|\v(t)\|_{\H}^2+\mu \int_0^T\|\v(t)\|_{\V}^2\d t\leq C\left(\|\u_0\|_{\H^2(\Omega) \cap \V}+\|\f\|_{\L^2}^2\|g_t\|_0^2\right).
		\end{align}
	\end{lemma} 
	\begin{proof}
		Taking the inner product with $\v(\cdot)$ to the first equation in \eqref{E5}, we find 
		\begin{align}\label{E8}
			\nonumber&\frac{1}{2}\frac{\d}{\d t}\|\v(t)\|_{\H}^2+\mu \|\v(t)\|_{\V}^2+\alpha \|\v(t)\|_{\H}^2+\beta r\||\u(t)|^\frac{r-1}{2}\v(t)\|_{\H}^2\\&\quad=(\f g_t(t),\v(t))-((\v(t) \cdot\nabla)\u(t),\v(t)).
		\end{align}
		Using H\"older's and Young's inequalities, we estimate  $|(\f g_t,\v)|$ as 
		\begin{align}\label{E9}
			|(\f g_t,\v)|\leq\|\f\|_{\L^2}\|g_t\|_{\mathrm{L}^{\infty}}\|\v\|_{\H}\leq\frac{1}{2\alpha}\|\f\|_{\L^2}^2\|g_t\|_{\mathrm{L}^{\infty}}^2+\frac{\alpha }{2}\|\v\|_{\H}^2.
		\end{align}
		\vskip 0.2 cm
		\noindent\textbf{Case I:} \emph{$d=2$ and $r \geq 1$.} Using H\"older's, Ladyzhenskaya's and Young's inequalities, we estimate $|((\v \cdot\nabla)\u,\v)|$ as
		\begin{align}\label{E10}
			|((\v \cdot\nabla)\u,\v)|\leq \|\v\|_{\widetilde{\L}^{4}}^2\|\nabla\u\|_{\H}\leq\sqrt{2}\|\v\|_{\H}\|\nabla\v\|_{\H}\|\nabla\u\|_{\H}\leq \frac{\mu}{2}\|\v\|_{\V}^2+\frac{1}{\mu}\|\v\|_{\H}^2\|\u\|_{\V}^2.
		\end{align}
		Substituting \eqref{E9} and \eqref{E10} in \eqref{E8}, and then integrating from $0$ to $t$, we find  
		\begin{align}\label{E11}
			&	\|\v(t)\|_{\H}^2+\mu \int_0^t\|\v(s)\|_{\V}^2\d s+\alpha \int_0^t\|\v(s)\|_{\H}^2\d s +2\beta r\int_0^t\||\u(s)|^\frac{r-1}{2}\v(s)\|_{\H}^2\d s\nonumber\\&\quad\leq\|\v(0)\|_{\H}^2+ \frac{T}{\alpha }\|\f\|_{\L^2}^2\|g_t\|_0^2 +\frac{2}{\mu}\int_0^t\|\u(s)\|_{\V}^2\|\v(s)\|_{\H}^2\d s.
		\end{align}
		An application of Gronwall's inequality in \eqref{E11} yields 
		\begin{align}\label{E12}
			\|\v(t)\|_{\H}^2\leq\left\{\|\v(0)\|_{\H}^2+ \frac{T}{\alpha }\|\f\|_{\L^2}^2\|g_t\|_0^2 \right\}\exp\left\{\frac{2}{\mu}\int_0^T\|\u(t)\|_{\V}^2\d t\right\},
		\end{align}
		for all $t\in[0,T]$. Thus, from \eqref{E11}, it is immediate that 
		\begin{align*}
			\nonumber&\sup_{t\in[0,T]}\|\v(t)\|_{\H}^2+\mu \int_0^T\|\v(t)\|_{\V}^2\d t+\alpha \int_0^T\|\v(t)\|_{\H}^2\d t+2\beta r\int_0^T\||\u(t)|^\frac{r-1}{2}\v(t)\|_{\H}^2\d t\\&\qquad  \leq C\left\{\|\u_0\|_{\H^2(\Omega) \cap \V}+ \frac{T}{\alpha}\|\f\|_{\L^2}^2\|g_t\|_0^2 \right\},
		\end{align*}
		and \eqref{E7} follows.
		\vskip 0.2 cm
		\noindent\textbf{Case II:} \emph{$d=3$ and $r >3$.}
		A calculation similar to \eqref{3j} yields
		\begin{align}\label{E14}
			|((\v \cdot \nabla)\u,\v)|\leq\frac{\mu }{2}\|\v\|_{\V}^2+\frac{\beta r}{2}\||\u|^{\frac{r-1}{2}}|\v|\|_{\H}^2+\frac{r-3}{2\mu(r-1)}\left(\frac{2}{\beta r \mu (r-1)}\right)^{\frac{2}{r-3}}\|\v\|_{\H}^2,
		\end{align}
		for $r>3$. Substituting \eqref{E9} and \eqref{E14} in \eqref{E8}, and then integrating from $0$ to $t$, we find
		\begin{align}\label{E15}
			&	\|\v(t)\|_{\H}^2+\mu \int_0^t\|\v(s)\|_{\V}^2\d s +\alpha \int_0^t\|\v(s)\|_{\H}^2\d s+\beta r\int_0^t\||\u(s)|^\frac{r-1}{2}\v(s)\|_{\H}^2\d s\nonumber\\&\quad\leq\|\v(0)\|_{\H}^2+\frac{T}{\alpha }\|\f\|_{\L^2}^2\|g_t\|_0^2+\frac{r-3}{\mu(r-1)}\left(\frac{2}{\beta r \mu (r-1)}\right)^{\frac{2}{r-3}}\int_0^t\|\v(s)\|_{\H}^2\d s,
		\end{align}
		for all $t \in [0,T]$. Applying Gronwall's inequality in \eqref{E15}, and then taking supremum on both sides over time from $0$ to $T$ in \eqref{E15}, one can easily get \eqref{E7}.

		\vskip 0.2 cm
		\noindent\textbf{Case III:} \emph{$d=r=3$.} The  term $\beta r\||\u(t)|^\frac{r-1}{2}\v(t)\|_{\H}^2$ in \eqref{E8} becomes $3\beta \||\u(t)|\v(t)\|_{\H}^2$ and
		\begin{align}\label{E17}
			|((\v \cdot \nabla)\u,\v)|\leq \|\v\|_{\V}\||\u||\v|\|_{\H}\leq\frac{\mu }{2}\|\v\|_{\V}^2+\frac{1}{2\mu}\||\u||\v|\|_{\H}^2.
		\end{align}
		Substituting \eqref{E9} and \eqref{E17} in \eqref{E8}, and then integrating from $0$ to $t$, we arrive at 
		\begin{align*}
			&	\|\v(t)\|_{\H}^2+\mu \int_0^t\|\v(s)\|_{\V}^2\d s+\alpha \int_0^t\|\v(s)\|_{\H}^2\d s+2\left(3\beta-\frac{1}{2\mu }\right)\int_0^t\||\u(s)||\v(s)|\|_{\H}^2\d s\nonumber\\&\qquad\leq  \|\v(0)\|_{\H}^2+\frac{T}{\alpha }\|\f\|_{\L^2}^2\|g_t\|_0^2,
		\end{align*}
		for all $t\in[0,T]$ and \eqref{E7} follows provided $6\beta\mu \geq 1$, which completes the proof.
	\end{proof}

	\begin{lemma}\label{lemma3}
		Let $\u(\cdot)$ be the strong solution of the CBF  equations \eqref{1a}-\eqref{1d}.  Then, for $r \geq 1, \ (d=2)\text{ and for} \ r \geq 3, \ (d=3)$, we have 
		\begin{align}\label{E18}
			&\mu	\sup_{t\in[0,T]}\|\u(t)\|_{\V}^2+ \frac{2\beta}{r+1}\sup_{t\in[0,T]}\|\u(t)\|_{\widetilde{\L}^{r+1}}^{r+1} +\int_0^T\|\u_t(t)\|_{\H}^2\d t \nonumber\\&\quad\leq C\left( \|\u_0\|_{\V}^2+\|\f\|_{\L^2}^2\|g\|_0^2\right).
		\end{align}
	\end{lemma}
	\begin{proof} 
		Taking the inner product with $\u_t(\cdot)$ in \eqref{1a}, we obtain 
		\begin{align}\label{E19}
			\nonumber&\|\u_t(t)\|_{\H}^2+\frac{\mu }{2}\frac{\d}{\d t}\|\nabla\u(t)\|_{\H}^2+\frac{\alpha }{2}\frac{\d}{\d t}\|\u(t)\|_{\H}^2+\beta(|\u(t)|^{r-1}\u(t),\u_t(t))\\&\quad=(\f g(t),\u_t(t))-((\u(t) \cdot\nabla)\u(t),\u_t(t)).
		\end{align}
		It can be easily seen that
		\begin{align}\label{E20}
			(|\u(t)|^{r-1}\u(t),\u_t(t))=\frac{1}{r+1}\frac{\d}{\d t}\|\u(t)\|_{\widetilde{\L}^{r+1}}^{r+1}, 
		\end{align}
		for $r\geq 1$. Using H\"older's and Young's inequalities, we estimate the terms $|(\f g,\u_t)|$ and $|((\u \cdot\nabla)\u,\u_t)|$ as
		\begin{align}
			|(\f g,\u_t)|&\leq\|g\|_{\mathrm{L}^{\infty}}\|\f\|_{\L^2}\|\u_t\|_{\H}\leq\frac{1}{2}\|\u_t\|_{\H}^2+\frac{1}{2}\|g\|_{\mathrm{L}^{\infty}}^2\|\f\|_{\L^2}^2,\label{E21} \\
			|((\u \cdot\nabla)\u,\u_t)|&\leq\|\u\|_{\widetilde{\L}^4}^2\|\u_t\|_{\V}\leq\frac{1}{2}\|\u_t\|_{\V}^2+\frac{1}{2}\|\u\|_{\widetilde{\L}^4}^4 \label{E22}.
		\end{align}
		Substituting the estimates \eqref{E20}-\eqref{E22} in \eqref{E19}, and then integrating from $0$ to $t$, we find   
		\begin{align}\label{E23}
			&\frac{2\beta}{r+1}\|\u(t)\|_{\widetilde{\L}^{r+1}}^{r+1}+\mu\|\u(t)\|_{\V}^2+\alpha\|\u(t)\|_{\H}^2+\int_0^t\|\u_t(s)\|_{\H}^2\d s\nonumber\\&\quad\leq \frac{2\beta}{r+1}\|\u_0\|_{\widetilde{\L}^{r+1}}^{r+1}+\mu \|\u_0\|_{\V}^2+\alpha\|\u_0\|_{\H}^2+T\|g\|_0^2\|\f\|_{\L^2}^2+\int_0^t\|\u_t(s)\|_{\V}^2\d s\nonumber\\&\qquad+\int_0^t\|\u(s)\|_{\widetilde{\L}^4}^4 \d s,
		\end{align}
		for all $t\in[0,T]$. Note that the last two terms on the right hand side of \eqref{E23} is bounded (see \eqref{E1} and \eqref{E5}). Using Sobolev embedding theorem on the first term on the right hand side of \eqref{E23}, and then taking supremum on both sides over time from $0$ to $T$, we get 
		\begin{align*}
			\frac{2\beta}{r+1}\sup_{t\in[0,T]}\|\u(t)\|_{\widetilde{\L}^{r+1}}^{r+1}+\mu \sup_{t\in[0,T]}\|\u(t)\|_{\V}^2+\int_0^T\|\u_t(t)\|_{\H}^2\d t\leq C\left( \|\u_0\|_{\V}^2+\|g\|_0^2\|\f\|_{\L^2}^2 \right),
		\end{align*}
		and it completes the proof.
	\end{proof}
	The next lemma provides the regularity of the solution for the case $r \geq 3$  and $d=2,3$ ($2\beta\mu\geq 1$ for $d=r=3$). 
	\begin{lemma}[Theorem 4.3, \cite{MTM4}]\label{lem2.5}
		Let $\boldsymbol{F}(x,t):=\f(x)g(x,t) \in \mathrm{W}^{1,1}([0,T];\H)$ and $\u_0 \in \V$ be such that $\Delta\u_0 \in \H$. Then, for $r\geq3 \ (2\beta \mu \geq1 \ \text{for} \  r=3)$, there exists one and only one function $\u :[0,T] \to \V \cap \widetilde{\L}^{r+1}$ that satisfies: $$ \u \in \mathrm{W}^{1,\infty}([0,T];\H), \ \ \Delta\u \in \mathrm{L}^\infty(0,T;\H).$$ Moreover, $\u \in \mathrm{L}^{\infty}(0,T;\H^2(\Omega) \cap \V)$.
	\end{lemma}
	One can check Corollary 2, \cite{KT2} for a result similar to that of Lemma \ref{lem2.5}. 	An analogous result of Lemma \ref{lem2.5} for the case of $d=2, \ r\in[1,3]$ can be obtained by using the $m$-accretive quantization of the linear and nonlinear operators (cf. Theorems 1.6 and 1.8 in Chapter 4, \cite{VB} for the abstract theory and Sec. 5, \cite{VBSS} for 2D NSE).
	
	\begin{lemma}\label{lemma2.6}
		Let $\u(\cdot)$ be the unique strong solution of the CBF  equations \eqref{1a}-\eqref{1d}.  Then, for $r \geq 1, \ (d=2)\text{ and for} \ r \geq 3, \ (d=3)$, we have  $\v\in\C([\epsilon,T];\V)\cap\mathrm{L}^2(0, T;\H^2(\Omega))$ and 
		\begin{align}\label{2.37}
			\sup_{t\in[\epsilon,T]}\|\v(t)\|_{\V}^2 +\int_\epsilon^T\|\v(t)\|_{\H^2}^2\d t < \infty,
		\end{align}
		for any $0<\epsilon\leq T$.	
	\end{lemma}
	\begin{proof}
		We need to prove the estimate \eqref{2.37} only.	Taking the inner product with $-\Delta\v(\cdot)$ to the first equation in \eqref{E5}, we find 
		\begin{align}\label{E301}
			\nonumber&\frac{1}{2}\frac{\d}{\d t}\|\nabla\v(t)\|_{\H}^2+ \mu\|\Delta\v(t)\|_{\H}^2+\alpha \|\v(t)\|_{\V}^2\\&\quad=(\f g_t(t),-\Delta\v(t))-((\v(t) \cdot\nabla)\u(t),-\Delta\v(t))-((\u(t) \cdot \nabla)\v(t),-\Delta\v(t))\nonumber\\&\qquad-\beta r(|\u(t)|^{r-1}\v(t),-\Delta\v(t))-(\nabla p_t,-\Delta \v(t)),
		\end{align}
		for a.e. $t\in[\epsilon, T]$, for some $0<\epsilon\leq T$.	Using H\"older's and Young's inequalities, we estimate  $|(\f g_t,-\Delta\v)|$ as 
		\begin{align*}
			|(\f g_t,-\Delta\v)|\leq\|\f\|_{\L^2}\|g_t\|_{\mathrm{L}^{\infty}}\|\Delta\v\|_{\H}\leq\frac{7}{\mu}\|\f\|_{\L^2}^2\|g_t\|_{\mathrm{L}^{\infty}}^2+\frac{\mu }{14}\|\Delta\v\|_{\H}^2.
		\end{align*}
		\vskip 0.2 cm
		\noindent\textbf{Case I:} \emph{$d=2$ and $r \geq 1$.} Using H\"older's, Gagliardo-Nirenberg's and Young's inequalities, we estimate  $|((\v \cdot\nabla)\u,-\Delta\v)|$ as 
		\begin{align}\label{2.39}
			|((\v \cdot\nabla)\u,-\Delta\v)| &\leq  \|\v\|_{\widetilde{\L}^4} \|\nabla\u\|_{\widetilde{\L}^4}\|\Delta\v\|_{\H} \nonumber\\&\leq 
			C\|\v\|_{\H}^\frac{1}{2}	\|\v\|_{\V}^\frac{1}{2}\|\u\|_{\H}^\frac{1}{4}	\|\u\|_{\H^2(\Omega)\cap \V}^\frac{3}{4}\|\Delta\v\|_{\H}\nonumber\\&\leq C\|\v\|_{\H}	\|\v\|_{\V}\|\u\|_{\H}^\frac{1}{2}	\|\u\|_{\H^2(\Omega)\cap \V}^\frac{3}{2}+\frac{\mu}{14}\|\Delta\v\|_{\H}^2
			\nonumber\\&\leq C\|\v\|_{\H}^2	\|\u\|_{\H}	\|\u\|_{\H^2(\Omega)\cap \V}^3+\frac{\alpha}{4}\|\v\|_{\V}^2+\frac{\mu}{14}\|\Delta\v\|_{\H}^2.
		\end{align}
		We estimate the terms $|((\u \cdot\nabla)\v,-\Delta\v)|$ and $\beta r|(|\u|^{r-1}\v,-\Delta\v)|$ using H\"older's, Agmon's and Young's inequalities as
		\begin{align}
			|((\u \cdot\nabla)\v,-\Delta\v)| &\leq  \|\u\|_{\widetilde{\L}^\infty} \|\v\|_{\V}\|\Delta\v\|_{\H} \nonumber\\&\leq
			C\|\u\|_{\H}\|\u\|_{\H^2(\Omega)\cap \V}\|\v\|_{\V}^2+\frac{\mu}{14}\|\Delta\v\|_{\H}^2,\label{2.40}\\
			\beta r|(|\u|^{r-1}\v,-\Delta\v)| &\leq  \beta r\|\u\|_{\widetilde{\L}^\infty}^{r-1}\|\v\|_{\H}\|\Delta\v\|_{\H} 	\nonumber\\&\leq C \|\u\|_{\H}^{r-1}\|\u\|_{\H^2(\Omega)\cap \V}^{r-1}\|\v\|_{\H}^2+\frac{\mu}{14}\|\Delta\v\|_{\H}^2.\label{2.41}
		\end{align}
		Taking the divergence on both sides of \eqref{E5} and then using the divergence free condition on $\boldsymbol{F}$, we get
		\begin{align}\label{E302}
			-\Delta p_t&=\nabla \cdot\left[ ((\u \cdot \nabla)\v)+ ((\v \cdot \nabla)\u)+\beta (|\u|^{r-1}\v)\right]\nonumber\\
			p_t&=(-\Delta)^{-1}\left\{\nabla \cdot \left[((\u \cdot \nabla)\v)+ ((\v \cdot \nabla)\u)+\beta (|\u|^{r-1}\v)\right]\right\},
		\end{align}
		in the weak sense.	Taking gradient on both sides in \eqref{E302}, we estimate the term $|(\nabla p_t,-\Delta \v)|$ as
		\begin{align}\label{2.43}
			|(\nabla p_t,-\Delta \v)| &\leq \|\nabla p_t\|_{\H}\|\Delta \v\|_{\H} \nonumber\\&\leq \|(\u \cdot \nabla)\v\|_{\H}\|\Delta \v\|_{\H}+\|(\v \cdot \nabla)\u\|_{\H}\|\Delta \v\|_{\H} +\beta \||\u|^{r-1}\v\|_{\H}\|\Delta \v\|_{\H} \nonumber\\&\leq \frac{3\mu}{14}\|\Delta \u\|_{\H}^2+ C\|\v\|_{\H}^2	\|\u\|_{\H}	\|\u\|_{\H^2(\Omega)\cap \V}^3+\frac{\alpha}{4}\|\v\|_{\V}^2+C\|\u\|_{\H}\|\u\|_{\H^2(\Omega)\cap \V}\|\v\|_{\V}^2\nonumber\\&\quad+C \|\u\|_{\H}^{r-1}\|\u\|_{\H^2(\Omega)\cap \V}^{r-1}\|\v\|_{\H}^2,
		\end{align}
		where we have used the estimates \eqref{2.39}-\eqref{2.41}. Substituting these estimates \eqref{2.39}-\eqref{2.43} in \eqref{E301}, and then integrating from $\epsilon$ to $t$, we find 
		\begin{align}\label{E303}
			&\|\nabla\v(t)\|_{\H}^2+\mu \int_\epsilon^t\|\Delta \u(s)\|_{\H}^2 \d s+\alpha  \int_\epsilon^t \|\v(s)\|_{\V}^2\d s \nonumber\\&\quad\leq 	\|\nabla\v(\cdot,\epsilon)\|_{\H}^2+\frac{14T}{\mu}\|\f\|_{\L^2}^2\|g_t\|_0^2+C\int_\epsilon^t\|\v(s)\|_{\H}^2	\|\u(s)\|_{\H}	\|\u(s)\|_{\H^2(\Omega)\cap \V}^3 \d s\nonumber\\&\qquad+C\int_\epsilon^t\|\u(s)\|_{\H}\|\u(s)\|_{\H^2(\Omega)\cap \V}\|\v(s)\|_{\V}^2 \d s+C\int_\epsilon^t \|\u(s)\|_{\H}^{r-1}\|\u(s)\|_{\H^2(\Omega)\cap \V}^{r-1}\|\v(s)\|_{\H}^2 \d s,
		\end{align}
		for all $t \in [\epsilon,T]$. An application of Gronwall's inequality in \eqref{E303} gives
		\begin{align*}
			\|\nabla\v(t)\|_{\H}^2&\leq \exp \bigg(CT\sup_{t\in[\epsilon,T]}\|\u(t)\|_{\H}\sup_{t\in[\epsilon,T]}\|\u(t)\|_{\H^2(\Omega)\cap \V}\bigg)\bigg\{	\|\nabla\v(\cdot,\epsilon)\|_{\H}^2\nonumber\\&\quad+\frac{14T}{\mu}\|\f\|_{\L^2}^2\|g_t\|_0^2+CT\sup_{t\in[\epsilon,T]}\|\v(t)\|_{\H}^2\sup_{t\in[\epsilon,T]}	\|\u(t)\|_{\H} \sup_{t\in[\epsilon,T]}\|\u(t)\|_{\H^2(\Omega)\cap \V}^3 \nonumber\\&\quad+CT\sup_{t\in[\epsilon,T]} \|\u(t)\|_{\H}^{r-1}\sup_{t\in[\epsilon,T]}\|\u(t)\|_{\H^2(\Omega)\cap \V}^{r-1}\sup_{t\in[\epsilon,T]}\|\v(t)\|_{\H}^2\bigg\},
		\end{align*}
		for all $t \in [\epsilon,T]$. Integrate the above estimate over $\epsilon$ from $0$ to $T$ and then using the energy estimates given in Lemmas \ref{lemma1}-\ref{lem2.5}, we deduce 
		\begin{align}\label{E304}
			\|\nabla\v(t)\|_{\H}^2&\leq \frac{C}{ T}\bigg\{\int_0^T\|\nabla\v(\cdot,\epsilon)\|_{\H}^2\d\epsilon+\|\f\|_{\L^2}^2\|g_t\|_0^2\bigg\}\nonumber\\&\leq C\bigg(\|\u_0\|_{\H^2(\Omega) \cap \V}+\|\f\|_{\L^2}^2\|g_t\|_0^2\bigg).
		\end{align}
		Thus, from \eqref{E303}, it is immediate that
		\begin{align}\label{E305}
			\sup_{t\in[\epsilon,T]}	\|\nabla\v(t)\|_{\H}^2+\int_\epsilon^T\|\Delta\v(t)\|_{\H}^2 \d t\leq  C\left(\|\u_0\|_{\H^2(\Omega) \cap \V}+\|\f\|_{\L^2}^2\|g_t\|_0^2\right),
		\end{align}
		and $\Delta \v \in \mathrm{L}^2(0,T;\H)$. 
		\vskip 0.2 cm
		\noindent\textbf{Case II:} \emph{$d=3$ and $r \geq 3$.}  Using H\"older's, Gagliardo-Nirenberg's and Young's inequalities, we estimate  $|((\v \cdot\nabla)\u,-\Delta\v)|$ as
		\begin{align}\label{2.46}
			|((\v \cdot\nabla)\u,-\Delta\v)| &\leq \|\v\|_{\widetilde{\L}^4} \|\nabla\u\|_{\widetilde{\L}^4}\|\Delta\v\|_{\H} \nonumber\\& \leq
			C\|\v\|_{\H}^\frac{1}{4}	\|\v\|_{\V}^\frac{3}{4}\|\u\|_{\H}^\frac{1}{8}	\|\u\|_{\H^2(\Omega)\cap \V}^\frac{7}{8}\|\Delta\v\|_{\H}\nonumber\\&\leq C\|\v\|_{\H}^\frac{1}{2}	\|\v\|_{\V}^\frac{3}{2}\|\u\|_{\H}^\frac{1}{4}	\|\u\|_{\H^2(\Omega)\cap \V}^\frac{7}{4}+\frac{\mu}{14}\|\Delta\v\|_{\H}^2 \nonumber
			\\& \leq C\|\v\|_{\H}^2	\|\u\|_{\H}	\|\u\|_{\H^2(\Omega)\cap \V}^7+\frac{\alpha}{4}\|\v\|_{\V}^2+\frac{\mu}{14}\|\Delta\v\|_{\H}^2.
		\end{align}
		We estimate the terms $|((\u \cdot\nabla)\v,-\Delta\v)|$ and $\beta r|(|\u|^{r-1}\v,-\Delta\v)|$ using H\"older's, Agmon's and Young's inequalities as 
		\begin{align}
			|((\u \cdot\nabla)\v,-\Delta\v)|& \leq \|\u\|_{\widetilde{\L}^\infty} \|\v\|_{\V}\|\Delta\v\|_{\H} \nonumber\\&\leq 	C\|\u\|_{\V}\|\u\|_{\H^2(\Omega)\cap \V}\|\v\|_{\V}^2+\frac{\mu}{14}\|\Delta\v\|_{\H}^2,\label{2.47}\\
			\beta r|(|\u|^{r-1}\v,-\Delta\v)| &\leq  \beta r\|\u\|_{\widetilde{\L}^\infty}^{r-1}\|\v\|_{\H}\|\Delta\v\|_{\H}	\nonumber\\&\leq C \|\u\|_{\V}^{r-1}\|\u\|_{\H^2(\Omega)\cap \V}^{r-1}\|\v\|_{\H}^2+\frac{\mu}{14}\|\Delta\v\|_{\H}^2.\label{2.48}
		\end{align}
		Using the estimates \eqref{2.46}-\eqref{2.48}, we estimate the term $|(\nabla p_t,-\Delta \v)|$ as
		\begin{align}\label{2.49}
			|(\nabla p_t,-\Delta \v)| &\leq \|\nabla p_t\|_{\H}\|\Delta \v\|_{\H} \nonumber\\&\leq\|(\u \cdot \nabla)\v\|_{\H}\|\Delta \v\|_{\H}+\|(\v \cdot \nabla)\u\|_{\H}\|\Delta \v\|_{\H} +\beta \||\u|^{r-1}\v\|_{\H}\|\Delta \v\|_{\H}\nonumber \\&\leq \frac{3\mu}{14}\|\Delta \u\|_{\H}^2+ C\|\v\|_{\H}^2	\|\u\|_{\H}	\|\u\|_{\H^2(\Omega)\cap \V}^7+\frac{\alpha}{4}\|\v\|_{\V}^2+C\|\u\|_{\V}\|\u\|_{\H^2(\Omega)\cap \V}\|\v\|_{\V}^2\nonumber\\&\quad+C \|\u\|_{\V}^{r-1}\|\u\|_{\H^2(\Omega)\cap \V}^{r-1}\|\v\|_{\H}^2.
		\end{align}
		Substituting the estimates \eqref{2.46}-\eqref{2.49} in \eqref{E301}, and then integrating from $ \epsilon$ to $t$, we find 
		\begin{align}\label{E306}
			&\|\nabla\v(t)\|_{\H}^2+\mu \int_\epsilon^t\|\Delta \u(s)\|_{\H}^2 \d s+\alpha  \int_\epsilon^t \|\v(s)\|_{\V}^2\d s \nonumber\\&\quad\leq 	\|\nabla\v(\cdot,\epsilon)\|_{\H}^2+\frac{14T}{\mu}\|\f\|_{\L^2}^2\|g_t\|_0^2+C\int_\epsilon^t\|\v(s)\|_{\H}^2	\|\u(s)\|_{\H}	\|\u(s)\|_{\H^2(\Omega)\cap \V}^7 \d s\nonumber\\&\qquad+C\int_\epsilon^t\|\u(s)\|_{\V}\|\u(s)\|_{\H^2(\Omega)\cap \V}\|\v(s)\|_{\V}^2 \d s+C\int_\epsilon^t \|\u(s)\|_{\V}^{r-1}\|\u(s)\|_{\H^2(\Omega)\cap \V}^{r-1}\|\v(s)\|_{\H}^2 \d s,
		\end{align}
		for all $t \in [\epsilon,T]$ with $\epsilon >0$. An application of Gronwall's inequality in \eqref{E306}, and then a calculation similar to $r \geq 1$ and $d=2$ yield 
		\begin{align*}
			\sup_{t\in[\epsilon,T]}\|\nabla\v(t)\|_{\H}^2 +\int_\epsilon^T\|\Delta \v(t)\|_{\H}^2\d t<\infty,
		\end{align*}
		for $r\geq 3$ and $d=3$, and $\Delta\v \in \mathrm{L}^{2}(0,T;\H)$. Since $\Omega$ is of class $\C^2$, from the elliptic regularity for the Stokes problem (Cattabriga regularity theorem, see \cite{LCa,Te1}, etc), one can conclude that $\v\in\mathrm{L}^2(0,T;\H^2(\Omega))$. 
	\end{proof}
	
	\begin{lemma}\label{lemma2.7}
		Let $\u(\cdot)$ be the unique strong solution of the CBF  equations \eqref{1a}-\eqref{1d}.  Then, for $r \geq 1, \ (d=2)\text{ and for} \ r \geq 3, \ (d=3)$, we have 
		\begin{align}\label{E35}
			\sup_{t\in[\epsilon,T]}\|\v(t)\|_{\V}^2 +\int_\epsilon^T\|\v_t(t)\|_{\H}^2\d t < \infty,
		\end{align}
		for any $0<\epsilon\leq T$.	
	\end{lemma}
	\begin{proof}
		Taking the inner product with $\v_t(\cdot)$ to the first equation in \eqref{E5}, we find 
		\begin{align}\label{E36}
			\nonumber&\|\v_t(t)\|_{\H}^2+\frac{\mu}{2}\frac{\d}{\d t}\|\v(t)\|_{\V}^2+\frac{\alpha}{2}\frac{\d}{\d t} \|\v(t)\|_{\H}^2=(\f g_t(t),\v_t(t))\\&\quad-((\v(t) \cdot\nabla)\u(t),\v_t(t))-((\u(t) \cdot \nabla)\v(t),\v_t(t))-\beta r(|\u(t)|^{r-1}\v(t),\v_t(t)),
		\end{align}
		for a.e. $t\in[\epsilon, T]$, for some $0<\epsilon\leq T$.	Using H\"older's and Young's inequalities, we estimate  $|(\f g_t,\v_t)|$ as 
		\begin{align*}
			|(\f g_t,\v_t)|\leq\|\f\|_{\L^2}\|g_t\|_{\mathrm{L}^{\infty}}\|\v_t\|_{\H}\leq4\|\f\|_{\L^2}^2\|g_t\|_{\mathrm{L}^{\infty}}^2+\frac{1 }{8}\|\v_t\|_{\H}^2.
		\end{align*}
		\vskip 0.2 cm
		\noindent\textbf{Case I:} \emph{$d=2$ and $r \geq 1$.} Calculations similar to \eqref{2.39}-\eqref{2.41} yield 
		\begin{align*}
			|((\v \cdot\nabla)\u,\v_t)| &\leq C\|\v\|_{\H}	\|\v\|_{\V}\|\u\|_{\H}^\frac{1}{2}	\|\u\|_{\H^2(\Omega)\cap \V}^\frac{3}{2}+\frac{1}{8}\|\v_t\|_{\H}^2,\nonumber\\
			|((\u \cdot\nabla)\v,\v_t)|&\leq
			C\|\u\|_{\H}^\frac{1}{2}\|\u\|_{\H^2(\Omega)\cap \V}^\frac{1}{2}\|\v\|_{\V}\|\v_t\|_{\H}\leq 	C\|\u\|_{\H}\|\u\|_{\H^2(\Omega)\cap \V}\|\v\|_{\V}^2+\frac{1}{8}\|\v_t\|_{\H}^2,\\
			\beta r|(|\u|^{r-1}\v,\v_t)|& 
			\leq C \|\u\|_{\H}^\frac{r-1}{2}\|\u\|_{\H^2(\Omega)\cap \V}^\frac{r-1}{2}\|\v\|_{\H}\|\v_t\|_{\H} 	\leq C \|\u\|_{\H}^{r-1}\|\u\|_{\H^2(\Omega)\cap \V}^{r-1}\|\v\|_{\H}^2+\frac{1}{8}\|\v_t\|_{\H}^2.
		\end{align*}
		Substituting these estimates in \eqref{E36}, and then integrating from $0<\epsilon$ to $t$, we find
		\begin{align}\label{E38}
			\nonumber&\int_\epsilon^t\|\v_t(s)\|_{\H}^2\d s+\mu \|\v(t)\|_{\V}^2+\alpha \|\v(t)\|_{\H}^2 \nonumber\\&\quad\leq \mu \|\v(\cdot,\epsilon)\|_{\V}^2+\alpha \|\v(\cdot,\epsilon)\|_{\H}^2+8T\|\f\|_{\L^2}^2\|g_t\|_0^2\nonumber\\&\qquad+C\int_\epsilon^t \|\v(s)\|_{\H}	\|\v(s)\|_{\V}\|\u(s)\|_{\H}^\frac{1}{2}	\|\u(s)\|_{\H^2(\Omega)\cap \V}^\frac{3}{2}\d s\nonumber\\&\qquad+C\int_\epsilon^t \|\u(s)\|_{\H}\|\u(s)\|_{\H^2(\Omega)\cap \V}\|\v(s)\|_{\V}^2 \d s +C\int_\epsilon^t \|\u(s)\|_{\H}^{r-1}\|\u(s)\|_{\H^2(\Omega)\cap \V}^{r-1}\|\v(s)\|_{\H}^2\d s,
		\end{align}
		for all $t \in [\epsilon,T]$. Applying Gronwall's inequality in \eqref{E38}, we obtain
		\begin{align*}
			\mu \|\v(t)\|_{\V}^2  
			&\leq \bigg\{\mu \|\v(\cdot,\epsilon)\|_{\V}^2+\alpha \|\v(\cdot,\epsilon)\|_{\H}^2+8T\|\f\|_{\L^2}^2\|g_t\|_0^2\nonumber\\&\qquad+C \sup_{t\in[\epsilon,T]}\|\v(t)\|_{\H}	\sup_{t\in[\epsilon,T]}\|\u(t)\|_{\H}^\frac{1}{2}	\sup_{t\in[\epsilon,T]}\|\u(t)\|_{\H^2(\Omega)\cap \V}^\frac{3}{2}(T-\epsilon)^\frac{1}{2}\left(\int_\epsilon^T \|\v(t)\|_{\V}^2\d t\right)^\frac{1}{2}\nonumber\\&\qquad+C T \sup_{t\in[\epsilon,T]}\|\u(t)\|_{\H}^{r-1}\sup_{t\in[\epsilon,T]}\|\u(t)\|_{\H^2(\Omega)\cap \V}^{r-1}\sup_{t\in[\epsilon,T]}\|\v(t)\|_{\H}^2\d t\bigg\}
			\\&\qquad \times	\exp \bigg(C T \sup_{t\in[\epsilon,T]}\|\u(t)\|_{\H} \sup_{t\in[\epsilon,T]}\|\u(t)\|_{\H^2(\Omega)\cap \V} \bigg),
		\end{align*}
		for all $t \in [\epsilon,T]$. 	Integrate the above estimate over $\epsilon$ from $0$ to $T$ and then using the energy estimates given in Lemmas \ref{lemma1}-\ref{lem2.5}, we deduce 
		\begin{align}
			\|\v(t)\|_{\V}^2&\leq \frac{C}{\mu  T}\bigg\{\left(\mu +\frac{\alpha}{\lambda_1}\right)\int_0^T\|\v(\cdot,\epsilon)\|_{\V}^2\d\epsilon+\|\f\|_{\L^2}^2\|g_t\|_0^2\bigg\}\nonumber\\&\leq C\left(\|\u_0\|_{\H^2(\Omega) \cap \V}+\|\f\|_{\L^2}^2\|g_t\|_0^2\right).
		\end{align}
		Thus from  \eqref{E38}, it is immediate that
		\begin{align*}
			\sup_{t\in[\epsilon,T]}\|\v(t)\|_{\V}^2 +\int_\epsilon^T\|\v_t(t)\|_{\H}^2\d t \leq C\left(\|\u_0\|_{\H^2(\Omega) \cap \V}+\|\f\|_{\L^2}^2\|g_t\|_0^2\right),
		\end{align*}
		and $\v_t \in \mathrm{L}^{2}(\epsilon,T;\H)$.
		
		\vskip 0.2 cm
		\noindent\textbf{Case II:} \emph{$d=3$ and $r \geq 3$.}  Calculations similar to \eqref{2.46}-\eqref{2.48} gives 
		\begin{align*}
			|((\v \cdot\nabla)\u,\v_t)| &\leq C\|\v\|_{\H}^\frac{1}{2}	\|\v\|_{\V}^\frac{3}{2}\|\u\|_{\H}^\frac{1}{4}	\|\u\|_{\H^2(\Omega)\cap \V}^\frac{7}{4}+\frac{1}{8}\|\v_t\|_{\H}^2,\nonumber\\
			|((\u \cdot\nabla)\v,\v_t)| & \leq
			C\|\u\|_{\V}^\frac{1}{2}\|\u\|_{\H^2(\Omega)\cap \V}^\frac{1}{2}\|\v\|_{\V}\|\v_t\|_{\H}\leq 	C\|\u\|_{\V}\|\u\|_{\H^2(\Omega)\cap \V}\|\v\|_{\V}^2+\frac{1}{8}\|\v_t\|_{\H}^2,\\
			\beta r|(|\u|^{r-1}\v,\v_t)| &
			\leq C \|\u\|_{\V}^\frac{r-1}{2}\|\u\|_{\H^2(\Omega)\cap \V}^\frac{r-1}{2}\|\v\|_{\H}\|\v_t\|_{\H} 	\leq C \|\u\|_{\V}^{r-1}\|\u\|_{\H^2(\Omega)\cap \V}^{r-1}\|\v\|_{\H}^2+\frac{1}{8}\|\v_t\|_{\H}^2.
		\end{align*}
		Substituting these estimates in \eqref{E36}, and then integrating from $\epsilon$ to $t$, we find
		\begin{align}\label{E40}
			\nonumber&\int_\epsilon^t\|\v_t(s)\|_{\H}^2\d s+\mu \|\v(t)\|_{\V}^2+\alpha \|\v(t)\|_{\H}^2 \nonumber\\&\quad\leq \mu \|\v(\cdot,\epsilon)\|_{\V}^2+\alpha \|\v(\cdot,\epsilon)\|_{\H}^2+8T\|\f\|_{\L^2}^2\|g_t\|_0^2\nonumber\\&\qquad+C\int_\epsilon^t \|\v(s)\|_{\H}^\frac{1}{2}	\|\v(s)\|_{\V}^\frac{3}{2}\|\u(s)\|_{\H}^\frac{1}{4}	\|\u(s)\|_{\H^2(\Omega)\cap \V}^\frac{7}{4}\d s\nonumber\\&\qquad+C\int_\epsilon^t \|\u(s)\|_{\V}\|\u(s)\|_{\H^2(\Omega)\cap \V}\|\v(s)\|_{\V}^2 \d s +C\int_\epsilon^t \|\u(s)\|_{\V}^{r-1}\|\u(s)\|_{\H^2(\Omega)\cap \V}^{r-1}\|\v(s)\|_{\H}^2\d s,
		\end{align}
		for all $t \in [\epsilon,T]$ with $\epsilon >0$. An application of Gronwall's inequality in \eqref{E40}, and then a calculation similar to $r \geq 1$ and $d=2$, one can easily conclude that
		\begin{align*}
			\sup_{t\in[\epsilon,T]}\|\v(t)\|_{\V}^2 +\int_\epsilon^T\|\v_t(t)\|_{\H}^2\d t<\infty,
		\end{align*}
		for $r\geq 3$ and $d=3$, and $\v_t \in \mathrm{L}^{2}(\epsilon,T;\H)$. 
	\end{proof}
	\begin{remark}
		From Lemmas \ref{lemma2.6} and \ref{lemma2.7}, it is clear that $$\u_t\in\C([\epsilon,T];\V)\cap\mathrm{L}^2(0,T;\H^2(\Omega)\cap \V),\ \u_{tt}\in\mathrm{L}^2(0,T;\H),$$ for any $0<\epsilon\leq T$. 	Furthermore the fact that $\u_t\in\mathrm{H}^1(0,T;\H)$  implies that $\u_t\in\C([0,T];\H)$. 
	\end{remark}

	\section{Proof of Theorem \ref{thm2}}\label{sec5}\setcounter{equation}{0}
	The energy estimates obtained in the previous section allow us to prove the existence and uniqueness of a solution to the inverse problem \eqref{1a}-\eqref{1e}	as well as the stability of the solution. For the existence of a solution to the  inverse problem \eqref{1a}-\eqref{1e},  making use of Theorem \ref{thm1}, 	it is enough to prove that the nonlinear operator $\mathcal{B}$ has a fixed point in $\mathcal{D}$, which follows by an application of Schauder's fixed point theorem:
	\begin{theorem}[Schauder's fixed point theorem \cite{HZ}]\label{thmS}
		Let $\mathcal{D}$ be a non-empty closed bounded subset of a Banach space $X$ and assume that $\mathcal{B} : \mathcal{D} \to \mathcal{D}$ is compact. Then, $\mathcal{B}$ has at least one fixed point in $\mathcal{D}$.
	\end{theorem}
	
	\subsection{Existence}
	We shall begin by proving Theorem \ref{thm2} (i) by checking that the nonlinear operator $\mathcal{B}$ defined in \eqref{1i} satisfies all the assumptions given in Theorem \ref{thmS}. The following lemma provides a well-posedness of the operator $\mathcal{B}$. 
	\begin{lemma}
		Let $\u_0, \boldsymbol{\varphi} \in \H^2(\Omega) \cap \V $, $\nabla \psi \in \G(\Omega)$, $g, g_t \in \C(\overline\Omega\times [0,T])$ satisfy the assumption \eqref{1f}. Then the operator $\mathcal{B}$ maps $\mathcal{D}$ into $\L^2(\Omega)$.
	\end{lemma}
	\begin{proof}
		Let $\f \in \mathcal{D}$. From \eqref{1j}, we deduce that
		\begin{align}\label{B1}
			\| \mathcal{B}\f\|_{\L^2}\leq \frac{1}{g_T}\left[\|\u_t(x,T)\|_{\H}+\|(\boldsymbol{\varphi}\cdot \nabla)\boldsymbol{\varphi}\|_{\H}+\|\nabla \psi-\mu \Delta\boldsymbol{\varphi}\|_{\H}+\alpha\|\boldsymbol{\varphi}\|_{\H}+ \beta\||\boldsymbol{\varphi}|^{r-1}\boldsymbol{\varphi}\|_{\H}\right].
		\end{align}
		Using H\"older's, Agmon's and Poincar\'e's inequalities, we obtain
		\begin{align*}
			\|(\boldsymbol{\varphi}\cdot \nabla)\boldsymbol{\varphi}\|_{\H} \leq\|\boldsymbol{\varphi}\|_{\widetilde{\L}^{\infty}}\|\nabla\boldsymbol{\varphi}\|_{\H}\leq C \|\boldsymbol{\varphi}\|_{\H}^\frac{1}{2}\|\boldsymbol{\varphi}\|_{\H^2(\Omega) \cap \V}^\frac{1}{2}\|\nabla\boldsymbol{\varphi}\|_{\H} \leq C\|\boldsymbol{\varphi}\|_{\H^2(\Omega) \cap \V}^{2}, \ \ (\text{for}\ \  d=2),
		\end{align*}
		\begin{align*}
			\|(\boldsymbol{\varphi}\cdot \nabla)\boldsymbol{\varphi}\|_{\H} \leq\|\boldsymbol{\varphi}\|_{\widetilde{\L}^{\infty}}\|\nabla\boldsymbol{\varphi}\|_{\H}\leq C \|\boldsymbol{\varphi}\|_{\V}^\frac{1}{2}\|\boldsymbol{\varphi}\|_{\H^2(\Omega) \cap \V}^\frac{1}{2}\|\nabla\boldsymbol{\varphi}\|_{\H} \leq C\|\boldsymbol{\varphi}\|_{\H^2(\Omega) \cap \V}^{2},\ \ (\text{for}\ \  d=3).
		\end{align*}
		Using Sobolev embedding theorem, one can easily find
		\begin{align}
			\beta\||\boldsymbol{\varphi}|^{r-1}\boldsymbol{\varphi}\|_{\H}&\leq \beta \|\boldsymbol{\varphi}\|_{\widetilde{\L}^{2r}}^r\leq C\|\boldsymbol{\varphi}\|_{\V}^r < +\infty \ \ \ \ (\text{for}\ \  d=2), \label{B2}\\ 	\beta\||\boldsymbol{\varphi}|^{r-1}\boldsymbol{\varphi}\|_{\H}&\leq \beta \|\boldsymbol{\varphi}\|_{\widetilde{\L}^{2r}}^r\leq C\|\boldsymbol{\varphi}\|_{\H^2(\Omega) \cap \V}^r < +\infty\ \ \ (\text{for}\ \  d=3).\label{B3}
		\end{align} 
		At the final time $t=T$, from  \eqref{E7}, we get
		\begin{align}\label{B4}
			\|\u_t(x,T)\|_{\H}^2\leq C\left(\|\u_0\|_{\H^2(\Omega) \cap \V}^2+\|\f\|_{\L^2}^2\|g_t\|_0^2\right).
		\end{align}
		From \eqref{E6}, we  infer that 
		\begin{align}\label{B5}
			\|\v(0)\|_{\H}\leq L_0,
		\end{align}
		holds for some positive constant $L_0$. Hence, by the above estimates, we obtain  
		\begin{align*}
			\| \mathcal{B}\f\|_{\H}&\leq \frac{1}{g_T}\left[L_0+C\|\f\|_{\L^2}\|g_t\|_0+C\|\boldsymbol{\varphi}\|_{\H^2(\Omega) \cap \V}^{2}+\|\nabla \psi-\mu \Delta\boldsymbol{\varphi}\|_{\H}+ C\|\boldsymbol{\varphi}\|_{\H^2(\Omega) \cap \V}^r\right],
		\end{align*}
		and  the proof follows.
	\end{proof}
	The next lemma establishes the existence of a solution to the inverse problem \eqref{1a}-\eqref{1e}.
	\begin{lemma}
		Let $\u_0 \in \H^{2}(\Omega) \cap \V$ and $g, g_t \in \C(\overline\Omega\times [0,T])$ satisfy the assumption \eqref{1f}. Then the operator $\mathcal{B}$ is completely continuous on $\mathcal{D}$.
	\end{lemma}
	\begin{proof}
		It is enough to prove that the operator $\mathcal{A}$ is completely continuous. We will choose an element $\f$ in $\mathcal{D}$ and consider an arbitrary sequence $\{\f_k\}_{k=1}^{\infty}$ of elements $\f_k \in \mathcal{D}$ such that 
		\begin{align}
			\|\f_k-\f\|_{\L^2} \rightarrow 0 \    \text{ as } \ k \rightarrow \infty.
		\end{align}
		Let $\{\u_k,\nabla p_k\}$ be the solution of the direct problem \eqref{1a}-\eqref{1d} corresponding to the external forcing $\f_k g$ and the initial velocity $\u_0$ and let $\{\u,\nabla p\}$ be the solution of the same problem corresponding to the external forcing $\f g$ and the initial velocity $\u_0$. It is clear that the functions $(\u_k-\u)$ and $\nabla (p_k-p)$ satisfy the system
		\begin{equation}\label{C1}
			\left\{
			\begin{aligned}
				(\u_k-\u)_t-\mu \Delta (\u_k-\u)+\alpha(\u_k-\u)+\nabla(p_k-p)&=\boldsymbol{F}_k, \ \ \text{in} \ \Omega \times (0,T), \\ \nabla \cdot (\u_k-\u)&=0, \ \ \ \ \text{in} \ \Omega \times (0,T), \\
				\u_k-\u&=\boldsymbol{0}, \ \ \ \ \text{on} \ \partial\Omega \times [0,T), \\  \u_k-\u&=\boldsymbol{0}, \ \  \ \  \text{in} \ \Omega \times \{0\},
			\end{aligned}
			\right.
		\end{equation}
		where 
		\begin{align}\label{F_k}
			\boldsymbol{F}_k=(\f_k-\f)g-(\u_k \cdot \nabla)\u_k+(\u\cdot \nabla)\u-\beta(|\u_k|^{r-1}\u_k-|\u|^{r-1}\u).
		\end{align}
		The rest of the proof is divided into following steps. 
		\vskip 0.2 cm 
		\noindent\textbf{Step I:} \emph{$\u_k\to\u$  in $\mathrm{L}^{\infty}(0,T;\H)\cap\mathrm{L}^2(0,T;\V)\cap\mathrm{L}^{r+1}(0,T;\wi\L^{r+1})$:} 
		Taking the inner product with $(\u_k-\u)(\cdot)$ to the first equation in \eqref{C1}, we obtain
		\begin{align}\label{C2}
			&\frac{1}{2}\frac{\d}{\d t}\|(\u_k-\u)(t)\|_{\H}^2+\mu \|(\u_k-\u)(t)\|_{\V}^2+\alpha\|(\u_k-\u)(t)\|_{\H}^2\nonumber\\&\quad=((\f_k-\f)g(t),(\u_k-\u)(t))-\left(((\u_k-\u)(t) \cdot \nabla)\u(t),(\u_k-\u)(t) \right)\nonumber\\&\qquad-\beta \left(|\u_k(t)|^{r-1}\u_k(t)-|\u(t)|^{r-1}\u(t),(\u_k-\u)(t) \right) .
		\end{align} 
		\vskip 0.2 cm 
		\noindent\textbf{Case I:} \emph{$d=2,3$, $r>3$.} 
		A calculation similar to \eqref{3e} gives
		\begin{align}\label{C3}
			&\beta\left(|\u_k|^{r-1}\u_k-|\u|^{r-1}\u,\u_k-\u\right)\geq \frac{\beta}{2}\||\u_k|^{\frac{r-1}{2}}(\u_k-\u)\|_{\H}^2+\frac{\beta}{2}\||\u|^{\frac{r-1}{2}}(\u_k-\u)\|_{\H}^2,
		\end{align}
		for $r\geq 1$. An estimate similar to \eqref{3j} yields
		\begin{align}\label{C4}
			&|\left(((\u_k-\u) \cdot \nabla)\u,\u_k-\u \right)|=|\left(((\u_k-\u) \cdot \nabla)(\u_k-\u),\u\right)|\nonumber\\&\quad\leq\frac{\mu }{2}\|\u_k-\u\|_{\V}^2+\frac{\beta}{4}\||\u|^{\frac{r-1}{2}}(\u_k-\u)\|_{\H}^2+\frac{r-3}{2\mu(r-1)}\left(\frac{4}{\beta\mu (r-1)}\right)^{\frac{2}{r-3}}\|\u_k-\u\|_{\H}^2,
		\end{align}
		for $r>3$. Combining \eqref{C3} and \eqref{C4}, we obtain
		\begin{align}\label{C5}
			&\beta\left(|\u_k|^{r-1}\u_k-|\u|^{r-1}\u,\u_k-\u \right) +\left(((\u_k-\u) \cdot \nabla)(\u_k-\u),\u\right)\nonumber\\&\quad\geq \frac{\beta}{2}\||\u_k|^{\frac{r-1}{2}}(\u_k-\u)\|_{\H}^2+\frac{\beta}{4}\||\u|^{\frac{r-1}{2}}(\u_k-\u)\|_{\H}^2\nonumber\\&\qquad-\frac{r-3}{2\mu(r-1)}\left(\frac{4}{\beta\mu (r-1)}\right)^{\frac{4}{r-3}}\|\u_k-\u\|_{\H}^2-\frac{\mu }{2}\|\u_k-\u\|_{\V}^2.
		\end{align}
		It is important to note that
		\begin{align*}
			\|\u_k-\u\|_{\widetilde{\L}^{r+1}}^{r+1}&=\int_{\Omega}|\u_k(x)-\u(x)|^{r-1} |\u_k(x)-\u(x)|^{2} \d x \nonumber\\&\leq 2^{r-2} \int_{\Omega}\left(|\u_k(x)|^{r-1}+|\u(x)|^{r-1}\right) |\u_k(x)-\u(x)|^{2} \d x
			\nonumber\\& \leq  2^{r-2}\left( \||\u_k|^{\frac{r-1}{2}}(\u_k-\u)\|_{\H}^2+\||\u|^{\frac{r-1}{2}}(\u_k-\u)\|_{\H}^2\right).
		\end{align*}
		From the above inequality, we have
		\begin{align}\label{C6}
			\frac{2^{2-r}\beta}{4}\|\u_k-\u\|_{\widetilde{\L}^{r+1}}^{r+1}\leq \frac{\beta}{4}\||\u_k|^{\frac{r-1}{2}}(\u_k-\u)\|_{\H}^2+\frac{\beta}{4}\||\u|^{\frac{r-1}{2}}(\u_k-\u)\|_{\H}^2.
		\end{align}
		Thus, from \eqref{C5}, one can easily deduce that
		\begin{align}\label{C61}
			&\beta\left(|\u_k|^{r-1}\u_k-|\u|^{r-1}\u,\u_k-\u\right) +\left(((\u_k-\u) \cdot \nabla)(\u_k-\u),\u\right)\nonumber\\&\quad\geq \frac{\beta}{2^r}\|\u_k-\u\|_{\widetilde{\L}^{r+1}}^{r+1}-\frac{r-3}{2\mu(r-1)}\left(\frac{4}{\beta\mu (r-1)}\right)^{\frac{2}{r-3}}\|\u_k-\u\|_{\H}^2-\frac{\mu }{2}\|\u_k-\u\|_{\V}^2.
		\end{align}
		Using the Cauchy-Schwarz and Young's inequalities, we get
		\begin{align}\label{C62}
			\left((\f_k-\f)g,\u_k-\u\right)\leq \frac{1}{2\alpha}\|\f_k-\f\|_{\L^2}^2\|g\|_{\mathrm{L}^{\infty}}^2+\frac{\alpha}{2}\|\u_k-\u\|_{\H}^2.
		\end{align}
		Substituting \eqref{C61} and \eqref{C62} in \eqref{C2}, and then integrating from $0$ to $t$, we find
		\begin{align}\label{C7}
			&\|(\u_k-\u)(t)\|_{\H}^2+\mu \int_0^t \|(\u_k-\u)(s)\|_{\V}^2\d s\nonumber\\&\qquad+\alpha \int_0^t \|(\u_k-\u)(s)\|_{\H}^2\d s+\frac{\beta}{2^{r-1}}\int_0^t\|(\u_k-\u)(s)\|_{\widetilde{\L}^{r+1}}^{r+1}\d s\nonumber\\&\quad\leq \frac{T}{\alpha}\|\f_k-\f\|_{\L^2}^2\|g\|_0^2+\frac{r-3}{\mu(r-1)}\left(\frac{4}{\beta\mu (r-1)}\right)^{\frac{2}{r-3}} \int_0^t\|(\u_k-\u)(s)\|_{\H}^2\d s,
		\end{align}
		for all $t \in [0,T]$.
		Applying Gronwall's inequality in \eqref{C7}, and then taking supremum on both sides  over time from $0$ to $T$, we obtain
		\begin{align}\label{C8}
			&\sup_{t\in[0,T]}\|(\u_k-\u)(t)\|_{\H}^2+\mu \int_0^T \|(\u_k-\u)(t)\|_{\V}^2\d t \nonumber\\&\qquad+\alpha \int_0^t \|(\u_k-\u)(s)\|_{\H}^2\d t +\frac{\beta}{2^{r-1}}\int_0^T\|(\u_k-\u)(t)\|_{\widetilde{\L}^{r+1}}^{r+1}\d t\nonumber\\&\quad\leq \left(\frac{T}{\alpha}\|\f_k-\f\|_{\L^2}^2\|g\|_0^2 \right) \exp \left\{\frac{2(r-3)}{\mu(r-1)}\left(\frac{4}{\beta\mu (r-1)}\right)^{\frac{2}{r-3}}T\right\}.
		\end{align}
		Since the norm $\|\f_k-\f\|_{\L^2} \rightarrow 0 \ \ \text{as} \ \ k \rightarrow \infty $, from \eqref{C8}, it  follows that
		\begin{align}\label{cg}
			&	\sup_{t\in[0,T]}\|(\u_k-\u)(t)\|_{\H}^2+\mu \int_0^T \|(\u_k-\u)(t)\|_{\V}^2\d t+\frac{\beta}{2^{r-1}}\int_0^T\|(\u_k-\u)(t)\|_{\widetilde{\L}^{r+1}}^{r+1}\d t \nonumber\\& \rightarrow 0 \ \  \text{as} \ \ k \rightarrow \infty,
		\end{align}
		for $d=2,3$  and $r>3$.
		\vskip 0.2 cm 
		\noindent\textbf{Case II:} \emph{$d=r=3$.} 
		From \eqref{C3}, we find
		\begin{align}\label{C9}
			\beta\left(|\u_k|^2\u_k-|\u|^2\u,\u_k-\u\right)\geq \frac{\beta}{2}\||\u_k|(\u_k-\u)\|_{\H}^2+\frac{\beta}{2}\||\u|(\u_k-\u)\|_{\H}^2.
		\end{align}
		\begin{align}\label{C10}
			|\left(((\u_k-\u) \cdot \nabla)(\u_k-\u),\u \right)|&\leq \|\u_k-\u\|_{\V}\||\u|(\u_k-\u)\|_{\H} \nonumber\\&\leq \mu \theta \|\u_k-\u\|_{\V}^2+\frac{1}{4\mu \theta}\||\u|(\u_k-\u)\|_{\H}^2,
		\end{align}
		for  $0<\theta<1$,	where we have used the Cauchy-Schwarz and Young's inequalities. Combining \eqref{C9} and \eqref{C10}, we obtain
		\begin{align}\label{C11}
			\nonumber&\beta\left(|\u_k|^2\u_k-|\u|^2\u,\u_k-\u\right)+\left(((\u_k-\u) \cdot \nabla)(\u_k-\u),\u \right)\nonumber\\&\quad\geq \frac{\beta}{2}\||\u_k|(\u_k-\u)\|_{\H}^2+\left(\frac{\beta}{2}-\frac{1}{4\mu \theta }\right)\||\u|(\u_k-\u)\|_{\H}^2- \mu\theta\|(\u_k-\u)\|_{\V}^2\nonumber\\&\quad\geq \frac{1}{2}\left(\beta-\frac{1}{2\mu \theta}\right)\|\u_k-\u\|_{\widetilde{\L}^{4}}^4-\mu\theta\|\u_k-\u\|_{\V}^2.
		\end{align}
		Using the Cauchy-Schwarz and Young's inequalities, we get
		\begin{align}\label{C63}
			\left((\f_k-\f)g,\u_k-\u\right)\leq \frac{1}{2\alpha}\|\f_k-\f\|_{\L^2}^2\|g\|_{\mathrm{L}^{\infty}}^2+\frac{\alpha}{2}\|\u_k-\u\|_{\H}^2.
		\end{align}
		Substituting \eqref{C11} and \eqref{C63} in \eqref{C2}, we deduce that
		\begin{align*}
			&\|(\u_k-\u)(t)\|_{\H}^2+2\mu (1-\theta)\int_0^t \|(\u_k-\u)(s)\|_{\V}^2\d  s+\alpha\int_0^t \|(\u_k-\u)(s)\|_{\H}^2\d  s\nonumber\\&\quad+\left(\beta-\frac{1}{2\mu \theta}\right)\int_0^t\|(\u_k-\u)(s)\|_{\widetilde{\L}^{4}}^{4}\d s\leq \frac{T}{\alpha}\|\f_k-\f\|_{\L^2}^2\|g\|_0^2.
		\end{align*}
		For $2\beta \mu \geq1$ and convergence of the norm $\|\f_k-\f\|_{\L^2}  \rightarrow 0 \ \ \text{as} \ \ k \rightarrow \infty $, it is immediate that
		\begin{align*}
			&\sup_{t\in[0,T]}\|(\u_k-\u)(t)\|_{\H}^2+\mu(1-\theta) \int_0^T \|(\u_k-\u)(t)\|_{\V}^2\d  t+\left(\beta-\frac{1}{2\mu \theta }\right)\int_0^T\|(\u_k-\u)(t)\|_{\widetilde{\L}^{4}}^{4}\d t  \\& \rightarrow 0 \ \ \text{as} \ \  k \rightarrow \infty.
		\end{align*}
		\vskip 0.2 cm 
		\noindent\textbf{Case III:} \emph{$d=2$, $r\in[1,3]$.} Combining \eqref{C3} and \eqref{C6} and substituting it in \eqref{C2}, we obtain
		\begin{align*}
			&\frac{1}{2}\frac{\d}{\d t}\|(\u_k-\u)(t)\|_{\H}^2+\mu \|(\u_k-\u)(t)\|_{\V}^2+\alpha \|(\u_k-\u)(t)\|_{\H}^2+\frac{\beta}{2^{r}}\|(\u_k-\u)(t)\|_{\widetilde{\L}^{r+1}}^{r+1}\\&\leq\|\f_k-\f\|_{\L^2}\|g\|_{\mathrm{L}^{\infty}}\|(\u_k-\u)(t)\|_{\H} +\|(\u_k-\u)(t)\|_{\widetilde{\L}^{4}}^{2} \|\u(t)\|_{\V}\\&\leq  \|\f_k-\f\|_{\L^2}\|g\|_{\mathrm{L}^{\infty}}\|(\u_k-\u)(t)\|_{\H} +\sqrt{2}\|(\u_k-\u)(t)\|_{\H}\|(\u_k-\u)(t)\|_{\V} \|\u(t)\|_{\V} \\&\quad\leq \frac{1}{2 \alpha}\|\f_k-\f\|_{\L^2}^2\|g\|_{\mathrm{L}^{\infty}}^2 +\frac{\alpha}{2}\|(\u_k-\u)(t)\|_{\H}^2+\frac{\mu}{2}\|(\u_k-\u)(t)\|_{\V}^2+\frac{1}{\mu} \|(\u_k-\u)(t)\|_{\H}^2\|\u(t)\|_{\V}^2,
		\end{align*} 
		where we have used H\"older's, Ladyzhenskaya's and Young's inequalities. Integrating the above inequality from $0$ to $t$, we find
		\begin{align}\label{C12}
			&\|(\u_k-\u)(t)\|_{\H}^2+\mu \int_0^t \|(\u_k-\u)(s)\|_{\V}^2\d s\nonumber\\&\qquad+\alpha\int_0^t \|(\u_k-\u)(s)\|_{\H}^2\d s +\frac{\beta}{2^{r-1}}\int_0^t\|(\u_k-\u)(s)\|_{\widetilde{\L}^{r+1}}^{r+1}\d s  \nonumber\\&\quad\leq \frac{T}{\alpha}\|\f_k-\f\|_{\L^2}^2\|g\|_0^2+\frac{2}{\mu} \int_0^t\|(\u_k-\u)(s)\|_{\H}^2 \|\u(s)\|_{\V}^2\d s,
		\end{align}
		for all $t \in [0,T]$. An application of Gronwall's inequality in \eqref{C12}, followed by taking supremum on both side over time from $0$ to $T$, we obtain
		\begin{align*}
			&\sup_{t\in[0,T]}\|(\u_k-\u)(t)\|_{\H}^2+\mu \int_0^T \|(\u_k-\u)(t)\|_{\V}^2\d  t\\&\qquad+\alpha\int_0^T \|(\u_k-\u)(t)\|_{\H}^2\d t+\frac{\beta}{2^{r-1}}\int_0^T\|(\u_k-\u)(t)\|_{\widetilde{\L}^{r+1}}^{r+1}\d t\\&\quad\leq \left(\frac{T}{ \alpha}\|\f_k-\f\|_{\L^2}^2\|g\|_0^2\right) \exp \left(\frac{2}{\mu}\int_0^T\|\u(t)\|_{\V}^2\d t \right), 
		\end{align*}
		and \eqref{cg} follows because the sequence $\{\f_k\}_{k=1}^{\infty}$ tends to $\f$ in the $\L^2$-norm.\
		\vskip 0.2 cm 
		\noindent\textbf{Step II:} \emph{$(\u_k)_t\to\u_t$ in $\mathrm{L}^2(0,T;\H)$.}
		From \eqref{F_k}, we rewrite $\boldsymbol{F}_k$ as
		\begin{align*}
			\boldsymbol{F}_k=(\f_k-\f)g-((\u_k-\u) \cdot \nabla)\u-(\u_k \cdot \nabla)(\u_k-\u)-\beta(|\u_k|^{r-1}\u_k-|\u|^{r-1}\u).
		\end{align*}
		Note that 
		\begin{align}\label{FC}
			\|\boldsymbol{F}_k\|_{\mathrm{L}^{2}(0,T;\H)}&\leq\|\f_k-\f\|_{\L^2}\|g\|_0+\|((\u_k-\u) \cdot \nabla)\u\|_{\mathrm{L}^{2}(0,T;\H)}\nonumber\\&\quad+\|(\u_k \cdot \nabla)(\u_k-\u)\|_{\mathrm{L}^{2}(0,T;\H)}+\beta\||\u_k|^{r-1}\u_k-|\u|^{r-1}\u\|_{\mathrm{L}^{2}(0,T;\H)}.
		\end{align}
		Next, we show that $\|\boldsymbol{F}_k\|_{\mathrm{L}^{2}(0,T;\H)} \rightarrow 0 \ \  \text{as} \ \ k \rightarrow \infty$. For this, we need  to show that each individual term in the right hand side of \eqref{FC} tends to zero as $k \rightarrow \infty$.
		It can be easily seen that
		\begin{align*}
			\|\f_k-\f\|_{\L^2}\|g\|_0 \rightarrow 0 \ \  \text{as} \ \ k \rightarrow \infty.
		\end{align*}
		Using H\"older's inequality, we get 
		\begin{align}\label{C13}
			&\|(\u_k \cdot \nabla)(\u_k-\u)\|_{\mathrm{L}^{2}(0,T;\H)}^2\nonumber=\int_0^T\int_{\Omega}|\u_k|^2|\nabla (\u_k-\u)||\nabla (\u_k-\u)| \d x \d t \nonumber\\&\quad\leq \left(\int_0^T\int_{\Omega}|\nabla (\u_k-\u)|^2 \d x \d t\right)^\frac{1}{2}\left(\int_0^T\int_{\Omega}|\u_k|^4|\nabla (\u_k-\u)|^2 \d x \d t \right)^\frac{1}{2}
			\nonumber\\&\quad\leq \|\u_k-\u\|_{\mathrm{L}^{2}(0,T;\V)}\left(\int_0^T\|\u_k(t)\|_{\widetilde{\L}^{\infty}}^4\| (\u_k-\u)(t)\|_{\V}^2 \d t \right)^\frac{1}{2}. 
		\end{align}
		For $d=3$, by using Agmon's inequality and energy estimates, it is easy to deduce that
		\begin{align*}
			&\left(\int_0^T\|\u_k(t)\|_{\widetilde{\L}^{\infty}}^4\| (\u_k-\u)(t)\|_{\V}^2 \d t \right)^\frac{1}{2} \\&\quad\leq C \sup_{t\in[0,T]}\left(\|\u_k(t)\|_{\V} +\|\u(t)\|_{\V}\right)\left(\int_0^T\|\u_k(t)\|_{\V}^2\| \u_k(t)\|_{\H^2(\Omega) \cap \V}^2 \d t \right)^\frac{1}{2} \\&\quad\leq C T^\frac{1}{2} \sup_{t\in[0,T]}\|\u_k(t)\|_{\V} \sup_{t\in[0,T]}\|\u_k(t)\|_{\H^2(\Omega) \cap \V} < +\infty.
		\end{align*}
		Similarly, for $d=2$, we have
		\begin{align*}
			&\left(\int_0^T\|\u_k(t)\|_{\widetilde{\L}^{\infty}}^4\| (\u_k-\u)(t)\|_{\H}^2 \d t \right)^\frac{1}{2} \\&\quad\leq C \sup_{t\in[0,T]}\left(\|\u_k(t)\|_{\V} +\|\u(t)\|_{\V}\right)\left(\int_0^T\|\u_k(t)\|_{\H}^2\| \u_k(t)\|_{\H^2(\Omega) \cap \V}^2 \d t \right)^\frac{1}{2} \\&\quad\leq C T^\frac{1}{2} \sup_{t\in[0,T]}\|\u_k(t)\|_{\H} \sup_{t\in[0,T]}\|\u_k(t)\|_{\H^2(\Omega) \cap \V} < +\infty.
		\end{align*}
		The above estimates and the convergence of the norm $\|\u_k-\u\|_{\mathrm{L}^{2}(0,T;\V)} \rightarrow 0 \ \ \text{as} \ k \rightarrow \infty$ imply that the norm $\|(\u_k \cdot \nabla)(\u_k-\u)\|_{\mathrm{L}^{2}(0,T;\H)} \rightarrow 0 \ \ \text{as} \ k \rightarrow \infty.$ Once again an application of H\"older's inequality yields 
		\begin{align}\label{C14}
			&\|((\u_k-\u) \cdot \nabla)\u\|_{\mathrm{L}^{2}(0,T;\H)}^2=\int_0^T\int_{\Omega}|\nabla \u|^2|\u_k-\u|| \u_k-\u| \d x \d t \nonumber\\&\qquad\leq \left(\int_0^T\int_{\Omega}| \u_k-\u|^2 \d x \d t\right)^\frac{1}{2}\left(\int_0^T\int_{\Omega}|\nabla\u|^4|\u_k-\u|^2 \d x \d t \right)^\frac{1}{2}
			\nonumber\\&\qquad\leq T^\frac{1}{2}\sup_{t\in[0,T]}\|(\u_k-\u)(t)\|_{\H} \left(\int_0^T\|(\u_k-\u)(t)\|_{\widetilde{\L}^{\infty}}^2\| \u(t)\|_{\V}^4 \d t \right)^\frac{1}{2}. 
		\end{align}
		An application of Agmon's inequality and the energy estimates given in Lemmas \ref{lemma1}-\ref{lem2.5}  imply that $\int_0^T\|(\u_k-\u)(t)\|_{\widetilde{\L}^{\infty}}^2\| \u(t)\|_{\V}^4 \d t<+\infty$, and  the convergence of the norm $\|\u_k-\u\|_{\mathrm{L}^{\infty}(0,T;\H)} \rightarrow 0 \ \ \text{as} \ \ k \rightarrow \infty$ easily gives $\|((\u_k-\u) \cdot \nabla)\u\|_{\mathrm{L}^{2}(0,T;\H)} \rightarrow 0 \ \ \text{as} \ \ k \rightarrow \infty.$
		Let us define $h(\u):=|\u|^{r-1}\u$. Then,	using Taylor's formula (Theorem 7.9.1, \cite{PGC}), we obtain (cf. \cite{MTM6})
		\begin{align}\label{C141}
			&\||\u_k|^{r-1}\u_k-|\u|^{r-1}\u\|_{\mathrm{L}^{2}(0,T;\H)}^2=\int_0^T\||\u_k(t)|^{r-1}\u_k(t)-|\u(t)|^{r-1}\u(t)\|_{\H}^2 \d t\nonumber\\&\qquad=\int_0^T\left\|\int_0^1 h'(\theta \u_k+(1-\theta)\u)\d \theta(\u_k(t)-\u(t))\right\|_{\H}^2 \d t 
			\nonumber\\&\qquad\leq r^2\int_0^T\left(\|\u_k(t)\|_{\widetilde{\L}^\infty}+\|\u(t)\|_{\widetilde{\L}^\infty}\right)^{2(r-1)}\|\u_k(t)-\u(t)\|_{\H}^2\d t \nonumber\\&\qquad\leq r^2 \sup_{t\in[0,T]}\|\u_k(t)-\u(t)\|_{\H}^2\int_0^T\left(\|\u_k(t)\|_{\widetilde{\L}^\infty}+\|\u(t)\|_{\widetilde{\L}^\infty}\right)^{2(r-1)}\d t,
		\end{align}
		for $r \geq 1$. 
		Once again applying   Agmon's inequality and the energy estimates in Lemmas \ref{lemma1}-\ref{lem2.5}, we get  $\int_0^T\left(\|\u_k(t)\|_{\widetilde{\L}^\infty}+\|\u(t)\|_{\widetilde{\L}^\infty}\right)^{2(r-1)}\d t<+\infty$  and then using the convergence of the norm $\|\u_k-\u\|_{\mathrm{L}^\infty(0,T;\H)} \rightarrow 0 \ \ \text{as} \ \ k \rightarrow \infty$, we deduce $$\||\u_k|^{r-1}\u_k-|\u|^{r-1}\u\|_{\mathrm{L}^{2}(0,T;\H)} \rightarrow 0 \ \ \text{as} \ k \rightarrow \infty.$$
		Hence, one can easily conclude that
		\begin{align}\label{F0}
			\|\boldsymbol{F}_k\|_{\mathrm{L}^{2}(0,T;\H)} \rightarrow 0 \ \  \text{as} \ \ k \rightarrow \infty.
		\end{align}
		Taking the inner product with $(\u_k-\u)_t(\cdot)$ to the first equation in \eqref{C1}, we find
		\begin{align*}
			&\|(\u_k-\u)_t\|_{\H}^2+\frac{\mu}{2}\frac{\d}{\d t}\|(\u_k-\u)(t)\|_{\V}^2+\frac{\alpha}{2}\frac{\d}{\d t}\|(\u_k-\u)(t)\|_{\H}^2 \\&\quad\leq \|\boldsymbol{F}_k(t)\|_{\H}\|(\u_k-\u)_t(t)\|_{\H}
			\\&\quad\leq \frac{1}{2} \|\boldsymbol{F}_k(t)\|_{\H}^2+\frac{1}{2}\|(\u_k-\u)_t(t)\|_{\H}^2,
		\end{align*}
		where we have used the Cauchy-Schwarz and Young's inequalities. Integrating the above inequality from $0$ to $t$, we obtain
		\begin{align*}
			\|(\u_k-\u)_t\|_{\mathrm{L}^{2}(0,T;\H)} \leq  \|\boldsymbol{F}_k\|_{\mathrm{L}^{2}(0,T;\H)},
		\end{align*}
		since the initial condition is homogeneous. Whence by virtue of \eqref{F0}, we obtain
		\begin{align}\label{C16}
			\|(\u_k-\u)_t\|_{\mathrm{L}^{2}(0,T;\H)} \rightarrow 0 \ \  \text{as} \ \ k \rightarrow \infty.
		\end{align}
		\vskip 0.2cm
		\noindent\textbf{Step III:} \emph{The operator $\mathcal{A}$ is continuous on $\mathcal{D}$.}
		Differentiating $\boldsymbol{F}_{k}$ (see \eqref{F_k}) with respect to time $t$, we get
		\begin{align*}
			\boldsymbol{F}_{kt}&=(\f_k-\f)g_t-(\u_{kt} \cdot \nabla)\u_k-(\u_k \cdot \nabla)\u_{kt}+(\u_t\cdot \nabla)\u\\&\quad+(\u\cdot \nabla)\u_t-\beta r(|\u_k|^{r-1}\u_{kt}-|\u|^{r-1}\u_t).
		\end{align*}
		\begin{align}\label{FB}
			\|\boldsymbol{F}_{kt}\|_{\mathrm{L}^{2}(0,T;\H)} \nonumber&\leq \|(\f_k-\f)g_t\|_{\mathrm{L}^{2}(0,T;\H)}+\|(\u_{kt} \cdot \nabla)\u_k\|_{\mathrm{L}^{2}(0,T;\H)}+\|(\u_k \cdot \nabla)\u_{kt}\|_{\mathrm{L}^{2}(0,T;\H)}\nonumber\\&\quad+\|(\u_t\cdot \nabla)\u\|_{\mathrm{L}^{2}(0,T;\H)}+\|(\u\cdot \nabla)\u_t\|_{\mathrm{L}^{2}(0,T;\H)}\nonumber\\&\quad+\beta r\||\u_k|^{r-1}\u_{kt}\|_{\mathrm{L}^{2}(0,T;\H)}+\beta r\||\u|^{r-1}\u_t\|_{\mathrm{L}^{2}(0,T;\H)}.
		\end{align}
		Next, we show that the norm $\|\boldsymbol{F}_{kt}\|_{\mathrm{L}^{2}(0,T;\H)}$ is bounded as $ k \rightarrow \infty$. In order to do this, we bound each individual term in the right hand side of \eqref{FB} as $ k \rightarrow \infty$.	It is easy to see that
		\begin{align}\label{C17}
			\|(\f_k-\f)g_t\|_{\mathrm{L}^{2}(0,T;\H)}\leq \|\f_k-\f\|_{\L^2}\|g_t\|_0 \rightarrow 0 \ \  \text{as} \ \ k \rightarrow \infty.
		\end{align}
		For $d=3$, applying H\"older's, Gagliardo-Nirenberg's and Agmon's inequalities, and then using energy estimates, we arrive at
		\begin{align*}
			&\|(\u_t\cdot \nabla)\u\|_{\mathrm{L}^{2}(0,T;\H)}^2+\|(\u\cdot \nabla)\u_t\|_{\mathrm{L}^{2}(0,T;\H)}^2\\&\quad\leq \int_0^T\|\u_t(t)\|_{\widetilde{\L}^4}^2\|\nabla\u(t)\|_{\widetilde{\L}^4}^2 \d t+\int_0^T\|\u(t)\|_{\widetilde{\L}^\infty}^2\|\nabla\u_t(t)\|_{\H}^2 \d t \\&\quad\leq C \int_0^T\|\u_t(t)\|_{\H}^\frac{1}{2}\|\u_t(t)\|_{\V}^\frac{3}{2}\|\u(t)\|_{\H}^\frac{1}{4}\|\u(t)\|_{\H^2(\Omega) \cap \V}^\frac{7}{4} \d t\\&\qquad+C\int_0^T\|\u(t)\|_{\V}\|\u(t)\|_{\H^2(\Omega) \cap \V}\|\u_t(t)\|_{\V}^2 \d t \\&\quad\leq C \sup_{t\in[0,T]}\|\u_t(t)\|_{\H}^\frac{1}{2}\sup_{t\in[0,T]}\|\u(t)\|_{\H}^\frac{1}{4}\sup_{t\in[0,T]}\|\u(t)\|_{\H^2(\Omega) \cap \V}^\frac{7}{4} \left(\int_0^T \|\u_t(t)\|_{\V}^\frac{3}{2}\d t\right)\\&\qquad+C\sup_{t\in[0,T]}\|\u(t)\|_{\V}\sup_{t\in[0,T]}\|\u(t)\|_{\H^2(\Omega) \cap \V}\left(\int_0^T\|\u_t(t)\|_{\V}^2 \d t\right) \\&\quad\leq C T^\frac{1}{4} \left(\int_0^T \|\u_t(t)\|_{\V}^{2}\d t\right)^\frac{3}{4}+C\left(\int_0^T\|\u_t(t)\|_{\V}^2 \d t\right)< +\infty. 
		\end{align*}
		Similarly, for $d=2$, we have
		\begin{align*}
			&\|(\u_t\cdot \nabla)\u\|_{\mathrm{L}^{2}(0,T;\H)}^2+\|(\u\cdot \nabla)\u_t\|_{\mathrm{L}^{2}(0,T;\H)}^2\\&\quad\leq \int_0^T\|\u_t(t)\|_{\widetilde{\L}^4}^2\|\nabla\u(t)\|_{\widetilde{\L}^4}^2 \d t+\int_0^T\|\u(t)\|_{\widetilde{\L}^\infty}^2\|\nabla\u_t(t)\|_{\H}^2 \d t \\&\quad\leq C \int_0^T\|\u_t(t)\|_{\H}\|\u_t(t)\|_{\V}\|\u(t)\|_{\H}^\frac{1}{2}\|\u(t)\|_{\H^2(\Omega) \cap \V}^\frac{3}{2} \d t\\&\qquad+C\int_0^T\|\u(t)\|_{\H}\|\u(t)\|_{\H^2(\Omega) \cap \V}\|\u_t(t)\|_{\V}^2 \d t \\&\quad\leq C \sup_{t\in[0,T]}\|\u_t(t)\|_{\H}\sup_{t\in[0,T]}\|\u(t)\|_{\H}^\frac{1}{2}\sup_{t\in[0,T]}\|\u(t)\|_{\H^2(\Omega) \cap \V}^\frac{3}{2} \left(\int_0^T \|\u_t(t)\|_{\V}\d t\right)\\&\qquad+C\sup_{t\in[0,T]}\|\u(t)\|_{\H}\sup_{t\in[0,T]}\|\u(t)\|_{\H^2(\Omega) \cap \V}\left(\int_0^T\|\u_t(t)\|_{\V}^2 \d t\right) \\&\quad\leq C T^\frac{1}{2}\left(\int_0^T \|\u_t(t)\|_{\V}^{2}\d t\right)^\frac{1}{2}+C\left(\int_0^T\|\u_t(t)\|_{\V}^2 \d t\right)< +\infty. 
		\end{align*}
		Hence, we have 
		\begin{align}\label{C18}
			\|(\u_t\cdot \nabla)\u\|_{\mathrm{L}^{2}(0,T;\H)}+\|(\u\cdot \nabla)\u_t\|_{\mathrm{L}^{2}(0,T;\H)} < +\infty.
		\end{align}
		Using similar arguments, it can be easily deduced that
		\begin{align}\label{C19}
			\|(\u_{kt}\cdot \nabla)\u_k\|_{\mathrm{L}^{2}(0,T;\H)}+\|(\u_k\cdot \nabla)\u_{kt}\|_{\mathrm{L}^{2}(0,T;\H)} < +\infty.
		\end{align}
		Now, we estimate the norm $\||\u|^{r-1}\u_t\|_{\mathrm{L}^{2}(0,T;\H)}$ as
		\begin{align}\label{C20}
			&\||\u|^{r-1}\u_t\|_{\mathrm{L}^{2}(0,T;\H)}^2=\int_0^T\int_{\Omega}|\u|^{2(r-1)}|\u_t|^2\d x \d t \nonumber\\&\quad\leq \int_0^T\|\u_t(t)\|_{\H}^2 \|\u(t)\|_{\widetilde{\L}^\infty}^{2(r-1)} \d t \leq \sup_{t\in[0,T]}\|\u_t(t)\|_{\H}^2 \int_0^T \|\u(t)\|_{\widetilde{\L}^\infty}^{2(r-1)} \d t.
		\end{align}
		An application of Agmon's inequality and the energy estimates given in Lemmas \ref{lemma1}-\ref{lem2.5}, imply that  $\int_0^T\|\u(t)\|_{\widetilde{\L}^\infty}^{2(r-1)}\d t<+\infty$. Hence, from \eqref{C20}, we get 
		\begin{align}\label{C21}
			\||\u|^{r-1}\u_t\|_{\mathrm{L}^{2}(0,T;\H)} <+\infty.
		\end{align}
		Using similar arguments, we obtain
		\begin{align}\label{C22}
			\||\u_k|^{r-1}\u_{kt}\|_{\mathrm{L}^{2}(0,T;\H)} <+\infty.
		\end{align}
		Thus, we infer that the norm $\|\boldsymbol{F}_{kt}\|_{\mathrm{L}^{2}(0,T;\H)}$ is bounded as $k \rightarrow \infty$.
		
		Differentiating \eqref{C1} with respect to $t$ and then taking the inner product with $(\u_k-\u)_t(\cdot)$, we find
		\begin{align*}
			\frac{1}{2}\frac{\d}{\d t}\|(\u_k-\u)_t(t)\|_{\H}^2+\mu\|(\u_k-\u)_t(t)\|_{\V}^2+\alpha\|(\u_k-\u)_t(t)\|_{\H}^2=\left(\boldsymbol{F}_{kt}(t),(\u_k-\u)_t(t)\right).
		\end{align*}
		An application of H\"older's inequality in the above equation gives
		\begin{align*}
			\frac{1}{2}\int_{\epsilon}^T\frac{\d}{\d t}\|(\u_k-\u)_t(t)\|_{\H}^2 \d t \leq\int_{\epsilon}^T\|\boldsymbol{F}_{kt}(t)\|_{\H}\|(\u_k-\u)_t(t)\|_{\H}\d t, \ \ \ 0<\epsilon\leq T,
		\end{align*} 
		yielding
		\begin{align*}
			\frac{1}{2}\|(\u_k-\u)_t(\cdot,T)\|_{\H}^2  \leq \frac{1}{2}\|(\u_k-\u)_t(\cdot,\epsilon)\|_{\H}^2+\|\boldsymbol{F}_{kt}\|_{\mathrm{L}^{2}(0,T;\H)}\|(\u_k-\u)_t\|_{\mathrm{L}^{2}(0,T;\H)}.
		\end{align*}
		Since $\|(\u_k-\u)_t(\cdot,t)\|_{\H}$ is continuous on the segment [0,T], we are able to pass the limit $\epsilon \rightarrow 0$ in the above inequality, and, as a final result, we get the estimate
		\begin{align}\label{ad}
			\|(\u_k-\u)_t(\cdot,T)\|_{\H}^2  \leq \|\boldsymbol{F}_{kt}\|_{\mathrm{L}^{2}(0,T;\H)}\|(\u_k-\u)_t\|_{\mathrm{L}^{2}(0,T;\H)},
		\end{align}
		since the initial condition is homogeneous. It was proved earlier that the norm 
		$\|\boldsymbol{F}_{kt}\|_{\mathrm{L}^{2}(0,T;\H)}$ is bounded as $k \rightarrow \infty$. Consequently, using the convergence of the norm $\|(\u_k-\u)_t\|_{\mathrm{L}^{2}(0,T;\H)} \rightarrow 0$ as $k \rightarrow \infty$ (see \eqref{C16}) and the definition of the operator $\mathcal{A}$ in \eqref{ad}, we establish
		\begin{align*}
			\|\mathcal{A}\f_k-\mathcal{A}\f\|_{\H}=\|(\u_k-\u)_t(\cdot,T)\|_{\H} \rightarrow 0 \ \ \text{as} \ \ k \rightarrow \infty,
		\end{align*}
		and so the nonlinear operator $\mathcal{A}$ is continuous on $\mathcal{D}$.\
		\vskip 0.2cm
		\noindent\textbf{Step IV:} \emph{The operator $\mathcal{A}$ is completely continuous.} 
		Let us now show that any bounded subset of the set $\mathcal{D}$ is carried by the operator $\mathcal{A}$ into a compact set in the space $\L^2(\Omega)$. We shall write system \eqref{1a}-\eqref{1d} in the form
		\begin{align}
			&\u_t-\mu \Delta\u=-\nabla p+\boldsymbol{F}_1, \  \ \ \  \ \  \nabla\cdot\u=0, \ \ \  \ \text{ in } \ \Omega\times(0,T),\label{C23} \\
			&\u=\u_0, \ \ \ \ \text{ in } \ \Omega \times \{0\}, \ \ \ \ \ \ \ \  \ \  \u=\boldsymbol{0},\ \   \text{ on } \ \partial\Omega\times[0,T), \label{C24} 
		\end{align}
		where $$\boldsymbol{F}_1=\f g-(\u\cdot\nabla)\u-\alpha \u-\beta|\u|^{r-1}\u.$$ Differentiating the above system \eqref{C23}-\eqref{C24} with respect to time $t$, we find
		\begin{align}
			&\u_{tt}-\mu \Delta\u_t=-\nabla p_t+\boldsymbol{F}_{1t}, \  \ \  \ \  \nabla\cdot\u_t=0, \ \ \ \  \text{ in } \ \Omega\times(0,T),\label{C25} \\
			&\u_t=\boldsymbol{a}_0, \ \ \ \ \text{ in } \ \Omega \times \{0\},  \ \ \ \ \ \ \ \ \ \ \ \ \u_t=\boldsymbol{0},\ \   \text{ on } \ \partial\Omega\times[0,T), \label{C26} 
		\end{align}
		where $$\boldsymbol{F}_{1t}=\f g_t-\partial_{t}((\u\cdot\nabla)\u+\alpha \u+\beta|\u|^{r-1}\u) \ \ \text{and} \ \ \boldsymbol{a}_0=\P_{\H}(\mu \Delta \u_0+\boldsymbol{F}_1(x,0)).$$ As stated above, the norm $\|\boldsymbol{F}_{kt}\|_{\mathrm{L}^{2}(0,T;\H)}$ is bounded as $k \rightarrow \infty$. Using  similar arguments, we conclude that the norm $\|\boldsymbol{F}_{1t}\|_{\mathrm{L}^{2}(0,T;\H)}$ is also bounded and we have the following estimate:
		\begin{align*}
			\|\boldsymbol{F}_{1t}\|_{\mathrm{L}^{2}(0,T;\H)}^2\leq C(\|\u_0\|_{\H^2(\Omega)\cap \V}^2+ \|\f\|_{\L^2}^2\|g_t\|_0^2+\|\f\|_{\L^2}^2\|g\|_0^2).
		\end{align*}
		Let $\epsilon \in (0,T)$ be an arbitrary fixed number. As long as $\|\u_t(\cdot,t)\|_{\V}$ is continuous on the segment $[\epsilon,T]$ (see Lemma \ref{lemma2.6}), there is a $\tau \in [\epsilon,T]$ such that
		\begin{align}\label{C27}
			\int_{\epsilon}^T\|\u_t(\cdot,t)\|_{\V}^2\d t=(T-\epsilon)\|\u_t(\cdot,\tau)\|_{\V}^2.
		\end{align}
		From the system \eqref{C25}-\eqref{C26}, we deduce that
		\begin{align}\label{C28}
			\int_{\tau}^{T}\int_{\Omega}|\u_{tt}-\mu \P_{\H} \Delta \u_t|^2\d x \d t=\int_{\tau}^{T}\int_{\Omega}| \P_{\H} \boldsymbol{F}_{1t}|^2\d x \d t.
		\end{align}
		Note that 
		\begin{align*}
			2\int_{\tau}^{T}\int_{\Omega}\u_{tt} \cdot \mu \P_{\H} \Delta \u_t\d x \d t=-\mu \int_{\tau}^{T}\frac{\d}{\d t}\|\u_t(\cdot,t)\|_{\V}^2\d t,
		\end{align*}
		Using the above expression	 in \eqref{C28} results in
		\begin{align*}
			\int_{\tau}^{T}\int_{\Omega}\left(|\u_{tt}|^2+|\mu \P_{\H} \Delta \u_t|^2\right)\d x \d t=\int_{\tau}^{T}\int_{\Omega}| \P_{\H} \boldsymbol{F}_{1t}|^2\d x \d t-\mu\|\u_t(\cdot,T)\|_{\V}^2+\mu\|\u_t(\cdot,\tau)\|_{\V}^2,
		\end{align*}
		which implies that 
		\begin{align*}
			\mu\|\u_t(\cdot,T)\|_{\V}^2 \leq\|\boldsymbol{F}_{1t}\|_{\mathrm{L}^{2}(0,T;\H)}^2+\frac{\mu}{T-\epsilon}\|\u_t(\cdot,t)\|_{\mathrm{L}^{2}(0,T;\V)}^2,
		\end{align*}
		where we have used \eqref{C27}. Using the estimate \eqref{E12}, it can be easily seen that
		\begin{align*}
			\|\u_t(\cdot,T)\|_{\V}^2 &\leq \frac{1}{\mu}\|\boldsymbol{F}_{1t}\|_{\mathrm{L}^{2}(0,T;\H)}^2+\frac{C}{T-\epsilon}\left(\|\u_t(0)\|_{\H}^2+\|\f\|_{\L^2}^2\|g_t\|_0^2\right)\\&\leq\frac{C}{T-\epsilon}\left(\| \mu \Delta\u_0-(\u_0 \cdot \nabla)\u_0-\alpha \u_0-\beta|\u_0|^{r-1}\u_0\|_{\H}^2+\sup_{x \in \Omega}|g(x,0)|^2\|\boldsymbol{f}\|_{\L^2}^2\right)\\&\quad+\frac{1}{\mu}\|\boldsymbol{F}_{1t}\|_{\mathrm{L}^{2}(0,T;\H)}^2+\frac{C}{(T-\epsilon)}\|\f\|_{\L^2}^2\|g_t\|_0^2.
		\end{align*}
		As stated above, the norm $\|\boldsymbol{F}_{1t}\|_{\mathrm{L}^{2}(0,T;\H)}$ is bounded. Finally, we obtain the following estimate for the nonlinear operator $\mathcal{A}$
		\begin{align}\label{C30}
			\|\mathcal{A}\f\|_{\V}^2 &\leq\frac{C}{T-\epsilon}\left(\| \mu \Delta\u_0-(\u_0 \cdot \nabla)\u_0-\alpha \u_0-\beta|\u_0|^{r-1}\u_0\|_{\H}^2+\sup_{x \in \Omega}|g(x,0)|^2\|\boldsymbol{f}\|_{\L^2}^2\right)\nonumber\\&\quad+C(\|\u_0\|_{\H^2(\Omega)\cap \V}^2+ \|\f\|_{\L^2}^2\|g_t\|_0^2+\|\f\|_{\L^2}^2\|g\|_0^2)+\frac{C}{(T-\epsilon)}\|\f\|_{\L^2}^2\|g_t\|_0^2. 
		\end{align}
		
		Let $\mathcal{D}_1$ be an arbitrary bounded subset of the set $\mathcal{D}$. Remember that $$\mathcal{D}=\{\boldsymbol{f} \in \L^2(\Omega) \ : \ \|\f\|_{\L^2} \leq 1\}$$ is referred to as the range of the operator $\mathcal{A}$. From estimate \eqref{C30}, it follows that the nonlinear operator $\mathcal{A}$, takes $\L^2(\Omega)$ into $\V$ and maps $\mathcal{D}_1$ into a certain set  $\widetilde{\mathcal{D}}_1$, bounded in the space $\V$. On account of  Rellich's compactness theorem, the set $\widetilde{\mathcal{D}}_1$ is compact in the space $\L^2(\Omega)$, and therefore the operator $\mathcal{A}$ is continuous on $\mathcal{D}$ and any bounded 
		subset of $\mathcal{D}$ is mapped into a compact set in $\L^2(\Omega)$. Therefore, it is clear that $\mathcal{A}$ is completely continuous on $\mathcal{D}$. It is easy to verify that the operator $\mathcal{B}$, defined by relation \eqref{1i}, is also completely 
		continuous on $\mathcal{D}$, since $\mathcal{B}$ is  the composition of a nonlinear completely continuous operator and a linear bounded operator. So, the nonlinear operator equation of second kind \eqref{1j} has a solution. Hence, the inverse problem \eqref{1a}-\eqref{1e} has a solution. 
	\end{proof}
	\subsection{Uniqueness and stability}
	We have proved the existence of solution $\{\u,\nabla p,\f\}$ to the inverse problem \eqref{1a}-\eqref{1e} in the previous subsection. Now, in order to get a result on the uniqueness and stability, we first provide some supporting Lemmas. Let $(\u_i,\nabla p_i,\f_i)$ $(i=1,2)$ be the solutions of the inverse problem \eqref{1a}-\eqref{1e} corresponding to the given data $(\u_{0i},\boldsymbol{\varphi}_i,\nabla\psi_i,g_i)$ $(i=1,2)$ and set 
	\begin{align*}
		&\u:=\u_1-\u_2, \ \ \ \ \ \nabla p:=\nabla(p_1-p_2), \ \  \ \ \f:=\f_1-\f_2, \ \ \ \u_0:=\u_{01}-\u_{02}, \\& \boldsymbol{\boldsymbol{\varphi}}:=\boldsymbol{\varphi}_1-\boldsymbol{\varphi}_2, \ \ \ \ \nabla\psi:=\nabla(\psi_1-\psi_2), \ \ \  g:=g_1-g_2.
	\end{align*}
	The following lemma establishes the stability of the velocity $\u$ of the solution of the inverse problem.
	\begin{lemma}
		Let	$\u_0 \in \H^2(\Omega) \cap \V$, $g, g_t \in \C(\overline\Omega\times [0,T])$ satisfy the assumptions \eqref{1f} and $\f \in \L^2(\Omega)$. Then, for $r \in [1, \infty)$, we have 
		\begin{align}\label{S1}
			&\sup_{t\in[0,T]}\|\u(t)\|_{\H}^2+\mu \int_0^T\|\u(t)\|_{\V}^2\d t +\frac{\beta}{2^{r-1}}\int_0^T\|\u(t)\|_{\widetilde{\L}^{r+1}}^{r+1}\d t \nonumber\\&\quad\leq C \left(\|\u_0\|_{\H}^2+ \|\f\|_{\L^2}^2+\|g\|_0^2\right).
		\end{align} 
	\end{lemma}
	\begin{proof}
		Subtracting the equations for $\{\u_i,\nabla p_i, \f_i\} \ (i=1,2)$, we see that
		\begin{align}\label{S2}
			&\u_t-\mu \Delta \u+(\u_1 \cdot \nabla)\u\nonumber+(\u \cdot \nabla)\u_2+\alpha \u\\&\quad+\beta \left(|\u_1|^{r-1}\u_1-|\u_2|^{r-1}\u_2\right)+\nabla p=\f g_1+\f_2 g.
		\end{align}
		Taking the inner product with $\u(\cdot)$ of the equation \eqref{S2}, we find
		\begin{align}\label{S3}
			&\frac{1}{2}\frac{\d}{\d t}\|\u(t)\|_{\H}^2+\mu \|\u(t)\|_{\V}^2+\alpha\|\u(t)\|_{\H}^2\nonumber\\&=((\f g_1(t)+\f_2 g(t),\u(t)) -((\u(t) \cdot \nabla)\u_2(t),\u(t))\nonumber\\&\quad-\beta \left(|\u_1(t)|^{r-1}\u_1(t)-|\u_2(t)|^{r-1}\u_2(t),\u(t)\right).
		\end{align}
		Using H\"older's and Young's inequalities, we estimate $|(\f g_1+\f_2 g,\u)|$ as
		\begin{align}\label{S4}
			|(\f g_1+\f_2 g,\u)|&\leq \left(\|\f\|_{\L^2}\|g_1\|_{\mathrm{L}^\infty}+\|\f_2\|_{\L^2}\|g\|_{\mathrm{L}^\infty}\right)\|\u\|_{\H}\nonumber\\&\leq\frac{1}{2 \alpha}\left(\|\f\|_{\L^2}\|g_1\|_{\mathrm{L}^\infty}+\|\f_2\|_{\L^2}\|g\|_{\mathrm{L}^\infty}\right)^2+\frac{\alpha}{2}\|\u\|_{\H}^2.
		\end{align}
		\vskip 0.2 cm
		\noindent\textbf{Case I:} \emph{$d=3$ and $r > 3$.}
		An estimate similar to \eqref{C61} gives
		\begin{align}\label{S5}
			&\beta(|\u_1|^{r-1}\u_1-|\u_2|^{r-1}\u_2,\u) +((\u \cdot \nabla)\u_2,\u) \nonumber\\&\quad\geq \frac{\beta}{2^r}\|\u\|_{\widetilde{\L}^{r+1}}^{r+1}-\frac{\mu }{2}\|\u\|_{\V}^2-\frac{r-3}{2\mu(r-1)}\left(\frac{4}{\beta\mu (r-1)}\right)^{\frac{2}{r-3}}\|\u\|_{\H}^2,
		\end{align}
		for $r>3$. Substituting \eqref{S4} and \eqref{S5} in \eqref{S3}, and then integrating from $0$ to $t$, we obtain
		\begin{align}\label{S6}
			&\|\u(t)\|_{\H}^2+\mu \int_0^t\|\u(s)\|_{\V}^2\d  s+\alpha \int_0^t\|\u(s)\|_{\H}^2\d s +\frac{\beta}{2^{r-1}}\int_0^t\|\u(s)\|_{\widetilde{\L}^{r+1}}^{r+1}\d s \nonumber\\&\quad \leq \|\u_0\|_{\H}^2+ \frac{T}{\alpha}\left(\|\f\|_{\L^2}\|g_1\|_0+\|\f_2\|_{\H}\|g\|_0\right)^2+\frac{r-3}{\mu(r-1)}\left(\frac{4}{\beta\mu (r-1)}\right)^{\frac{2}{r-3}}\int_0^t \|\u(s)\|_{\H}^2 \d s.
		\end{align} 
		An application of Gronwall's inequality in \eqref{S6}, followed by taking supremum on both sides over time from $0$ to $T$ results in 
		\begin{align*}
			&\sup_{t\in[0,T]}\|\u(t)\|_{\H}^2+\mu \int_0^T\|\u(t)\|_{\V}^2\d  t+\alpha \int_0^T\|\u(t)\|_{\H}^2\d t+\frac{\beta}{2^{r-1}}\int_0^T\|\u(t)\|_{\widetilde{\L}^{r+1}}^{r+1}\d t \\&\quad \leq \left\{\|\u_0\|_{\H}^2+ \frac{T}{\alpha}\left(\|\f\|_{\L^2}\|g_1\|_0+\|\f_2\|_{\H}\|g\|_0\right)^2\right\}
			\exp \left\{\frac{r-3}{\mu(r-1)}\left(\frac{4}{\beta\mu (r-1)}\right)^{\frac{2}{r-3}}T\right\},
		\end{align*} 
		and \eqref{S1} follows.\
		\vskip 0.2 cm
		\noindent\textbf{Case II:} \emph{$d=r=3$.}
		A calculation similar to \eqref{C11} gives
		\begin{align}\label{S7}
			&\beta(|\u_1|^{2}\u_1-|\u_2|^{2}\u_2,\u) +((\u \cdot \nabla)\u_2,\u) \geq \frac{1}{2}\left(\beta-\frac{1}{2\mu \theta}\right)\|\u\|_{\widetilde{\L}^4}^4-\mu \theta\|\u\|_{\V}^2. 
		\end{align}
		Substituting \eqref{S4} and \eqref{S7} in \eqref{S3}, we infer that 
		\begin{align*}
			&\|\u(t)\|_{\H}^2+2\mu(1-\theta)\int_0^t\|\u(s)\|_{\V}^2\d  s+\alpha\int_0^t\|\u(s)\|_{\H}^2\d s +\left(\beta-\frac{1}{2\mu \theta}\right)\int_0^t\|\u(s)\|_{\widetilde{\L}^{4}}^{4}\d s \\&\quad \leq \|\u_0\|_{\H}^2+ \frac{T}{\alpha}\left(\|\f\|_{\L^2}\|g_1\|_0+\|\f_2\|_{\H}\|g\|_0\right)^2,
		\end{align*} 
		and \eqref{S1} follows, provided $2\beta \mu \geq1$.
		\vskip 0.2 cm
		\noindent\textbf{Case III:} \emph{$d=2$ and $r \geq 1$.} 
		We note that
		\begin{align}\label{S9}
			|((\u \cdot \nabla)\u_2,\u)|\leq\|\u\|_{\widetilde{\L}^{4}}^2\|\nabla\u_2\|_{\H}\leq\sqrt{2}\|\u\|_{\H}\|\u\|_{\V}\|\u_2\|_{\V}\leq\frac{\mu}{2}\|\u\|_{\V}^2+\frac{1}{\mu}\|\u\|_{\H}^2\|\u_2\|_{\V}^2.
		\end{align}
		\begin{align}\label{S10}
			\beta(|\u_1|^{r-1}\u_1-|\u_2|^{r-1}\u_2,\u) \geq \frac{\beta}{2}\||\u_1|^{\frac{r-1}{2}}\u\|_{\H}^2+\frac{\beta}{2}\||\u_2|^{\frac{r-1}{2}}\u\|_{\H}^2 \geq \frac{\beta}{2^{r}}\|\u\|_{\widetilde{\L}^{r+1}}^{r+1}.
		\end{align}
		Substituting \eqref{S4}, \eqref{S9} and \eqref{S10} in \eqref{S3}, and then integrating from $0$ to $t$, we obtain
		\begin{align}\label{S11}
			&\|\u(t)\|_{\H}^2+\mu \int_0^t\|\u(s)\|_{\V}^2\d  s+\alpha\int_0^t\|\u(s)\|_{\H}^2\d s +\frac{\beta}{2^{r-1}}\int_0^t\|\u(s)\|_{\widetilde{\L}^{r+1}}^{r+1}\d s \nonumber\\&\quad \leq \|\u_0\|_{\H}^2+ \frac{T}{\alpha}\left(\|\f\|_{\L^2}\|g_1\|_0+\|\f_2\|_{\H}\|g\|_0\right)^2+\frac{2}{\mu}\int_0^t\|\u_2(s)\|_{\V}^2 \|\u(s)\|_{\H}^2 \d s,
		\end{align} 
		for all $t \in [0,T]$.	Applying Gronwall's inequality in \eqref{S11}, one can easily see that
		\begin{align*}
			\|\u(t)\|_{\H}^2\leq \exp \left(\frac{2}{\mu}\int_0^T\|\u_2(t)\|_{\V}^2\d t\right)\left\{\|\u_0\|_{\H}^2+ \frac{T}{\alpha}\left(\|\f\|_{\L^2}\|g_1\|_0+\|\f_2\|_{\H}\|g\|_0\right)^2\right\},
		\end{align*}
		for all $t \in [0,T]$. Thus, from \eqref{S11}, it is immediate that
		\begin{align*}
			&\sup_{t\in[0,T]}\|\u(t)\|_{\H}^2+\mu \int_0^T\|\u(t)\|_{\V}^2\d t+\alpha\int_0^T\|\u(t)\|_{\H}^2\d t +\frac{\beta}{2^{r-1}}\int_0^T\|\u(t)\|_{\widetilde{\L}^{r+1}}^{r+1}\d t \\&\quad \leq \exp \left(\frac{2}{\mu}\int_0^T\|\u_2(t)\|_{\V}^2\d t\right)\left\{\|\u_0\|_{\H}^2+ \frac{T}{\alpha}\left(\|\f\|_{\L^2}\|g_1\|_0+\|\f_2\|_{\H}\|g\|_0\right)^2\right\},
		\end{align*} 
		and \eqref{S1} follows.
	\end{proof}
	\begin{lemma}
		Let	$\u_0 \in \H^2(\Omega) \cap \V$, $g, g_t \in \C(\overline\Omega\times [0,T])$ satisfy the assumptions \eqref{1f} and $\f \in \L^2(\Omega)$. Then, we have
		\begin{align}\label{S111}
			\sup_{t\in[0,T]}\|\u(t)\|_{\V}^2+\int_0^T\|\u_t(t)\|_{\H}^2\d t \leq C \bigg\{ \|\u_0\|_{\V}^2+\|\f\|_{\L^2}^2+\|g\|_0^2\bigg\}.
		\end{align}
	\end{lemma}
	\begin{proof}
		Taking the inner product with $\u_t(\cdot)$ of the equation \eqref{S2}, we find
		\begin{align}\label{S12}
			&\|\u_t(t)\|_{\H}^2+ \frac{\mu}{2}\frac{\d}{\d t} \|\u(t)\|_{\V}^2+\frac{\alpha}{2}\frac{\d}{\d t} \|\u(t)\|_{\H}^2\\&=(\f g_1(t)+\f_2 g(t),\u_t(t)) -((\u(t) \cdot \nabla)\u_2(t),\u_t(t))\nonumber\\&\quad-((\u_1(t) \cdot \nabla)\u(t),\u_t(t))-\beta (|\u_1(t)|^{r-1}\u_1(t)-|\u_2(t)|^{r-1}\u_2(t),\u_t(t) ):=\sum_{j=1}^4I_j.
		\end{align}
		Next, we estimate each $I_j$'s $(j=1,2,3,4)$ separately as follows:
		Using H\"older's and Young's inequalities, we estimate $I_1$ as
		\begin{align}\label{S13}
			I_1&\leq \left(\|\f\|_{\L^2}\|g_1\|_{\mathrm{L}^\infty}+\|\f_2\|_{\L^2}\|g\|_{\mathrm{L}^\infty}\right)\|\u_t\|_{\H}\nonumber\\&\leq2\left(\|\f\|_{\L^2}\|g_1\|_{\mathrm{L}^\infty}+\|\f_2\|_{\L^2}\|g\|_{\mathrm{L}^\infty}\right)^2+\frac{1}{8}\|\u_t\|_{\H}^2.
		\end{align}
		\vskip 0.2 cm
		\noindent\textbf{Case I:} \emph{$d=2$ and $r \geq 1$.} Applying H\"older's, Poincar\'e's and Young's inequalities, we estimate $I_2$ and $I_3$ as
		\begin{align}
			I_2 &\leq \|\u\|_{\widetilde{\L}^{4}} \|\nabla\u_2\|_{\widetilde{\L}^{4}}\|\u_t\|_{\H}\leq C\|\u\|_{\H}^{\frac{1}{2}}\|\u\|_{\V}^{\frac{1}{2}}\|\u_2\|_{\H}^{\frac{1}{4}}\|\u_2\|_{\H^2(\Omega) \cap \V}^{\frac{3}{4}}\|\u_t\|_{\H}\nonumber\\&\leq C\|\u\|_{\V}\|\u_2\|_{\H^2(\Omega) \cap \V}\|\u_t\|_{\H}\leq\frac{1}{8}\|\u_t\|_{\H}^2+C\|\u\|_{\V}^2\|\u_2\|_{\H^2(\Omega) \cap \V}^2,\label{S14} \\
			I_3 &\leq \|\u_1\|_{\widetilde{\L}^{\infty}} \|\nabla\u\|_{\H}\|\u_t\|_{\H}\leq C\|\u_1\|_{\H}^{\frac{1}{2}}\|\u_1\|_{\H^2(\Omega)\cap\V}^{\frac{1}{2}}\|\u\|_{\V}\|\u_t\|_{\H}\nonumber\\&\leq \frac{1}{8}\|\u_t\|_{\H}^2+C\|\u_1\|_{\H}\|\u_1\|_{\H^2(\Omega) \cap \V}\|\u\|_{\V}^2. 	\label{S15}
		\end{align}
		Let us define $h_1(\u):=|\u|^{r-1}\u$. Then using Taylor's formula and Agmon's inequality, we obtain 
		\begin{align}\label{S16}
			I_4&\leq\beta \||\u_1|^{r-1}\u_1-|\u_2|^{r-1}\u_2\|_{\H}\|\u_t\|_{\H}\nonumber\\& \leq\beta\left\|\int_0^1h'_1(\theta\u_1+(1-\theta)\u_2)\d \theta (\u_1-\u_2)\right\|_{\H}\|\u_t\|_{\H} \nonumber\\&\leq \beta r\sup_{0<\theta<1} \|\theta\u_1+(1-\theta)\u_2\|_{\widetilde{\L}^\infty}^{r-1}\|\u_1-\u_2\|_{\H}\|\u_t\|_{\H}\nonumber\\&\leq C\left(\|\u_1\|_{\widetilde{\L}^\infty}+\|\u_2\|_{\widetilde{\L}^\infty}\right)^{r-1}\|\u\|_{\H} \|\u_t\|_{\H}\nonumber\\&\leq C\left(\|\u_1\|_{\H}^\frac{1}{2}\|\u_1\|_{\H^2(\Omega) \cap \V}^\frac{1}{2}+\|\u_2\|_{\H}^\frac{1}{2}\|\u_2\|_{\H^2(\Omega) \cap \V}^\frac{1}{2}\right)^{r-1}\|\u\|_{\H}\|\u_t\|_{\H}\nonumber\\&\leq C\left(\|\u_1\|_{\H}^\frac{1}{2}\|\u_1\|_{\H^2(\Omega) \cap \V}^\frac{1}{2}+\|\u_2\|_{\H}^\frac{1}{2}\|\u_2\|_{\H^2(\Omega) \cap \V}^\frac{1}{2}\right)^{2(r-1)}\|\u\|_{\H}^2+\frac{1}{8}\|\u_t\|_{\H}^2,
		\end{align}
		for $r\geq1$. Substituting the estimates \eqref{S13}-\eqref{S16} in \eqref{S12}, and then integrating from $0$ to $t$, we have
		\begin{align}\label{S17}
			&\int_0^t\|\u_t(s)\|_{\H}^2\d s+ \mu \|\u(t)\|_{\V}^2+\alpha \|\u(t)\|_{\H}^2 \nonumber \\&\quad\leq\mu \|\u_0\|_{\V}^2+\alpha\|\u_0\|_{\H}^2+4T\big(\|\f\|_{\L^2}\|g_1\|_0+\|\f_2\|_{\H}\|g\|_0\big)^2\nonumber\\&\qquad+C\int_0^t\|\u(s)\|_{\V}^2\|\u_2(s)\|_{\H^2(\Omega) \cap \V}^2\d s+C\int_0^t \|\u_1(s)\|_{\H}\|\u_1(s)\|_{\H^2(\Omega) \cap \V}\|\u(s)\|_{\V}^2\d s\nonumber\\&\qquad+C\int_0^t \left(\|\u_1(s)\|_{\H}^\frac{1}{2}\|\u_1(s)\|_{\H^2(\Omega) \cap \V}^\frac{1}{2}+\|\u_2(s)\|_{\H}^\frac{1}{2}\|\u_2(s)\|_{\H^2(\Omega) \cap \V}^\frac{1}{2}\right)^{2(r-1)}\|\u(s)\|_{\H}^2\d s,
		\end{align}
		for all $t \in [0,T]$. An application of Gronwall's inequality in \eqref{S17} gives
		\begin{align}\label{3.63}
			&\mu \|\u(t)\|_{\V}^2
			\nonumber	\\&\leq \bigg\{\mu \|\u_0\|_{\V}^2+\alpha \|\u_0\|_{\H}^2+4T\big(\|\f\|_{\L^2}\|g_1\|_0+\|\f_2\|_{\H}\|g\|_0\big)^2 \nonumber\\&\quad+CT\sup_{t\in[0,T]}\left(\|\u_1(t)\|_{\H}^\frac{1}{2}\|\u_1(t)\|_{\H^2(\Omega) \cap \V}^\frac{1}{2}+\|\u_2(t)\|_{\H}^\frac{1}{2}\|\u_2(t)\|_{\H^2(\Omega) \cap \V}^\frac{1}{2}\right)^{2(r-1)} \sup_{t\in[0,T]}\|\u(t)\|_{\H}^2\bigg\} \nonumber\\&\quad\times\exp\bigg\{C T\bigg(\sup_{t\in[0,T]}\|\u_2(t)\|_{\H^2(\Omega) \cap \V}^2+\sup_{t\in[0,T]}\|\u_1(t)\|_{\H}\sup_{t\in[0,T]}\|\u_1(t)\|_{\H^2(\Omega) \cap \V}\bigg)\bigg\},
		\end{align}
		for all $t \in [0,T]$. Thus, using \eqref{S1} and \eqref{3.63} in \eqref{S17}, we immediately get \eqref{S111}.
		\vskip 0.2 cm
		\noindent\textbf{Case II:} \emph{$d=3$ and $r \geq 3$.} 
		We estimate the terms $I_2$ and $I_3$ using H\"older's, Poincar\'e's and Young's inequalities as
		\begin{align}
			I_2 &\leq \|\u\|_{\widetilde{\L}^{4}} \|\nabla\u_2\|_{\widetilde{\L}^{4}}\|\u_t\|_{\H}\leq C\|\u\|_{\H}^{\frac{1}{4}}\|\u\|_{\V}^{\frac{3}{4}}\|\u_2\|_{\H}^{\frac{1}{8}}\|\u_2\|_{\H^2(\Omega) \cap \V}^{\frac{7}{8}}\|\u_t\|_{\H}\nonumber\\&\leq C\|\u\|_{\V}\|\u_2\|_{\H^2(\Omega) \cap \V}\|\u_t\|_{\H}\leq\frac{1}{8}\|\u_t\|_{\H}^2+C\|\u\|_{\V}^2\|\u_2\|_{\H^2(\Omega) \cap \V}^2,\label{S18}\\
			I_3 &\leq \|\u_1\|_{\widetilde{\L}^{\infty}} \|\nabla\u\|_{\H}\|\u_t\|_{\H}\leq C\|\u_1\|_{\V}^{\frac{1}{2}}\|\u_1\|_{\H^2(\Omega)\cap\V}^{\frac{1}{2}}\|\u\|_{\V}\|\u_t\|_{\H}\nonumber\\&\leq \frac{1}{8}\|\u_t\|_{\H}^2+C\|\u_1\|_{\V}\|\u_1\|_{\H^2(\Omega) \cap \V}\|\u\|_{\V}^2.\label{S19}
		\end{align}
		A calculation similar to\eqref{S16} yields the estimate 
		\begin{align}\label{S161}
			I_4&\leq\beta \||\u_1|^{r-1}\u_1-|\u_2|^{r-1}\u_2\|_{\H}\|\u_t\|_{\H}\nonumber\\&\leq C\left(\|\u_1\|_{\V}^\frac{1}{2}\|\u_1\|_{\H^2(\Omega) \cap \V}^\frac{1}{2}+\|\u_2\|_{\V}^\frac{1}{2}\|\u_2\|_{\H^2(\Omega) \cap \V}^\frac{1}{2}\right)^{2(r-1)}\|\u\|_{\H}^2+\frac{1}{8}\|\u_t\|_{\H}^2,
		\end{align}
		for $r\geq1$. Substituting \eqref{S13}, \eqref{S18}, \eqref{S19} and \eqref{S161} in \eqref{S12}, and then integrating from $0$ to $t$, we obtain
		\begin{align}\label{S20}
			&\int_0^t\|\u_t(s)\|_{\H}^2\d s+ \mu \|\u(t)\|_{\V}^2+\alpha\|\u(t)\|_{\H}^2 \nonumber\\&\quad\leq\mu \|\u_0\|_{\V}^2+\alpha\|\u_0\|_{\H}^2+4T\big(\|\f\|_{\L^2}\|g_1\|_0+\|\f_2\|_{\H}\|g\|_0\big)^2\nonumber\\&\qquad+C\int_0^t\|\u(s)\|_{\V}^2\|\u_2(s)\|_{\H^2(\Omega) \cap \V}^2\d s+C\int_0^t \|\u_1(s)\|_{\V}\|\u_1(s)\|_{\H^2(\Omega) \cap \V}\|\u(s)\|_{\V}^2\d s\nonumber\\&\qquad+C\int_0^t \left(\|\u_1(s)\|_{\V}^\frac{1}{2}\|\u_1(s)\|_{\H^2(\Omega) \cap \V}^\frac{1}{2}+\|\u_2(s)\|_{\V}^\frac{1}{2}\|\u_2(s)\|_{\H^2(\Omega) \cap \V}^\frac{1}{2}\right)^{2(r-1)}\|\u(s)\|_{\H}^2\d s,
		\end{align}
		for all $t \in [0,T]$. Applying Gronwall's inequality in \eqref{S20}, and then taking supremum on both side over time from $0$ to $T$, one reach at \eqref{S111}.
	\end{proof}
	The next lemma establishes the stability of the spatially varying function $\f$ of the solution of the inverse problem, which will complete the proof of part (ii) of Theorem \ref{thm2}.
	\begin{lemma}
		Let $\u_0, \boldsymbol{\varphi} \in \H^2(\Omega) \cap \V $, $\nabla \psi \in \G(\Omega)$, $g, g_t \in \C(\overline\Omega\times [0,T])$ satisfy the assumption \eqref{1f}, and $\f \in \L^2(\Omega)$. Then, we have
		\begin{align}\label{S21}
			\|\f\|_{\L^2}\leq C\left(\|\u_0\|_{\H^2(\Omega) \cap \V}+\|g\|_0+\|g_t\|_0+\|\nabla \boldsymbol{\varphi}\|_{\H}+\|\nabla \psi-\mu \Delta \boldsymbol{\varphi}\|_{\H}\right).
		\end{align}
	\end{lemma}
	\begin{proof}
		Differentiating \eqref{S2} with respect to time $t$, and then taking the inner product with $\u_t(\cdot)$, we find
		\begin{align}\label{S22}
			&\frac{1}{2}\frac{\d}{\d t}\|\u_t(t)\|_{\H}^2+ \mu \|\u_t(t)\|_{\V}^2+\alpha\|\u_t(t)\|_{\H}^2\nonumber\\&\quad=((\f g_{1t}(t)+\f_2 g_t(t),\u_t(t)) -((\u_{1t}(t) \cdot \nabla)\u(t),\u_t(t))-((\u_t(t) \cdot \nabla)\u_2(t),\u_t(t))\nonumber\\&\qquad-((\u(t) \cdot \nabla)\u_{2t}(t),\u_t(t))-\beta r \left(|\u_1(t)|^{r-1}\u_{1t}(t)-|\u_2(t)|^{r-1}\u_{2t}(t),\u_t (t)\right)\nonumber\\&\quad:=\sum_{j=1}^5E_j.
		\end{align}
		Next, we estimate each $E_j$'s $(j=1,2,3,4,5)$ separately as follows:
		Using H\"older's and Young's inequalities, we obtain
		\begin{align}\label{S23}
			E_1 &\leq \left(\|\f\|_{\L^2}\|g_{1t}\|_{\mathrm{L}^\infty}+\|\f_2\|_{\L^2}\|g_t\|_{\mathrm{L}^\infty}\right)\|\u_t\|_{\H}\nonumber\\&\leq\frac{1}{2 \alpha}\left(\|\f\|_{\L^2}\|g_{1t}\|_{\mathrm{L}^\infty}+\|\f_2\|_{\L^2}\|g_t\|_{\mathrm{L}^\infty}\right)^2+\frac{\alpha}{2}\|\u_t\|_{\V}^2.
		\end{align}
		\vskip 0.2 cm
		\noindent\textbf{Case I:} \emph{$d=2$ and $r \geq 1$.} 
		Using H\"older's, Ladyzhenskaya's, Poincar\'e's and Young's inequalities, we get
		\begin{align}
			E_2 \nonumber&\leq \|\u_{1t}\|_{\widetilde{\L}^4} \|\nabla\u\|_{\H}\|\u_t\|_{\widetilde{\L}^4}\leq\left(\frac{2}{\lambda_1}\right)^{\frac{1}{2}} \|\u_{1t}\|_{\V} \|\u\|_{\V}\|\u_t\|_{\V}\\&\leq \frac{3}{ \lambda_1 \mu} \|\u_{1t}\|_{\V}^2 \|\u\|_{\V}^2+\frac{\mu}{6}\|\u_t\|_{\V}^2,\label{S24}\\
			E_3 &\leq  \|\nabla\u_2\|_{\H}\|\u_t\|_{\widetilde{\L}^4}^2\leq \sqrt{2} \|\nabla\u_2\|_{\H}\|\u_t\|_{\H}\|\u_t\|_{\V} \leq \frac{3}{\mu} \|\nabla\u_2\|_{\H}^2\|\u_t\|_{\H}^2+\frac{\mu}{6}\|\u_t\|_{\V}^2,\label{S25}\\
			E_4& \leq \|\u\|_{\widetilde{\L}^4} \|\nabla\u_{2t}\|_{\H}\|\u_t\|_{\widetilde{\L}^4}\leq \frac{3}{ \lambda_1 \mu}\|\u\|_{\V}^2 \|\u_{2t}\|_{\V}^2+\frac{\mu}{6}\|\u_t\|_{\V}^2.	\label{S26}
		\end{align}
		Let us define $h_2(\u):=|\u|^{r-1}$. Using Taylor's formula and H\"older's, Young's, Agmon's and Gagliardo-Nirenberg's  inequalities, it can be deduced that
		\begin{align}\label{S27}
			E_5 \nonumber&=-\beta r\big(|\u_1|^{r-1}\u_{1t}-|\u_2|^{r-1}\u_{2t},\u_{1t}-\u_{2t}\big)\nonumber\\&=-\beta r\big\{(|\u_1|^{r-1}(\u_{1t}-\u_{2t}),\u_{1t}-\u_{2t})+((|\u_1|^{r-1}-|\u_2|^{r-1})\u_{2t},\u_{1t}-\u_{2t})\big\}
			\nonumber\\&\leq-\beta r((|\u_1|^{r-1}-|\u_2|^{r-1})\u_{2t},\u_{1t}-\u_{2t}) 
			\nonumber\\&=-\beta r\left(\int_0^1h'_2(\theta\u_1+(1-\theta)\u_2) (\u_1-\u_2) \d \theta \u_{2t},\u_{1t}-\u_{2t}\right)
			\nonumber\\& \leq  \beta(r-1) r
			\left(\|\u_1\|_{\widetilde{\L}^\infty}+\|\u_2\|_{\widetilde{\L}^\infty}\right)^{r-2}\|\u_1-\u_2\|_{\widetilde{\L}^6}\|\u_{2t}\|_{\widetilde{\L}^3}\|\u_{1t}-\u_{2t}\|_{\H}\nonumber\\&\leq C
			\|\u_{2t}\|_{\widetilde{\L}^3}^2\|\u_{1t}-\u_{2t}\|_{\H}^2+C\left(\|\u_1\|_{\widetilde{\L}^\infty}+\|\u_2\|_{\widetilde{\L}^\infty}\right)^{2(r-2)}\|\u_1-\u_2\|_{\widetilde{\L}^6}^2\nonumber\\&
			\leq C \|\u_{2t}\|_{\V}^2
			\|\u_{t}\|_{\H}^2+C\left(\|\u_1\|_{\H}^\frac{1}{2}\|\u_1\|_{\H^2(\Omega)\cap \V}^\frac{1}{2}+\|\u_2\|_{\H}^\frac{1}{2}\|\u_2\|_{\H^2(\Omega)\cap \V}^\frac{1}{2}\right)^{2(r-2)}\|\u\|_{\V}^2,
		\end{align}
		for $r\geq 2$. Substituting the estimates \eqref{S23}-\eqref{S27} in \eqref{S22}, and then integrating from $0$ to $t$, we obtain
		\begin{align}\label{S28}
			&\|\u_t(t) \|_{\H}^2+\mu \int_0^t \|\u_t (s)\|_{\V}^2\d s+\alpha\int_0^t \|\u_t (s)\|_{\H}^2\d s \nonumber\\&\quad\leq \|\u_t(0) \|_{\H}^2+\frac{T}{\alpha}\left(\|\f\|_{\L^2}\|g_{1t}\|_0+\|\f_2\|_{\H}\|g_t\|_0\right)^2 + \frac{6}{ \lambda_1 \mu}\int_0^t\|\u_{1t}(s)\|_{\V}^2\|\u(s)\|_{\V}^2\d s\nonumber\\&\qquad+\frac{6}{\lambda_1 \mu}\int_0^t\|\u_{2t}(s)\|_{\V}^2\|\u(s)\|_{\V}^2 \d s+\frac{6}{\mu}\int_0^t\|\nabla\u_2(s)\|_{\H}^2\|\u_t(s)\|_{\H}^2 \d s\nonumber\\&\qquad+C\int_0^t   \bigg\{\left(\|\u_1(s)\|_{\H}^\frac{1}{2}\|\u_1(s)\|_{\H^2(\Omega)\cap \V}^\frac{1}{2}+\|\u_2(s)\|_{\H}^\frac{1}{2}\|\u_2(s)\|_{\H^2(\Omega)\cap \V}^\frac{1}{2}\right)^{2(r-2)}\|\u(s)\|_{\V}^2\nonumber\\&\qquad+\|\u_{2t}(s)\|_{\V}^2
			\|\u_{t}(s)\|_{\H}^2\bigg\}\d s,
		\end{align}
		for all $t \in [0,T]$. An application of Gronwall's inequality in \eqref{S28} gives the following estimate
		\begin{align}\label{3.76}
			&\sup_{t\in[0,T]}\|\u_t(t) \|_{\H}^2+\mu \int_0^T \|\u_t (t)\|_{\V}^2\d t +\alpha\int_0^T \|\u_t (t)\|_{\H}^2\d t \nonumber\\&\leq \bigg\{\|\u_t(0) \|_{\H}^2+\frac{T}{ \alpha }\big(\|\f\|_{\L^2}\|g_{1t}\|_0+\|\f_2\|_{\H}\|g_t\|_0\big)^2 \nonumber\\&\quad+ \frac{6}{ \lambda_1 \mu}\sup_{t\in[0,T]}\|\u(t)\|_{\V}^2\int_0^T\left(\|\u_{1t}(t)\|_{\V}^2+\|\u_{2t}(t)\|_{\V}^2 \right)\d t \nonumber\\&\quad+C T \sup_{t\in[0,T]}\|\u(t)\|_{\V}^2 \sup_{t\in[0,T]}\left(\|\u_1(t)\|_{\H}^\frac{1}{2}\|\u_1(t)\|_{\H^2(\Omega)\cap \V}^\frac{1}{2}+\|\u_2(t)\|_{\H}^\frac{1}{2}\|\u_2(t)\|_{\H^2(\Omega)\cap \V}^\frac{1}{2}\right)^{2(r-2)}\bigg\}\nonumber\\&\quad \times\exp \bigg(CT\sup_{t\in[0,T]}\|\nabla\u_2(t)\|_{\H}^2+C\int_0^T\|\u_{2t}(t)\|_{\V}^2\d t\bigg),
		\end{align}
		for all $t\in[0,T]$. From \eqref{S2}, one can easily see that 
		\begin{align*}
			\u_t(0)&=\mu \Delta \u_0-(\u_{10} \cdot \nabla)\u_0\nonumber-(\u_0 \cdot \nabla)\u_{20}-\alpha \u_0-\beta \left(|\u_{10}|^{r-1}\u_{10}-|\u_{20}|^{r-1}\u_{20}\right)\nonumber\\&\quad-\nabla p_0+\f g_{1}(0)+\f_2 g(0),
		\end{align*}
		For any $\boldsymbol{\upsilon}\in\H$, we get 
		\begin{align*}
			|(\u_t(0),\boldsymbol{\upsilon})|&\leq\big(\mu\| \Delta \u_0\|_{\H}+\|\u_{10}\|_{\wi\L^{\infty}}\|\nabla\u_0\|_{\H}+\|\u_0\|_{\wi\L^{\infty}}\|\nabla\u_{20}\|_{\H}+\alpha\|u_0\|_{\H}\nonumber\\&\quad+C(\|\u_{10}\|_{\H}^{r-1}+\|\u_{20}\|_{\H}^{r-1})\|\u_0\|_{\H}+\|\f\|_{\L^2}\|g_1\|_{0}+\|\f_2\|_{\L^2}\|g\|_{0}\big)\|\boldsymbol{\upsilon}\|_{\H}\nonumber\\&\leq C\big(\|\u_0\|_{\H^2(\Omega)\cap\V}+\|\f\|_{\L^2}+\|g\|_0\big)\|\boldsymbol{\upsilon}\|_{\H},
		\end{align*}
		where we have used H\"older's and Agmon's inequalities. Since, it is true for any $\boldsymbol{\upsilon}\in\H$, we easily have $\|\u_t(0)\|_{\H}\leq C\big(\|\u_0\|_{\H^2(\Omega)\cap\V}+\|\f\|_{\L^2}+\|g\|_0\big)$. 
		Using the energy estimates given in Lemmas \ref{lemma1}-\ref{lem2.5} and \eqref{S111} in the inequality \eqref{3.76}, we obtain the following estimate:
		\begin{align*}
			\sup_{t\in[0,T]}\|\u_t(t) \|_{\H}^2 \leq C \left( \|\u_0\|_{\H^2(\Omega) \cap \V}^2+\|\f\|_{\L^2}^2+\|g\|_0^2+\|g_t\|_0^2 \right),
		\end{align*}
		and as a result, we have
		\begin{align}\label{S29}
			\|\u_t(\cdot,T) \|_{\H}\leq C\big(\|\u_0\|_{\H^2(\Omega) \cap \V}+\|\f\|_{\L^2}+\|g\|_0+\|g_t\|_0\big).
		\end{align}
		Using the final overdetermination data in \eqref{S2}, it can be seen that
		\begin{align*}
			\f g_1+\f_2 g&=\u_t(\cdot,T)+(\boldsymbol{\varphi}_1 \cdot \nabla)\boldsymbol{\varphi}\nonumber+(\boldsymbol{\varphi} \cdot \nabla)\boldsymbol{\varphi}_2-\mu \Delta \boldsymbol{\varphi}+\nabla \psi\\&\quad+\beta \left(|\boldsymbol{\varphi}_1|^{r-1}\boldsymbol{\varphi}_1-|\boldsymbol{\varphi}_2|^{r-1}\boldsymbol{\varphi}_2\right),
		\end{align*}
		which leads to
		\begin{align}\label{S30}
			g_T\|\f\|_{\L^2}&\leq\|\f g_1\|_{\L^2}\nonumber\\&\leq \|\u_t(\cdot,T)\|_{\H}+\|\boldsymbol{\varphi}_1\|_{\widetilde{\L}^\infty}\|\nabla \boldsymbol{\varphi}\|_{\H}+\|\boldsymbol{\varphi}\|_{\widetilde{\L}^4}\|\nabla \boldsymbol{\varphi}_2\|_{\widetilde{\L}^4}\nonumber\\&\quad+\|\nabla \psi-\mu \Delta \boldsymbol{\varphi}\|_{\H}+C\beta(\|\boldsymbol{\varphi}_1\|_{\widetilde{\L}^{\infty}}^{r-1}+\|\boldsymbol{\varphi}_2\|_{\widetilde{\L}^{\infty}}^{r-1})\|\boldsymbol{\varphi}\|_{\H}+\|\f_2\|_{\L^2}\|g\|_0.
		\end{align}
		Using Agmon's and Gagliardo-Nirenberg's inequalities, and Sobolev's embedding theorem, we obtain
		\begin{align}\label{S31}
			g_T\|\f\|_{\L^2}&\leq
			\|\u_t(\cdot,T)\|_{\H}+C\|\boldsymbol{\varphi}_1\|_{\H}^\frac{1}{2}\|\boldsymbol{\varphi}_1\|_{\H^2(\Omega) \cap \V}^\frac{1}{2}\|\nabla \boldsymbol{\varphi}\|_{\H}+\|\nabla \psi-\mu \Delta \boldsymbol{\varphi}\|_{\H}\nonumber\\&\quad+C\|\boldsymbol{\varphi}\|_{\H}^\frac{1}{2}\|\nabla\boldsymbol{\varphi}\|_{\H}^\frac{1}{2}\| \boldsymbol{\varphi}_2\|_{\H}^\frac{1}{4}\| \boldsymbol{\varphi}_2\|_{\H^2(\Omega)\cap \V}^\frac{3}{4}+\|\f_2\|_{\L^2}\|g\|_0\nonumber\\&\quad +C\beta(\|\boldsymbol{\varphi}_1\|_{\H}^{\frac{r-1}{2}}\|\boldsymbol{\varphi}_1\|_{\H^2\cap\V}^{\frac{r-1}{2}}+\|\boldsymbol{\varphi}_2\|_{\H}^{\frac{r-1}{2}}\|\boldsymbol{\varphi}_2\|_{\H^2\cap\V}^{\frac{r-1}{2}})\|\boldsymbol{\varphi}\|_{\H}\nonumber\\&\leq
			\|\u_t(\cdot,T)\|_{\H}+C\left(\| \nabla\boldsymbol{\varphi}\|_{\H}+\|\nabla \psi-\mu \Delta \boldsymbol{\varphi}\|_{\H}+\|g\|_0\right).
		\end{align}
		Substituting \eqref{S29} in \eqref{S31}, one can easily deduce that 
		\begin{align*}
			\|\f\|_{\L^2}\leq C\left(\|\u_0\|_{\H^2(\Omega) \cap \V}+\|g\|_0+\|g_t\|_0+\|\nabla \boldsymbol{\varphi}\|_{\H}+\|\nabla \psi-\mu \Delta \boldsymbol{\varphi}\|_{\H}\right),
		\end{align*}
		which is \eqref{S21}.
		\vskip 0.2 cm
		\noindent\textbf{Case II:} \emph{$d=3$ and $r \geq 3$.} Using H\"older's, Poincar\'e's and Young's inequalities, we find
		\begin{align}
			E_2 \nonumber&\leq \|\u_{1t}(t)\|_{\widetilde{\L}^4} \|\nabla\u(t)\|_{\H}\|\u_t(t)\|_{\widetilde{\L}^4}\leq C \|\u_{1t}(t)\|_{\H}^{\frac{1}{4}}\|\u_{1t}(t)\|_{\V}^{\frac{3}{4}} \|\u(t)\|_{\V}\|\u_t(t)\|_{\H}^{\frac{1}{4}}\|\u_t(t)\|_{\V}^{\frac{3}{4}}\\&    \leq \frac{C}{\lambda_1^{\frac{1}{4}}}\|\u_{1t}(t)\|_{\V}\|\u(t)\|_{\V}\|\u_t(t)\|_{\V}\leq C \|\u_{1t}(t)\|_{\V}^2 \|\u(t)\|_{\V}^2+\frac{\mu}{6}\|\u_t(t)\|_{\V}^2,\label{S32}\\
			E_3 &\leq  \|\nabla\u_2(t)\|_{\H}\|\u_t(t)\|_{\widetilde{\L}^4}^2 \leq C  \|\nabla\u_2(t)\|_{\H}\|\u_t(t)\|_{\H}^\frac{1}{2}\|\u_t(t)\|_{\V}^\frac{3}{2} \nonumber\\&\leq C \|\nabla\u_2(t)\|_{\H}^4\|\u_t(t)\|_{\H}^2+\frac{\mu}{6}\|\u_t(t)\|_{\V}^2.	\label{S33}\\
			E_4 &\leq \|\u(t)\|_{\widetilde{\L}^4} \|\nabla\u_{2t}(t)\|_{\H}\|\u_t(t)\|_{\widetilde{\L}^4}\leq C \|\u_{2t}(t)\|_{\V}^2 \|\u(t)\|_{\V}^2+\frac{\mu}{6}\|\u_t(t)\|_{\V}^2.\label{S34}
		\end{align}
		A calculation similar to \eqref{S27} yields
		\begin{align}\label{S35}
			E_5  \leq C  \|\u_{2t}\|_{\V}^2
			\|\u_{t}\|_{\H}^2+C\left(\|\u_1\|_{\V}^\frac{1}{2}\|\u_1\|_{\H^2(\Omega)\cap \V}^\frac{1}{2}+\|\u_2\|_{\V}^\frac{1}{2}\|\u_2\|_{\H^2(\Omega)\cap \V}^\frac{1}{2}\right)^{2(r-2)}\|\u\|_{\V}^2.
		\end{align}
		Substituting the estimates \eqref{S23} and \eqref{S32}-\eqref{S35} in \eqref{S22}, and then integrating from $0$ to $t$, we obtain
		\begin{align}\label{S36}
			&\|\u_t(t) \|_{\H}^2+\mu \int_0^t \|\u_t (s)\|_{\V}^2\d s+\alpha\int_0^t \|\u_t (s)\|_{\H}^2\d s \nonumber\\&\quad\leq \|\u_t(0) \|_{\H}^2+\frac{T}{\alpha}\big(\|\f\|_{\L^2}\|g_{1t}\|_0+\|\f_2\|_{\H}\|g_t\|_0\big)^2 + C\int_0^t\big(\|\u_{1t}(s)\|_{\V}^2+\|\u_{2t}(s)\|_{\V}^2\big)\|\u(s)\|_{\V}^2\d s\nonumber\\&\qquad+ C\int_0^t\|\nabla\u_2(s)\|_{\H}^4\|\u_t(s)\|_{\H}^2 \d s+C\int_0^t \bigg\{ \|\u_{2t}(s)\|_{\V}^2
			\|\u_{t}(s)\|_{\H}^2\nonumber\\&\qquad+\left(\|\u_1(s)\|_{\V}^\frac{1}{2}\|\u_1(s)\|_{\H^2(\Omega)\cap \V}^\frac{1}{2}+\|\u_2(s)\|_{\V}^\frac{1}{2}\|\u_2(s)\|_{\H^2(\Omega)\cap \V}^\frac{1}{2}\right)^{2(r-2)}\|\u(s)\|_{\V}^2\bigg\}\d s,
		\end{align}
		for all $t \in [0,T]$. An application of Gronwall's inequality \eqref{S36} gives 
		\begin{align*}
			&\sup_{t\in[0,T]}\|\u_t(t) \|_{\H}^2+\mu \int_0^T \|\u_t (t)\|_{\V}^2\d t+\alpha\int_0^T \|\u_t (t)\|_{\H}^2\d t
			\nonumber\\&\leq \exp\bigg(CT\sup_{t\in[0,T]}\|\nabla\u_2(t)\|_{\H}^4+C\int_0^T\|\u_{2t}(t)\|_{\V}^2\d t\bigg)\bigg\{\|\u_t(0) \|_{\H}^2\\&\quad +\frac{T}{\alpha}\left(\|\f\|_{\L^2}\|g_{1t}\|_0+\|\f_2\|_{\H}\|g_t\|_0\right)^2 + C\sup_{t\in[0,T]} \|\u(t)\|_{\V}^2\int_0^T(\|\u_{1t}(t)\|_{\V}^2+\|\u_{2t}(t)\|_{\V}^2)\d t \nonumber\\&\quad+CT\sup_{t\in[0,T]}\|\u(t)\|_{\V}^2 \sup_{t\in[0,T]}\left(\|\u_1(s)\|_{\V}^\frac{1}{2}\|\u_1(s)\|_{\H^2(\Omega)\cap \V}^\frac{1}{2}+\|\u_2(s)\|_{\V}^\frac{1}{2}\|\u_2(s)\|_{\H^2(\Omega)\cap \V}^\frac{1}{2}\right)^{2(r-2)}\bigg\}.
		\end{align*}
		Using the energy estimates given in Lemmas \ref{lemma1}-\ref{lem2.5} and \eqref{S111} in the above inequality, we obtain the following estimate
		\begin{align*}
			\sup_{t\in[0,T]}\|\u_t(t) \|_{\H}^2 \leq C \left( \|\u_0\|_{\H^2(\Omega) \cap \V}^2+\|\f\|_{\L^2}^2+\|g\|_0^2+\|\f\|_{\L^2}+\|g_t\|_0^2 \right),
		\end{align*}
		and thus, we get
		\begin{align}\label{S37}
			\|\u_t(\cdot,T) \|_{\H}\leq C\big(\|\u_0\|_{\H^2(\Omega) \cap \V}+\|\f\|_{\L^2}+\|g\|_0+\|g_t\|_0\big).
		\end{align}
		Hence, using Agmon's and Gagliardo-Nirenberg's inequalities, and Sobolev's embedding theorem in \eqref{S30} gives
		\begin{align}\label{S38}
			g_T\|\f\|_{\L^2}&\leq
			\|\u_t(\cdot,T)\|_{\H}+C\|\boldsymbol{\varphi}_1\|_{\V}^\frac{1}{2}\|\boldsymbol{\varphi}_1\|_{\H^2(\Omega) \cap \V}^\frac{1}{2}\|\nabla \boldsymbol{\varphi}\|_{\H}\nonumber\\&\quad+C\|\boldsymbol{\varphi}\|_{\H}^\frac{1}{4}\|\nabla\boldsymbol{\varphi}\|_{\H}^\frac{3}{4}\| \boldsymbol{\varphi}_2\|_{\H}^\frac{1}{8}\| \boldsymbol{\varphi}_2\|_{\H^2(\Omega)\cap \V}^\frac{7}{8}+\|\nabla \psi-\mu \Delta \boldsymbol{\varphi}\|_{\H}\nonumber\\&\quad +C\beta\left(\|\boldsymbol{\varphi}_1\|_{\V}^{\frac{r-1}{2}}\|\boldsymbol{\varphi}_1\|_{\H^2(\Omega)\cap\V}^{\frac{r-1}{2}}+\|\boldsymbol{\varphi}_2\|_{\V}^{\frac{r-1}{2}}\|\boldsymbol{\varphi}_2\|_{\H^2(\Omega)\cap\V}^{\frac{r-1}{2}}\right)\|\boldsymbol{\varphi}\|_{\H}+\|\f_2\|_{\L^2}\|g\|_0\nonumber\\&\leq
			\|\u_t(\cdot,T)\|_{\H}+C\left(\| \nabla\boldsymbol{\varphi}\|_{\H}+\|\nabla \psi-\mu \Delta \boldsymbol{\varphi}\|_{\H}+\|g\|_0\right).
		\end{align}
		Substituting \eqref{S37} in \eqref{S38}, we finally obtain  \eqref{S21}.
	\end{proof}

	\medskip\noindent
	{\bf Data availability statement:}
	No new data were created or analysed in this study.
	
	\medskip\noindent
	{\bf Acknowledgments:} P. Kumar and M. T. Mohan would  like to thank the Department of Science and Technology (DST), India for Innovation in Science Pursuit for Inspired Research (INSPIRE) Faculty Award (IFA17-MA110). 
	

\begin{thebibliography}{99}
		
		\bibitem{SNA}	S.N. Antontsev and H.B. de Oliveira, The Navier–Stokes problem modified by an absorption term, \emph{Applicable Analysis}, {\bf 89}(12),  2010, 1805--1825. 
		
		\bibitem{VB} V. Barbu, {\it Analysis and control of nonlinear infinite dimensional	systems}, Academic Press, Boston, 1993.
		
		
		\bibitem{VBSS} V. Barbu and S. S. Sritharan, Flow invariance preserving feedback controllers for the Navier-Stokes equation, \emph{J. Math. Anal. Appl.}, {\bf 255}(1) (2001),  281--307. 
		
		\bibitem{MBFC}  M. Badra,  F. Caubet and  J. Dard\'e, Stability estimates for Navier-Stokes equations and application to inverse problems, \emph{Discrete Contin. Dyn. Syst. Ser. B}, {\bf 21}(8) (2016),  2379--2407. 
		
		\bibitem{IB} I. Bushuyev, Global uniqueness for inverse parabolic problems with final observation, \emph{Inverse
			Problems}, {\bf11} (1995), L11–L16.
		
		\bibitem{ZCQJ} 	Z. Cai and Q. Jiu, Weak and Strong solutions for the incompressible Navier-Stokes equations with damping, \emph{Journal of Mathematical Analysis and Applications}, {\bf 343} (2008), 799--809.
		
		\bibitem{LCa}  L. Cattabriga, Su un problema al contorno relativo al sistema di equazioni di Stokes, \emph{Rend. Mat. Sem. Univ. Padova}, {\bf  31} (1961), 308--340.
		
		\bibitem{PGC} 	P. G. Ciarlet, \emph{Linear and Nonlinear	Functional Analysis	with Applications}, SIAM Philadelphia, 2013.
		\bibitem{AYC} A. Y. Chebotarev, Inverse problem for Navier-Stokes systems with finite-dimensional overdetermination, \emph{Differential Equations}, Vol. 48, No. 8, (2012), 1153–1160.
		\bibitem{AYC1} A. Y. Chebotarev,  Inverse problems for stationary Navier-Stokes systems, \emph{Comput. Math. Math. Phys.}, {\bf 54}(3) (2014),  537--545. 
		\bibitem{MC} M. Chouli,  O. Y. Imanuvilov, J.-P. Puel and  M. Yamamoto, Inverse source problem for linearized Navier–Stokes equations with data in arbitrary sub-domain, \emph{Applicable Analysis}, {\bf 92}, (2012), 2127--2143. 
		\bibitem{JF} J. Fan and G. Nakamura, Well-posedness of an inverse problem of Navier–Stokes
		equations with the final overdetermination, \emph{Journal of Inverse and Ill-Posed Problems}, {\bf17} (2009), 565–584.
		\bibitem{JFMD} J. Fan,  M. Di Cristo,  Y. Jiang and G. Nakamura, Inverse viscosity problem for the Navier-Stokes equation, \emph{J. Math. Anal. Appl.}, {\bf  365}(2) (2010),  750--757. 
		\bibitem{JFGN} J. Fan and G. Nakamura, Local solvability of an inverse problem to the density-dependent Navier-Stokes equations, \emph{Appl. Anal.}, {\bf 87}(10-11) (2008),  1255--1265.
		\bibitem{CLF} 	C. L. Fefferman, K. W. Hajduk and J. C. Robinson,	\emph{Simultaneous approximation in Lebesgue and Sobolev norms via eigenspaces}, \url{https://arxiv.org/abs/1904.03337}.
		
		
		
		\bibitem{DFHM} D. Fujiwara, H. Morimoto, An $L^r$-theorem of the Helmholtz decomposition of vector fields, \emph{J. Fac. Sci. Univ. Tokyo Sect. IA Math.,} {\bf 24} (1977),	685--700.
		
		\bibitem{NLG} N. L. Gol’dman, Determination of the right-hand side in a quasilinear parobolic equation with a terminal observation, \emph{Diff. Eqns.}, \textbf{41} (2005), 384–392.
		\bibitem{KWH}	K. W. Hajduk and J. C. Robinson, Energy equality for the 3D critical convective Brinkman-Forchheimer equations, \emph{Journal of Differential Equations}, {\bf 263} (2017), 7141--7161.
		
		\bibitem{KWH1}	K. W. Hajduk, J. C. Robinson and W. Sadowski,	Robustness of regularity for the 3D convective Brinkman-Forchheimer equations, \emph{Journal of Mathematical Analysis and Applications},
		{\bf 500}(1) (2021), 125058
		\bibitem{OYMY}  O.Y. Imanuvilov and M. Yamamoto, Global uniqueness in inverse boundary value problems for the Navier-Stokes equations and Lam\'e system in two dimensions, \emph{Inverse Problems}, {\bf  31}(3) (2015),  035004, 46 pp.
		\bibitem{VI} V. Isakov, \emph{Inverse Problems for Partial Differential Equation}, 2nd ed. Now York, Springer,
		2004.
		\bibitem{VI1} V. Isakov, Inverse parabolic problems with the final overdetermination, \emph{Comm. Pure Appl. Math.},
		{\bf XLIV} (1991), 185–209.
		
		\bibitem{YJJF}Y. Jiang, J. Fan,  S. Nagayasu and  G. Nakamura,  Local solvability of an inverse problem to the Navier-Stokes equation with memory term, \emph{Inverse Problems }, {\bf 36}(6) (2020), 065007, 14 pp.
		
		\bibitem{KT2}  V. K. Kalantarov and S. Zelik, Smooth attractors for the Brinkman-Forchheimer equations with fast growing nonlinearities, \emph{Commun. Pure Appl. Anal.}, {\bf 11}	(2012) 2037--2054.
		
		\bibitem{AIK}  A. I. Korotkii,  Inverse problems of reconstructing parameters of the Navier-Stokes system, \emph{J. Math. Sci.}, {\emph 140} (2007), 808--831.
		
		\bibitem{PKKK} P. Kumar, K. Kinra and M. T. Mohan, A local in time existence and uniqueness result of an inverse problem for the Kelvin-Voigt fluids, \emph{Inverse Problems},  in press \url{https://doi.org/10.1088/1361-6420/ac1050}. 
		
		\bibitem{OAL}	O. A. Ladyzhenskaya, \emph{The Mathematical Theory of Viscous Incompressible Flow}, Gordon and Breach, New York, 1969.
		
		\bibitem{RYL} R.-Y. Lai,  G. Uhlmann and  J.-N. Wang,  Inverse boundary value problem for the Stokes and the Navier-Stokes equations in the plane, \emph{Arch. Ration. Mech. Anal.}, {\bf  215}(3) (2015), 811--829.
		\bibitem{PAM}	P. A. Markowich, E. S. Titi and S. Trabelsi,	Continuous data assimilation for the three-dimensional Brinkman-Forchheimer-extended Darcy model, \emph{Nonlinearity}, {\bf 29}(4) (2016), 1292--1328.
		\bibitem{MTM4} {M. T. Mohan}, On the convective Brinkman-Forchheimer equations (submitted).
		\bibitem{MTM6} {M. T. Mohan}, Stochastic  convective Brinkman-Forchheimer equations (submitted),	\url{https://arxiv.org/abs/2007.09376}. 
		\bibitem{POV} A. I. Prilepko, D. G. Orlovsky and I. A. Vasin, \emph{Methods for Solving Inverse Problems in Mathematical Physics}, Marcel Dekker, New York, 2000.
		\bibitem{PT} A. I. Prilepko and D. S. Tkachenko, Well-posedness of the inverse source problem for parabolic systems, \emph{Diff. Eqns.}, {\textbf{40}} (2004), 1619–1626.
		
		\bibitem{JCR4}	J.C. Robinson,  J.L. Rodrigo,  W. Sadowski, \emph{The three-dimensional Navier–Stokes equations, classical theory}, Cambridge Studies in Advanced Mathematics, Cambridge	University Press, Cambridge, UK, 2016. 
		\bibitem{Te} R. Temam,  \emph{Navier-Stokes Equations, Theory and Numerical Analysis}, North-Holland, Amsterdam, 1984.
		
		
		
		
		
		\bibitem{Te1} R. Temam, 	\emph{Navier-Stokes Equations and Nonlinear Functional Analysis}, Second Edition, CBMS-NSF Regional Conference Series in Applied Mathematics, 1995.
		\bibitem{VP} I. A. Vasin and A. I. Prilepko, The solvability of three dimensional inverse problem for the nonlinear Navier-Stokes Equations, \emph{U.S.S.R.Comput.Matks.Matk.Pkys.}, Vol. 30, No. 5, (1990), 189-199.
		
		\bibitem{ZZXW}	Z. Zhang, X. Wu and M. Lu, On the uniqueness of strong solution to the incompressible Navier-Stokes equations with damping, \emph{Journal of Mathematical Analysis and Applications}, {\bf 377} (2011), 414--419.
		
		
		\bibitem{YZ} 	Y. Zhou, Regularity and uniqueness for the 3D incompressible Navier-Stokes equations with damping, \emph{Applied Mathematics Letters}, {\bf 25} (2012), 1822--1825.
		\bibitem{HZ} H. Zou, et al, \emph{Handbook of differential equations: stationary partial differential equations}, Volume VI, Elsevier.
		
		
		
		
		
	\end{thebibliography}
\end{document}